\documentclass[11pt,reqno]{amsart} 
\usepackage[margin=1.1in]{geometry}
\usepackage{amsmath,amsthm,amssymb,bbm}
\usepackage{xifthen} 
\usepackage{xcolor}
\usepackage{mathrsfs} 
\usepackage{pgf,tikz,pgfplots} 
\usetikzlibrary{arrows}
\usepackage{enumitem}
\usepackage{comment}
\usepackage{chngcntr}
\usepackage{mathabx}
\usepackage{bm}
\usepackage{etoolbox}
\usepackage{enumitem}
\usepackage{floatrow}
\usepackage{float}
\usepackage{pkgfile}
\usepackage{dsfont}

\usepackage[numbers, square, sort]{natbib}

\newtheorem{thm}{Theorem}[section]

\newtheorem{corollary}{Corollary}[section]

\newtheorem{prop}{Proposition}[section]

\newtheorem{lem}{Lemma}[section]

\theoremstyle{definition} 
\newtheorem{defn}{Definition}[section]
\newtheorem{remark}{Remark}[section]

\newcommand{\revsag}[1]{\textcolor{black}{#1}}
\newcommand{\revmark}[1]{\textcolor{black}{#1}}
\allowdisplaybreaks

\begin{document}

\title[Inference in the Hypergraph $\bm \beta$-Model]{Degree Heterogeneity in Higher-Order Networks: Inference in the Hypergraph $\bm \beta$-Model}
\author[S. Nandy and B.B. Bhattacharya]{Sagnik Nandy and Bhaswar B. Bhattacharya}
\address{Department of Statistics and Data Science\\ University of Pennsylvania\\ Philadelphia\\ PA 19104\\ United States}
\email{sagnik@wharton.upenn.edu}
\address{Department of Statistics and Data Science\\ University of Pennsylvania\\ Philadelphia\\ PA 19104\\ United States}
\email{bhaswar@wharton.upenn.edu}

\begin{abstract} 
The $\bm \beta$-model for random graphs is commonly used for representing pairwise interactions in a network with degree heterogeneity. Going beyond pairwise interactions, \citet{stasi2014beta} introduced the hypergraph $\bm \beta$-model for capturing degree heterogeneity in networks with higher-order (multi-way) interactions. In this paper we initiate the rigorous study of the hypergraph $\bm \beta$-model with multiple layers, which allows for hyperedges of different sizes across the layers. To begin with, we derive the rates of convergence of the maximum likelihood (ML) estimate\revsag{s} and establish their minimax rate optimality. We also derive the limiting distribution of the ML estimate\revsag{s} and construct asymptotically valid confidence intervals for the model parameters. Next, we consider the goodness-of-fit problem in the hypergraph $\bm \beta$-model. Specifically, we establish the asymptotic normality of the likelihood ratio (LR) test under the null hypothesis, derive its detection threshold, and also its limiting power at the threshold. Interestingly, the detection threshold of the LR test turns out to be minimax optimal, that is, all tests are asymptotically powerless below this threshold. The theoretical results are further validated in numerical experiments. In addition to developing the theoretical framework for estimation and inference for hypergraph $\bm \beta$-models, the above results fill a number of gaps in the graph $\bm \beta$-model literature, such as the minimax optimality of the ML estimates and the non-null properties of the LR test, which, to the best of our knowledge, have not been studied before. 
\end{abstract}

\maketitle

\section{Introduction} 

The $\bm \beta$-model is an exponential family distribution on graphs with the degree sequence as the sufficient statistic. Specifically, in the $\bm \beta$-model with vertex set $[n] := \{1, 2, \ldots, n \}$, the edge $(i, j)$ is present independently with probability 
\begin{align}\label{eq:probabilitygraph}
p_{ij} := \frac{e^{\beta_i + \beta_j}}{1+ e^{\beta_i + \beta_j}}, 
\end{align} 
for $1 \leq i < j \leq n$ and $\bm \beta = (\beta_1, \beta_2, \ldots, \beta_n) \in \R^n$. This model was first considered by \citet{park2004statistical} and 
can also be viewed as the undirected version of the $p_1$-model that appear\revsag{s} in the earlier work of \citet{exponential1981graphs}.  
Thereafter, the $\bm \beta$-model has been widely used for capturing degree heterogeneity in networks (see \citet{degreesequence,chen2021analysis,graham2017degreeheterogeneity,jackson2008social}, among several others). The term $\bm \beta$-model can be attributed to the seminal paper of \citet{chatterjee2011random}, which provides the theoretical foundations for parameter estimation in this model.

While random graph models, such as the $\bm \beta$-model, are important tools for understanding binary (pairwise) relational data, in many modern applications interactions occur not just between pairs, but among groups of agents. 
Examples include folksonomy \cite{ghoshal2009random}, collaboration networks \citep{networkcoauthorship,network2016discussion,patania2017shape}, complex ecosystems \citep{grilli2017higher}, biological networks \cite{michoel2012alignment,petri2014homological}, 
circuit design \cite{karypis1999multilevel}, computer vision \cite{agarwal2005beyond}, among others. 
Hypergraphs provide the natural mathematical framework for modeling such higher-order interactions \cite{battiston2020networks,battiston2021complexsystems,network2016higher}. Formally, a hypergraph $H$ is denoted by $H = (V(H), E(H))$, where $V(H)$ is the vertex set (the agents in the network) and $E(H)$ is a collection of non-empty subsets of $V(H)$ \revsag{of cardinality greater than $1$}. The elements in $E(H)$,  referred to as {\it hyperedges}, represent the interactions among groups of agents. 
Motivated by the emergence of complex relational data with higher-order structures,  there has been a slew of recent results on modeling random hypergraphs, community detection, recovery, clustering, and motif analysis, among others (see \cite{ghoshdastidar2017uniform,ghoshdastidar2017consistency,ghoshdastidar2015provable,ghoshdastidar2015spectral,pal2021community,gracious2023neural,yadati2019hypergcn,hypergraphgeometric,higherordernetwork,hypergraph2022higher,nonparametric2021,ke2019community,hu2022multiway,yuan2022testing,chien2019minimax,ahn2019community,zhang2022exact,young2021hypergraph} and the references therein).

In this paper we study the hypergraph $\bm \beta$-model, introduced by \citet{stasi2014beta}, that allows one to incorporate degree heterogeneity \revsag{in networks with higher-order structures}. Like the graph $\bm \beta$-model \eqref{eq:probabilitygraph}, this is an exponential family on hypergraphs where the (hypergraph) degrees \revsag{(as defined in \eqref{eq:dsv})} are the sufficient statistics. In its general form it allows for layered hypergraphs with hyperedges of different sizes across the layers. To describe the model formally we need a few notations: For $r \geq 2$, denote by ${[n] \choose r}$ the collection of all $r$-element subsets of $[n]: =\{1, 2, \ldots, n\}$.  A hypergraph $H = (V(H), E(H))$ is said to be $r$-uniform if every element in $E(H)$ has cardinality $r$. (Clearly, 2-uniform hypergraphs are simple graphs.) We will denote by $\cH_{n, r}$ the collection of all $r$-uniform hypergraphs with vertex set $[n]$ and \revsag{by} $\cH_{n, [r]} := \bigcup_{s=2}^{r} \cH_{n, s}$, the collection of all hypergraphs with vertex set $[n]$ where every hyperedge has size at most $r$. Then the $r$-{\it layered hypergraph $\bm \beta$-model} is a probability distribution on $\cH_{n, [r]}$ defined as follows:

\begin{defn}\cite{stasi2014beta}
\label{def:layered_beta_model} 
Fix $r \geq 2$ and parameters $\bm B := (\bm \beta_{2}, \ldots,\bm \beta_{r})$, where $\bm \beta_{s}:=(\beta_{s, v})_{v \in [n]} \in \R^n$. The $r$-{\it layered hypergraph $\bm \beta$-model} is a random hypergraph in $\cH_{n, [r]}$, denoted by $\mathsf{H}_{[r]}(n, \bm B) $, where, for every $2 \leq s \leq r$, the hyperedge $\{v_{1}, v_{2}, \ldots, v_{s}\} \in {[n] \choose s}$ is present independently with probability: 
\begin{align}\label{eq:Hlayeredbeta}
p_{v_{1}, v_{2}, \ldots, v_{s}}:=\frac{ e^{\beta_{s, v_1}+\ldots+\beta_{s, v_s}}}{1 + e^{\beta_{s, v_1}+\ldots+\beta_{s, v_s}} }. 
\end{align}
\end{defn}



This model can be expressed as an exponential family on $\cH_{n, [r]}$ with the hypergraph degrees up to order $r$ as the sufficient statistics (see \eqref{eq:gl}). Specifically, the parameter $\beta_{s, u}$ encodes the popularity of the node $u \in [n]$ to form groups of size $s$, for $2 \leq s \leq r$. Consequently, $\beta_{s, u}$ controls the local density of hyperedges of size $s$ around the around node $u$. The model \eqref{eq:Hlayeredbeta} includes as a special case the classical graph $\bm \beta$-model (when $r=2$) and also the $r$-uniform hypergraph $\bm \beta$-model, where only the hyperedges of size $r$ are present. More formally, given parameters $\bm \beta= (\beta_1, \beta_2, \ldots, \beta_n) \in \R^n$, the {\it $r$-uniform hypergraph $\bm \beta$-model} is a random hypergraph in $\cH_{n, r}$, denoted by $\mathsf{H}_r(n, \bm \beta) $, where each $r$-element hyperedge $\{v_{1}, v_{2}, \ldots, v_{r}\} \in {[n] \choose r}$ is present independently with probability: 
\begin{align}\label{eq:Huniformbeta}
p_{v_{1}, v_{2}, \ldots, v_{r}}:=\frac{ e^{\beta_{v_1}+\ldots+\beta_{v_r}}}{1 + e^{\beta_{v_1}+\ldots+\beta_{v_r}} }. 
\end{align}

It is worth noting that, since the degrees \revsag{(defined later in \eqref{eq:dsv})} are the sufficient statistics in the aforementioned models, it is enough to observe only the degree sequences (not the entire network) for inference regarding the model parameters. This feature makes the $\bm \beta$-model particularly attractive because collecting information about the entire network can often be difficult for cost or privacy reasons. For example, \citet{elmer2020students} (see also \citet{zhang2023l2}) studied social networks between a group of Swiss students before and during COVID-19 lockdown, where, for privacy reasons, only the total number of connections of each student in the network (that is, the degrees of the vertices) were released. The $\bm \beta$-model is also relevant in the analysis of aggregated relational data, where instead of asking about connections between groups of individuals directly, one collects data on the number of connections of an individual with a given feature (see, for example,  \citet{aggregated2023consistently} and the references therein).

\citet{stasi2014beta} proposed two algorithms for computing the maximum likelihood (ML) estimates for the hypergraph $\bm \beta$ models described above and reported their empirical performance. However, the statistical properties of the ML estimates in these models have remained unexplored. 

\subsection{Summary of Results} 

In this paper we develop a framework for estimation and inference in the hypergraph $\bm \beta$-model. Along the way, we obtain a number of new results on the graph \revsag{$\bm \beta$-model} as well. The following is a summary of the results: 

\begin{itemize} 

\item {\it Estimation}: In Section \ref{sec:mle} we derive the rates of convergence of the ML estimates in $r$-layered hypergraph $\bm \beta$-model \eqref{eq:Hlayeredbeta},  both in the $L_\infty$ and the $L_2$ norms.\footnote{We denote by $\| \bm x\|_\infty$ and $\| \bm x \|_2$, the $L_\infty$ and the $L_2$ norms of a vector $\bm x$, respectively. } Specifically, we show that given a sample $H_n \sim \mathsf{H}_{[r]}(n, \bm B) $ from the $r$-layered hypergraph $\bm \beta$-model, the ML estimate $\hat{\bm B} = (\hat{\bm \beta}_{2},\ldots, \hat{\bm \beta}_{r})$ of $\bm B$ satisfies: 
\begin{align}\label{eq:mledescription} \revmark{
\|\hat{\bm \beta}_{s}-\bm\beta_{s}\|_2 \le \,C_s\, \sqrt{\frac{1}{\;n^{s-2}}} \quad \text{ and } \quad 
\|\hat{\bm \beta}_{s}-\bm\beta_{s}\|_\infty \le \,C_s\, \sqrt{\frac{\log n}{\;n^{s-1}}} , }
\end{align}
for some \revmark{constant $C_s>0$ (depending on $s$ and $\| \bm \beta_s \|_\infty$), where $2 \leq s \leq r$} , with probability going to 1 (see Theorem \ref{thm:layermle}). These extend the results of \citet{chatterjee2011random} on the graph $\bm \beta$-model, where the rate of convergence of the ML estimate was derived only in the $L_\infty$ norm, to the hypergraph case. Next, in Theorem \ref{thm:maxl2lb} we show that both the rates in \eqref{eq:mledescription} are, in fact, minimax optimal (up to a $\sqrt{\log n}$ factor for the $L_\infty$ norm). To the best of our knowledge, these are the first results showing the statistical optimality of the ML estimates in the $\bm \beta$-model even for the graph case. 

\item {\it Inference}: In Section \ref{sec:mleconfidenceinterval} we derive the asymptotic distribution of the ML estimate $\hat{\bm B}$. In particular, we prove that the finite dimensional distributions of the ML estimate converges to a multivariate Gaussian distribution (see Theorem \ref{thm:central_lim_thm}). Moreover, the covariance matrix of the Gaussian distribution can be estimated consistently, using which we can construct asymptotically valid confidence sets for the model parameters (see Theorem \ref{thm:conf_int_thm}).  

\item {\it Testing}: In Section \ref{sec:goodnessoffit} we study the goodness-of-fit problem for the hypergraph $\bm \beta$-model, that is, given $\bm \gamma \in \R^n$ we wish to distinguish: 
\begin{align}\label{eq:H0description}
H_0:\bm{\beta}_s=\bm \gamma \quad \mbox{versus} \quad H_1:\bm{\beta}_s \neq \bm \gamma.
\end{align} 
We show that the likelihood ratio (LR) statistic for this problem (centered and scaled appropriately) is asymptotically normal under $H_0$ (see Theorem \ref{thm:testing_inhom_level} for details). Using this result we construct an asymptotically level $\alpha$ test for \eqref{eq:H0description}. Next, we study the power properties of this test. In particular, we show that the detection threshold for the LR test in the $L_2$ norm is $n^{-\frac{2s-3}{4}}$, that is, the LR test is asymptotically powerful/powerless in detecting $\bm \gamma' \in \R^n$ depending on whether $\|\bm \gamma'-\bm \gamma\|_2$ is asymptotically greater/smaller than $n^{-\frac{2s-3}{4}}$, respectively. We also derive the limiting power function of the LR test at the threshold $\|\bm \gamma'-\bm \gamma\|_2 = \Theta(n^{-\frac{2s-3}{4}})$ (see Theorem \ref{thm:testing_inhom_power}). Further, in Theorem \ref{thm:H0l2} we show that this detection threshold is, in fact, minimax optimal, that is all tests are asymptotically powerless when $\|\bm \gamma'-\bm \gamma\|_2$ is asymptotically smaller than $n^{-\frac{2s-3}{4}}$. In Section \ref{sec:testingmax} we also obtain the detection threshold of the LR test in the $L_\infty$ norm and establish its optimality. Again, these appear to be the first results on the non-null properties of the LR test and its optimality in the $\bm \beta$-model for the graph case itself.  

\end{itemize} 
In Section \ref{sec:simulations} we illustrate the finite-sample performances of the proposed methods in simulations.

\subsection{Related Work on the Graph $\bm \beta$-Model}

As mentioned before, \citet{chatterjee2011random} initiated the rigorous study of estimation in the graph $\bm \beta$-model. They derived, among others things, the convergence rate of the ML estimate in the $L_\infty$ norm. Thereafter, \citet{rinaldo_mle_beta_model} derived necessary and sufficient conditions for the existence of the ML estimate in terms of the polytope of the degree sequences. 
\citet{yan_xu} proved the asymptotic normality of ML estimate and later, \citet{yan2016asymptotics} derived the properties of a moment based estimator.  
\citet{karwa2016inference} studied the problem of estimation in the $\bm \beta$-model under privacy constraints. 

In the context of hypothesis testing, \citet{degreetesting} considered the problem of sparse signal detection in the $\bm \beta$-model, that is, testing whether all the node parameters are zero versus whether a (possibly) sparse subset of them are non-zero. Recently, \citet{yan2022wilks} derived the asymptotic properties of the LR test for the goodness-of-fit problem in the graph $\bm \beta$-model, under the null hypothesis. 

The graph $\bm \beta$-model has also been generalized to incorporate additional information, such as covariates, directionality, sparsity, and weights (see \citet{chen2021analysis,chen2013directed,graham2017degreeheterogeneity,stein2020sparse,directed,hillar2013maximum,yan2019statistical,wahlstrom2017beta,zhang2023l2} and the references therein). For other exponential random graph models with functions of the degrees as sufficient statistics, see \citet{mukherjee2020degeneracy} and \citet{xu2021signal}.


\subsection{Asymptotic Notation} For positive sequences $\{a_n\}_{n\geq 1}$ and $\{b_n\}_{n\geq 1}$, $a_n = O(b_n)$ means $a_n \leq C_1 b_n$ and $a_n =\Theta(b_n)$ (and equivalently, $a_n \asymp b_n$) means $C_2 b_n \leq a_n \leq C_1 b_n$, for all $n$ large enough and positive constants $C_1, C_2$. Similarly, for positive sequences $\{a_n\}_{n\geq 1}$ and $\{b_n\}_{n\geq 1}$, $a_n \lesssim b_n$ means $a_n \leq C_1 b_n$ and $a_n \gtrsim b_n$ means $a_n \geq C_2 b_n$ for all $n$ large enough and positive constants $C_1, C_2$. Moreover, subscripts in the above notation,  for example $O_\square$, $\lesssim_\square$, $\gtrsim_\square$, and $\Theta_\square$,  denote that the hidden constants may depend on the subscripted parameters. Also, for positive sequences $\{a_n\}_{n\geq 1}$ and $\{b_n\}_{n\geq 1}$, $a_n \ll b_n$ means $a_n/b_n \rightarrow 0$ and $a_n \gg b_n$ means $a_n/b_n \rightarrow \infty$, as $n \rightarrow \infty$.

\section{Maximum Likelihood Estimation in Hypergraph $\bm \beta$-Models}  
\label{sec:mle}

In this section we consider the problem of parameter estimation in the hypergraph $\bm \beta$-model using the ML method. In Section \ref{sec:mleconvergence} we derived the rates of the consistency of the ML estimate. The central limit theorem  of the ML estimate and the construction of confidence intervals for the model parameters are presented in Section \ref{sec:mleconfidenceinterval}.

\subsection{Rates of Convergence} 
\label{sec:mleconvergence}

Given a sample $H_n \sim \mathsf{H}_{n, [r]}(n, \bm B)$ from the $r$-{\it layered hypergraph $\bm \beta$-model}, the likelihood function can be written as follows: 
\begin{align}
\revmark{
L_{n}(\bm B) = \prod_{2 \leq s \leq r} \prod_{\{v_{1}, v_{2}, \ldots, v_{s}\} \in {[n] \choose s}}  \frac{ e^{(\beta_{s, v_1}+\ldots+\beta_{s, v_s})\mathbbm{1}\{\{v_{1}, v_{2}, \ldots, v_{s}\} \in E(H_n)\}} }{1 + e^{\beta_{s, v_1}+\ldots+\beta_{s, v_s}} }. }
\end{align} 
Therefore, the negative log-likelihood is given by 
\begin{align}
\label{eq:gl}
\ell_{n}(\bm B) & := - \log L_{n}(\bm B) \nonumber \\ 
& = - \sum_{s=2}^{r} \left\{\sum_{v=1}^{n} \beta_{s, v} d_s(v)-\sum_{\{v_{1}, v_{2}, \ldots, v_{s}\} \in {[n] \choose s}}\log{\left(1+\exp\left(\beta_{s, v_1}+\ldots+\beta_{s, v_s}\right)\right)}\right\} , 
\end{align}
where 
\begin{align}\label{eq:dsv}\revsag{
d_s(v) := \sum_{\bm e \in E(H_n): | \bm e | = s }\mathbbm 1\{v  \in \bm e\}, }
\end{align} 
is the $s$-degree of the vertex $v \in [n]$, that is, the number of hyperedges of size $s$ in $H_n$ passing through $v$. 
The negative log-likelihood in \eqref{eq:gl} can be re-written as: 
\begin{equation}
\label{eq:lik_tot}
\ell_{n}(\bm B)= \sum_{s=2}^{r} \ell_{n, s} (\bm \beta_{s}) , 
\end{equation}  
where 
\begin{equation}
\label{eq:lik_k}
\ell_{n, s} (\bm \beta) := \sum_{ \{v_{1}, v_{2}, \ldots, v_{s}\} \in {[n] \choose s} }\log{\left(1+\exp\left(\beta_{s, v_1}+\ldots+\beta_{s, v_s}\right)\right)} - \sum_{v=1}^{n}\beta_{s, v} d_s(v) .
\end{equation}
Note that \eqref{eq:lik_tot} is separable in $\bm \beta_{2},\ldots,\bm \beta_{r}$, hence, the ML estimate of $\bm B = (\bm \beta_{2}, \ldots,\bm \beta_{r})$ is given by $\hat{\bm B}=(\hat{\bm \beta}_{2},\ldots,\hat{\bm \beta}_{r})$, where  
\begin{equation}
\label{eq:ml}
\hat{\bm \beta}_{s} := \argmin_{\bm \beta}\ell_{n, s}(\bm \beta).
\end{equation}  
\revmark{Therefore, differentiating \eqref{eq:gl} with respect to $\bm \beta_s$ and setting the gradient to zero we can conclude } that the ML estimate $\hat{\bm \beta}_{s}$ satisfies the following set of gradient equations: For all $v \in [n]$ and $2 \leq s \leq r$, 
\begin{equation}
\label{eq:def_k_unif}
d_s(v)=\sum_{ \{v_2, \ldots, v_s\} \in {{ [n] \backslash\{v\} } \choose { s-1 }} }\frac{ e^{\hat \beta_{s, v}+\hat \beta_{s, v_2}+\ldots+\hat \beta_{s, v_s}} }{1+ e^{\hat \beta_{s, v}+\hat \beta_{s, v_2}+\ldots+\hat \beta_{s, v_s}}},
\end{equation}
where ${{ [n] \backslash\{v\} } \choose { s-1 }}$ denotes the collection of all $(s-1)$-element subsets of $ [n] \backslash\{v\}$. \citet{stasi2014beta} presented two algorithms for computing the ML estimate, \revsag{namely,} an iterative proportional scaling algorithm and a fixed point algorithm, and showed that both algorithms converge if the ML estimate exists. 

In this paper we study the asymptotic properties of the ML estimates. In the following theorem we show that the likelihood equations \eqref{eq:def_k_unif} have a unique solution with high-probability and derive its rate of convergence. Recall we denote by $\| \bm x\|_\infty$ and $\| \bm x \|_2$, the $L_\infty$ and the $L_2$ norms of a vector $\bm x$, respectively. Also, denote $\cB_M = \{ \bm x : \| \bm x \|_\infty \leq M \}$, the $L_\infty$ the ball of radius $M$. Throughout we will assume $\bm \beta_s \in \cB_M$, for all $2 \leq s \leq r$, for some constant $M > 0$. 

\begin{thm}
\label{thm:layermle} 
Suppose $H_n \sim \mathsf{H}_{n, [r]}(n, \bm B)$ is a sample from the $r$-layered hypergraph $\bm \beta$-model as defined in \eqref{eq:Hlayeredbeta}.  Then with probability $1-o(1)$ the likelihood equations \eqref{eq:def_k_unif} have a unique solution $\hat{\bm B} = ( \hat{\bm \beta}_{2}, \ldots, \hat{\bm \beta}_{r})$, that satisfies: 
\begin{align}\label{eq:mlerlayer} 
\|\hat{\bm \beta}_{s}-\bm\beta_{s}\|_2 \lesssim_{s, M} \, \sqrt{\frac{1}{\;n^{s-2}}} \quad \text{ and } \quad 
\|\hat{\bm \beta}_{s}-\bm\beta_{s}\|_\infty \lesssim_{s, M} \, \sqrt{\frac{\log n}{\;n^{s-1}}} , 
\end{align}
for $2 \leq s \leq r$. 
\end{thm}

Theorem \ref{thm:layermle} provides the rates for the \revsag{ML estimates for the parameters in a $r$-layered hypergraph $\bm \beta$-model both in the  $L_\infty$ and $L_2$ norms}.  To interpret the rates in \eqref{eq:mlerlayer} note that $s$-degree of a vertex (recall \eqref{eq:dsv}) in the $r$-layered model $\mathsf{H}_{n, [r]}(n, \bm B)$ is $O(n^{s-1})$ with high probability. This means there are essentially $O(n^{s-1})$ independent hyperedges containing information about each parameter in the $s$-th layer. Hence, each parameter in the $s$-th layer can be estimated at the rate $1/\sqrt{n^{s-1}}$. Aggregating this over the $n$ coordinates gives the rates in  \eqref{eq:mlerlayer} for the vector of parameters $\bm \beta_s$ in the $s$-th layer. 

The proof of Theorem \ref{thm:layermle} is given in Appendix \ref{mle_rate}. The following discussion provides a high-level outline of the proof. 

\begin{itemize}

\item For the rate in the $L_2$ norm we first upper bound the gradient of the log-likelihood at the true parameter value. Specifically, we show that $\|\nabla {\ell}_{n, s}(\bm{\beta}_s)\|_2^2 = O(n^{s})$ with high probability (see Lemma \ref{lem:gradientH} for details). Next, we show that the smallest eigenvalue of the Hessian matrix $\nabla^2 {\ell}_{n, s}(\bm{\beta}_s)$ is bounded below by $n^{s-1}$ (up to constants) in a neighborhood of the true parameter (see Lemma \ref{lem:lower_bound_hess}). Then a Taylor expansion of the log-likelihood around the true parameter, combined with the above estimates, imply the rate in the $L_2$ norm in \eqref{eq:mlerlayer} (see Appendix \ref{sec:layermleL2pf} for details).

\item The proof of the rate in the $L_\infty$ norm is more involved. For the graph case,   \cite{chatterjee2011random} analyzed the fixed point algorithm for solving the ML equations and developed a stability version of the Erd\H os-Gallai condition (which provides a necessary and sufficient condition for a sequence of numbers to be the degree sequence of a graph) to derive the rate of ML estimate in the $L_\infty$ norm. One of the technical challenges in dealing with the hypergraph case is the absence of Erd\H os-Gallai-type characterizations of the degree sequence. 
To circumvent this issue, we take a more analytic approach based on the `leave-one-out' technique, that appear in the analysis of ranking models \cite{chen2022partial,Chen2019SpectralMA}.  
Here the idea is to decompose, for each $u \in [n]$, the log-likelihood function of the $s$-th layer $\ell_{n, s}$  (recall \eqref{eq:lik_k}) into two parts: one depending on $\beta_{s, u}$ and the other not depending on it. 
Using the part of the log-likelihood not depending on  $\beta_{s, u}$ we first analyze the properties of the constrained leave-one-out ML estimate, which is the maximizer of the part of the log-likelihood not depending on  $\beta_{s, u}$ in a neighborhood of the leave-one-out true parameter. Then from the part of the log-likelihood depending on $\beta_{s, u}$ we obtain, by a Taylor expansion around the true parameter value $\beta_{s, u}$, the $L_\infty$ rate in \eqref{eq:mlerlayer} with an extra additive error term which depends on the constrained leave-one-out ML estimate. Using the bound on the latter obtained earlier  we show this error term is negligible compared to the $L_\infty$ rate in \eqref{eq:mlerlayer}. 
\end{itemize} 

The following corollary about the $r$-uniform model is an immediate consequence of Theorem \ref{thm:layermle}. We record it separately for ease of referencing.

\begin{corollary}
\label{cor:uniformmle} 
Suppose $H_n \sim \mathsf{H}_{n, r}(n, \bm \beta)$ is a sample from the $r$-uniform hypergraph $\bm \beta$-model as defined in \eqref{eq:Huniformbeta}. Then with probability $1-o(1)$ the ML estimate $\hat{\bm \beta}$ is unique and 
\begin{align}\label{eq:uniformmle}  
\|\hat{\bm \beta}-\bm\beta \|_2 \lesssim_{r,M}  \sqrt{\frac{1}{\;n^{r-2}}} \quad \text{and} \quad 
\|\hat{\bm \beta}-\bm\beta \|_\infty \lesssim_{r,M}  \sqrt{\frac{\log n}{\;n^{r-1}}} . 
\end{align} 
\end{corollary}

\subsection{Minimax Rates}

In the following theorem we establish the tightness of the rates of ML estimate obtained in the previous section by proving matching lower bounds.

\begin{thm}\label{thm:maxl2lb}
Suppose $H_n \sim \mathsf{H}_{n, [r]}(n, \bm B)$, with $\bm B = (\bm \beta_2, \ldots, \bm \beta_r)$, such that \revsag{$\bm \beta_s \in \cB_M$}, for $2 \leq s \leq r$. Given $\delta \in (0, 1)$ there exists a constant $C$ (depending on $M$, $r$, and $\delta$) such that the following holds for estimation in the $L_2$ norm: 
\begin{align}\label{eq:l2lb}
\min_{\hat{\bm \beta}}\max_{\revsag{\bm \beta_s \in \cB_M}}\mathbb{P}\left(\|\hat{\bm \beta}-\bm\beta_s \|_2 \geq C \sqrt{\frac{1}{n^{s-2}}}\right) \ge 1 - \delta .
\end{align} 
Moreover, for estimation in the $L_\infty$ norm the following holds: \revsag{there exists $n_0 \ge 0$ such that for all $n \ge n_0$}
\begin{align}\label{eq:maxlb}  
\min_{\hat{\bm \beta}}\max_{\revsag{\bm \beta_s \in \cB_M}}\mathbb{P}\left(\|\hat{\bm \beta}-\bm\beta_s \|_\infty \geq C \revsag{\sqrt{\frac{\log n}{n^{s-1}}}}\right) \ge 1 - \delta . 
\end{align} 
\end{thm}

\revsag{This result shows that the ML estimate is minimax rate optimal in both the $L_2$ and the $L_\infty$ metrics. The proof of Theorem \ref{thm:maxl2lb} is given in Appendix \ref{sec:estimationlbpf}. The bounds in  \eqref{eq:l2lb} and \eqref{eq:maxlb} are proved using Fano's lemma. For the bound in \eqref{eq:l2lb} we construct $2^{\Theta(n)}$ well-separated points in the parameter space which have `small' average Kulbeck-Leibler (KL) divergence with the origin (see Appendix \ref{sec:l2lbpf}) and for the bound in \eqref{eq:maxlb} we construct $n$ well separated points which have `small' KL-divergence with the origin (see Appendix \ref{sec:maxlbpf})}


\subsection{Central Limit Theorems and Confidence Intervals} 
\label{sec:mleconfidenceinterval}

The results obtained in the previous section show that the vector ML estimates are consistent in the $L_\infty$-norm. However, for conducting asymptotically precise inference on the individual model parameters, we need to understand the limiting distribution of the ML estimates. In Theorem \ref{thm:central_lim_thm} below we show that the finite dimensional distributions of the ML estimates (appropriately scaled) converge to a multivariate Gaussian distribution. Using this result in Theorem \ref{thm:conf_int_thm} we construct joint confidence sets for any finite collection of parameters. Towards this, for $H_n \sim \mathsf{H}_{n, [r]}(n, \bm B)$ denote the variance of the $s$-degree of the node $v \in [n]$ as:                                       
\begin{align}\label{eq:sigmadegree}
\sigma_{s}(v)^2 := \Var[d_s(v)] = \sum_{ \{v_2, \ldots, v_s\} \in {{ [n] \backslash\{v\} } \choose { s-1 }} }\frac{ e^{\beta_{s, v}+\beta_{s, v_2}+\ldots+\beta_{s, v_s}} }{(1+ e^{\beta_{s, v}+\beta_{s, v_2}+\ldots+\beta_{s, v_s}})^2} . 
\end{align}
Then we have the following result: 
\begin{thm}
\label{thm:central_lim_thm} 
Suppose $H_n \sim \mathsf{H}_{n, [r]}(n, \bm B)$ is a sample from the $r$-layered hypergraph $\bm \beta$-model as defined in \eqref{eq:Hlayeredbeta}. 
\revmark{ Fix any collection of integers $a_1, a_2, \ldots, a_r \geq 1$ and any subsets of nodes $J_1, J_2, \ldots, J_r$ with cardinalities $a_1, a_2, \ldots, a_r$, respectively. } 
Then as $n \rightarrow \infty$, 
\begin{align}\label{eq:central_lim}
\begin{pmatrix} 
 [ \bm D_{2} (\hat{\bm \beta}_2 - \bm \beta_{2})]_{J_2} \\ 
 [ \bm D_{3} (\hat{\bm \beta}_3 - \bm \beta_{3})]_{J_3}  \\
\vdots \\ 
 [ \bm D_{r} (\hat{\bm \beta}_r - \bm \beta_{r})]_{J_r}
\end{pmatrix} 
\overset{D}{\rightarrow} \cN_{\sum_{s=2}^r a_s}(\bm 0, \bm I ) , 
\end{align}
where $\bm D_s = \mathrm{diag}\,(\sigma_s(v))_{v \in [n]}$, 
for $2 \leq s \leq r$ and for any vector $\bm x \in \R^n$, $[\bm x]_{J_s} = (x_{v})_{v \in [J_s]}^\top$. 
\end{thm}

\revsag{
\begin{remark}
Observe that since $\bm \beta_s \in \cB_M$ for all $2 \le s \le r$, $O(\sigma_{s}(v)^2) \asymp n^{s-1}$ for all $v \in [n]$. Therefore, the convergence rate of $ [ (\hat{\bm \beta}_s - \bm \beta_{s})]_{J_s} $ is of order $n^{(s-1)/2}$ for all $2 \le s \le r$.
\end{remark}
}

The proof of Theorem \ref{thm:central_lim_thm} is given in Appendix \ref{sec:distributionpf}. The idea of the proof is to linearize $\hat \beta_{s, v} - \beta_{s, v}$ in terms of the $s$-degree of the node $v \in [n]$. Since the $s$-degree of a node is the sum of independent random variables, applying Lindeberg's CLT gives the result in \eqref{eq:central_lim}.   
In the special case of the $r$-uniform model, Theorem \ref{thm:central_lim_thm} can be written in the following simpler form: 

\begin{corollary}
\label{cor:central_lim_thm} 
Suppose $H_n \sim \mathsf{H}_{n, r}(n, \bm \beta)$ is a sample from the $r$-uniform hypergraph $\bm \beta$-model as defined in \eqref{eq:Huniformbeta}. For all $v \in [n]$, let 
\[
\revsag{\sigma(v)^2 :=  \sum_{\{v_2, \ldots, v_r\} \in {{ [n] \backslash\{v\} } \choose { r-1 }} }\frac{ e^{\beta_{v}+\beta_{v_2}+\ldots+\beta_{v_r}} }{1+ e^{\beta_{v}+\beta_{v_2}+\ldots+\beta_{v_r}}}} . 
\] 
Then for any collection of $a \geq 1$ indices $J:=\{v_{1}, \cdots, v_{a}\} \in {[n] \choose a}$, 
as $n \rightarrow \infty$, 
\[
 [ \bm D]_{J}([\hat{\bm \beta}]_{J}-[\bm \beta]_{J}) 
\dto  \cN_{a}(\bm 0, \bm I ) , 
\]
where $\bm D = \mathrm{diag}\,(\sigma(v))_{v \in [n]}$, $[\bm D]_{J} = \mathrm{diag} \,(\sigma(v))_{v \in J}$, \revsag{$[\hat{\bm \beta}]_{J} = (\hat\beta_{v})_{v \in [J]}^\top$, and $[\bm \beta]_{J} = (\beta_{v})_{v \in [J]}^\top$}.  
\end{corollary}

To use the above results to construct confidence sets for the parameters, we need to consistently estimate the elements of the matrix $\bm D_s$. Note that the natural plug-in estimate of  $\sigma_{s}(v)$ is 
\begin{align}\label{eq:sigmadegreeestimate}
\hat \sigma_{s}(v)^2 := \sum_{\{v_2, \ldots, v_s\} \in {{ [n] \backslash\{v\} } \choose { s-1 }} }\frac{ e^{\hat \beta_{s, v}+ \hat \beta_{s, v_2}+\ldots+ \hat \beta_{s, v_s}} }{(1+ e^{\hat \beta_{s, v}+ \hat \beta_{s, v_2}+\ldots+ \hat \beta_{s, v_s}})^2 } . 
\end{align} 
This estimate turns out to be consistent for $\sigma_{s}(v)$, leading to the following result (see Appendix \ref{sec:distributionestimationpf} for the proof):

\begin{thm}
\label{thm:conf_int_thm} 
Suppose $H_n \sim \mathsf{H}_{n, [r]}(n, \bm B)$ is a sample from the $r$-layered hypergraph $\bm \beta$-model as defined in \eqref{eq:Hlayeredbeta}. \revmark{ Fix any collection of integers $a_1, a_2, \ldots, a_r \geq 1$ and any subsets of nodes $J_1, J_2, \ldots, J_r$ with cardinalities $a_1, a_2, \ldots, a_r$, respectively.} 
\revsag{Then, for all $\alpha \in (0,1)$, }
\begin{align}\label{eq:confidenceinterval}
\lim_{n \rightarrow \infty} \mathbb P \left(
\left\{
\sum_{s=2}^r ([(\hat{\bm \beta}_s - \bm \beta_{s})]_{J_s})^\top [ \hat{\bm D}_{s}^2]_{J_s} ([(\hat{\bm \beta}_s - \bm \beta_{s})]_{J_s}) \leq \chi^{2}_{\sum_{s=2}^r a_s, 1 - \alpha} \right\} \right) = 1-\alpha , 
\end{align}
where $\hat{\bm D}_s^2 = \mathrm{diag}(\hat \sigma_s(v)^2)_{v \in [n]}$, $[ \hat{\bm D}_s^2]_{J_s} = \mathrm{diag} ( \hat \sigma_s(v)^2)_{v \in J_s}$, for $2 \leq s \leq r$, and for $a \geq 1$, $\chi^{2}_{a, 1 - \alpha}$ is the $(1-\alpha)$-th quantile of the chi-squared distribution with $a$ degrees of freedom. 
\end{thm}

\section{Goodness-of-Fit: Asymptotics of the Likelihood Ratio Test and Minimax Detection Rates}
\label{sec:goodnessoffit}

In this section we consider the problem of testing for goodness-of-fit in the hypergraph $\bm \beta$-model. In particular, given $\bm \gamma \in \R^n$ and a sample $H_n \sim \mathsf{H}_{n, [r]}(n, \bm B)$, with $\bm B = (\bm \beta_2, \ldots, \bm \beta_r)$, we consider the following hypothesis testing problem: For $2 \leq s \leq r$, 
\begin{align}\label{eq:H0gamma}
H_0:\bm{\beta}_s=\bm \gamma \quad \mbox{versus} \quad H_1:\bm{\beta}_s \neq \bm \gamma.
\end{align} 
This section is organized as follows: In Section \ref{sec:asymptoticH0} we derive the asymptotic distribution and detection rates of the likelihood ratio (LR) test for the problem \eqref{eq:H0gamma}.  In Section \ref{sec:minimaxl2} we show that the detection rate of the LR test is minimax optimal for testing in $L_2$ norm. Rates for testing in $L_\infty$ norm are derived in Section \ref{sec:testingmax}.  

\subsection{Asymptotics of the Likelihood Ratio Test}
\label{sec:asymptoticH0}

Consider the LR statistic for the testing problem \eqref{eq:H0gamma}: 
\begin{align}\label{eq:lambda}
\log \Lambda_{n, s} = \ell_{n, s}(\bm \gamma)-\ell_{n, s}(\hat{\bm \beta}_s) ,   
\end{align} 
where $\ell_{n, s}$ is the \revsag{negative} log-likelihood function \eqref{eq:lik_k} and $\hat{\bm \beta}_s$ is the ML estimate \eqref{eq:ml}. The following theorem proves the limiting distribution of the LR statistic \eqref{eq:lambda} under $H_0$. 

\begin{thm}
\label{thm:testing_inhom_level}
Suppose \revsag{$\bm \gamma \in \cB_M $}. 
Then under $H_0$, 
\begin{equation}
\label{eq:lambdaH0}
  \lambda_{n, s} := \frac{2\log\Lambda_{n, s} - n}{\sqrt{2n}} \dto  \cN(0,1) , 
\end{equation} 
for $\log\Lambda_{n, s}$ as defined in \eqref{eq:lambda}. 
\end{thm} 

The proof of Theorem \ref{thm:testing_inhom_level} is given in Appendix \ref{sec:hypothesisdistributionpf}. To prove the result we first 
expand $\log \Lambda_{n, s}$ around the null parameter $\bm \gamma$ and derive an asymptotic expansion of $\lambda_{n, s}$ in terms of the sum of squares of the $s$-degree sequence $(d_s(1), d_s(2), \ldots, d_s(n))^\top$ (see \eqref{eq:degreequadratic}). Since the degrees are asymptotically independent (recall Theorem \ref{thm:central_lim_thm}), we can show that the sum of squares of the degrees (appropriately centered and scaled) is asymptotically normal (see Proposition \ref{ppn:lambdaH0}), establishing the result in \eqref{eq:lambdaH0}.

Theorem \ref{thm:testing_inhom_level} shows that the LR test 
\begin{align}\label{eq:testlambda}
\phi_{n, s} := \mathbbm 1 \left\{ |\lambda_{n, s}| >z_{\alpha/2} \right\} , 
\end{align}  
where $z_{\alpha/2}$ is the $(1-\alpha/2)$-th quantile of the standard normal distribution, has asymptotic level $\alpha$. To study the power of this test consider the following testing problem: 
\begin{align}\label{eq:H1gamma}
H_0:\bm{\beta}_s=\bm \gamma \quad \mbox{versus} \quad H_1:\bm{\beta}_s =\bm \gamma' , 
\end{align} 
where $\bm \gamma' \neq \bm \gamma$ is such that  $\|\bm \gamma - \bm \gamma' \|_2=O(1)$. Recall that $\bm d_s = (d_s(1), d_s(2), \ldots, d_s(n))^\top$ is the vector of $s$-degrees. Also, $\Cov_{\bm \gamma}[\bm d_s]$ will denote the covariance matrix of the vector of $s$-degrees (see \eqref{eq:v_s}). 

\begin{thm}
\label{thm:testing_inhom_power} 
Suppose \eqref{eq:lambdaH0} holds and $\bm \gamma'$ \revsag{be} as in \eqref{eq:H1gamma}, \revsag{then} the asymptotic power of the test $\phi_{n, s}$ defined in \eqref{eq:testlambda} satisfies: 
\begin{equation}\label{eq:asymptoticH1}
    \lim_{n \rightarrow \infty} \E_{\bm \gamma'}[\phi_{n, s}] = \begin{cases}
    \alpha & \mbox{if\; $\|\bm \gamma'-\bm \gamma\|_2 \ll n^{-\frac{2s-3}{4}}$,}\\
    1 & \mbox{if \;$\|\bm \gamma'-\bm \gamma\|_2 \gg n^{-\frac{2s-3}{4}}$.}
    \end{cases} 
\end{equation} 
Moreover, if $n^{\frac{2s-3}{4}} \|\bm \gamma'-\bm \gamma\|_2 \rightarrow \tau \in (0, \infty)$, then there exists $\eta \in (0, \infty)$ depending on $\tau$ such that 
$$\eta = \lim_{n \rightarrow \infty} \frac{(\bm \gamma'-\bm \gamma)^\top \Cov_{\bm \gamma}[\bm d_s] (\bm \gamma'-\bm \gamma)}{\sqrt n},$$ \revsag{where the limit always exists along a subsequence}, and 
\begin{equation}\label{eq:asymptoticgamma}
    \lim_{n \rightarrow \infty} \E_{\bm \gamma'}[\phi_{n, s}] = \P\left(\left| \cN(-\tfrac{\eta}{\sqrt 2}, 1)\right| > z_{\alpha/2}\right) . 
\end{equation}   
\end{thm}

The proof of Theorem \ref{thm:testing_inhom_power} is given in Appendix \ref{sec:asymptoticH1pf}. It entails analyzing the asymptotic distribution of the scaled LR statistic $\lambda_{n, s}$ under $H_1$ as in \eqref{eq:H1gamma}. Specifically, we show that when $\|\bm \gamma'-\bm \gamma\|_2 \ll n^{-\frac{2s-3}{4}}$, then $\lambda_{n, s} \dto \cN(0, 1)$, hence the LR test \eqref{eq:lambdaH0} is asymptotically powerless in detecting $H_1$. On the other hand,  if $\|\bm \gamma'-\bm \gamma\|_2 \gg n^{-\frac{2s-3}{4}}$, then the $\lambda_{n, s}$ diverges to infinity, hence the LR test is asymptotically powerful. In other words, $n^{-\frac{2s-3}{4}}$ is the detection threshold in the $L_2$ norm of the LR test. 
We also derive the limiting power function of the LR test at the threshold $n^{\frac{2s-3}{4}} \|\bm \gamma'-\bm \gamma\|_2 \rightarrow \tau \in (0, \infty)$. In this case, $\lambda_{n, s} \dto \cN(-\eta/\sqrt 2, 1)$, where `effective signal strength' $\eta$ is the limit of the scaled Mahalanobis distance between $\bm \gamma$ and $\bm 
\gamma'$, where the dispersion matrix is the covariance matrix of the degrees. \revmark{Here, the Mahalanobis distance between two vectors $\bm \gamma$ and $\bm \gamma'$ with the dispersion matrix $\bm\Sigma$ refers to the quantity $(\bm\gamma-\bm\gamma')^\top\bm\Sigma(\bm \gamma-\bm \gamma')$.}
In the next section we will show that this detection rate is, in fact, minimax optimal.

\subsection{Minimax Detection Rate in the $L_2$ Norm} 
\label{sec:minimaxl2}

In this section we will show that the detection threshold of the LR test obtained in Theorem \ref{thm:testing_inhom_power} is information-theoretically tight. 
To formalize this consider the testing problem: For $\varepsilon > 0$ and \revsag{$\bm \gamma \in \cB_M$}, 
\begin{align}\label{eq:gammal2}
H_0:\bm{\beta}_s=\bm \gamma \quad \mbox{versus} \quad H_1: \|\bm{\beta}_s -\bm \gamma \|_2 \geq \varepsilon . 
\end{align}
The worst-case risk of a test function $ \psi_n$ for the testing problem \eqref{eq:gammal2} is defined as: 
\begin{align}\label{eq:Rpsi}
\cR (\psi_n, \bm \gamma ) =  \P_{H_0}( \psi_n=1)   + \sup_{ \revsag{\bm \gamma' \in \cB_M} : \| \bm \gamma' - \bm \gamma \|_2 \geq \varepsilon} \P_{\bm \gamma'}( \psi_n=0), 
\end{align} 
which is the sum of the Type I error and the maximum possible Type II error of the test function $ \psi_n$. 
Given $H_n \sim \mathsf{H}_{n, s}(n, \bm \beta_s)$, for some $\revsag{\bm \beta_s \in \cB_M}$, and $\varepsilon = \varepsilon_n$ (depending on $n$), a sequence of test functions $\psi_n$ is said to be {\it asymptotically powerful} for \eqref{eq:Rpsi}, if for all $\bm \gamma \in \cB_M$ $\lim_{n \rightarrow \infty }\cR( \psi_n, \bm \gamma)=0$. On the other hand, a sequence of test functions $\psi_n$ is said to be {\it asymptotically powerless} for \eqref{eq:Rpsi}, if there exists $\revsag{\bm \gamma \in \cB_M}$ such that $\lim_{n \rightarrow \infty }\cR( \psi_n, \bm \gamma )=1$. 

\begin{thm}\label{thm:H0l2}
Given $H_n \sim \mathsf{H}_{n, s}(n, \bm \beta_s)$ and $\revsag{\bm \gamma \in \cB_M}$, consider the testing problem \eqref{eq:gammal2}. Then the following hold: 

\begin{itemize} 

\item[(a)] The LR test \eqref{eq:testlambda} is asymptotically powerful for \eqref{eq:gammal2}, when $\varepsilon \gg n^{-\frac{2s-3}{4}}$. 

\item[(b)] On the other hand, all tests are asymptotically powerless for \eqref{eq:gammal2}, when $\varepsilon \ll n^{-\frac{2s-3}{4}}$. 

\end{itemize}

\end{thm}

The result in Theorem \ref{thm:H0l2} (a) is a direct consequence of Theorem \ref{thm:testing_inhom_power}. The proof of Theorem \ref{thm:H0l2} (b) is given in Appendix \ref{sec:H0l2pf}.  For this we chose $\bm \gamma = \bm 0 \in \R^n$ and randomly perturb (that is, randomly add or subtract $\varepsilon/\sqrt n$) the coordinates of $\bm \gamma$ to construct $\revsag{\bm \beta_s \in \cB_M}$ satisfying $\|\bm{\beta}_s -\bm \gamma \|_2 \geq \varepsilon$. Then a second-moment calculation of the likelihood ratio shows that detecting these two models is impossible for $\varepsilon \ll n^{-\frac{2s-3}{4}}$. These results combined show that $n^{-\frac{2s-3}{4}}$ is the minimax detection rate for the testing problem \eqref{eq:gammal2} and the LR test attain the minimax rate.


\begin{remark} (Comparison between testing and estimation rates.) Recall from \eqref{eq:mlerlayer} and \eqref{eq:l2lb} that the minimax rate of estimating $\hat {\bm \beta}_s$ in the $L_2$ norm is $n^{-\frac{s-2}{2}}$. On the other hand, Theorem \ref{thm:H0l2} shows that the minimax rate of testing in the $L_2$ norm is $n^{-\frac{2s-3}{4}} \ll n^{-\frac{s-2}{2}}$. For example, in the graph case (where $s=2$), the estimation rate is $\Theta(1)$ whereas the rate of testing is $n^{-\frac{1}{4}}$. This is an instance of the well-known phenomenon that high-dimensional estimation is, in general, harder that testing in the squared-error loss. 
\end{remark}

\subsection{Testing in the $L_\infty$ Norm} 
\label{sec:testingmax}

In this section we consider the goodness-of-fit problem when \revsag{the} separation is measured in the $L_\infty$ norm. This complements our results on estimation in \revsag{the} $L_\infty$ norm in Theorem \ref{thm:layermle}. Towards this, as in \eqref{eq:gammal2}, consider the testing problem: For $\varepsilon > 0$ and $\revsag{\bm \gamma \in \cB_M}$, 
\begin{align}\label{eq:H0maximum}
H_0:\bm{\beta}_s=\bm \gamma \quad \mbox{versus} \quad H_1: \|\bm{\beta}_s -\bm \gamma \|_\infty \geq \varepsilon . 
\end{align} 
In this case the minimax risk of a test function is defined as in \eqref{eq:Rpsi} with the $L_2$ norm $\| \bm \gamma' - \bm \gamma \|_2$ replaced by the $L_\infty$ norm $\| \bm \gamma' - \bm \gamma \|_\infty$. Then consider the test: 
\begin{align}\label{eq:test_max}
\phi_{n, s}^{\max} := \mathbbm 1\left\{ \|\hat{\bm \beta}_{s} - \bm \gamma \|_\infty \geq \,2C \sqrt{\frac{\log n}{n^{s-1}}} \right\},
\end{align} 
where $C: = C(s, M) > 0$ is chosen according to \eqref{eq:mlerlayer} such that 
$$\P_{\bm \kappa} \left( \|\hat{\bm \beta}_{s} - \bm \kappa \|_\infty \leq \, C \sqrt{\frac{\log n}{n^{s-1}}} \right) \rightarrow 1, $$ 
for all $\revsag{\bm \kappa \in \cB_M}$. This implies,  $\E_{\bm \gamma}[\phi_{n, s}^{\max}] \rightarrow 0$. Also, for $\revsag{\bm \gamma' \in \cB_M}$ such that $\| \bm \gamma - \bm \gamma' \|_\infty \geq \varepsilon$, 
\begin{align}\label{eq:lmaxpower}
\E_{\bm \gamma'}[\phi_{n, s}^{\max}] = \P_{\bm \gamma'}\left( \|\hat{\bm \beta}_{s} - \bm \gamma \|_\infty \geq \,2C \sqrt{\frac{\log n}{n^{s-1}}} \right) \geq 
\P_{\bm \gamma'} \left( \|\hat{\bm \beta}_{s} - \bm \gamma' \|_\infty \leq \, C \sqrt{\frac{\log n}{n^{s-1}}} \right)
\rightarrow 1 , 
\end{align}
whenever $\varepsilon \gg \sqrt{\log n/n^{s-1}}$. This is because $\|\hat{\bm \beta}_{s} - \bm \gamma' \|_\infty \leq\, C \sqrt{\log n/n^{s-1}}$ implies,   
$$ \| \hat{\bm \beta}_{s} - \bm \gamma \|_\infty \geq \|\bm \gamma - \bm \gamma' \|_\infty - \|\hat{\bm \beta}_{s} - \bm \gamma' \|_\infty \geq  \varepsilon - C \sqrt{\frac{\log n}{n^{s-1}}} \geq 2C \sqrt{\frac{\log n}{n^{s-1}}}, $$ 
whenever $\varepsilon \gg \sqrt{\log n/n^{s-1}}$. This implies that the test 
$\phi_{n, s}^{\max}$ in \eqref{eq:test_max} is asymptotically powerful for \eqref{eq:H0maximum} whenever $\varepsilon \gg \sqrt{\log n/n^{s-1}}$. The following result shows that this rate is optimal (up to a factor of $\sqrt{\log n}$) for testing in the $L_\infty$ norm.

\begin{thm}\label{thm:H0lmax}
Given $H_n \sim \mathsf{H}_{n, s}(n, \bm \beta_s)$ and $\revsag{\bm\gamma, \bm \beta_s \in \cB_M}$, consider the testing problem \eqref{eq:H0maximum}. Then the following hold: 

\begin{itemize} 

\item[(a)] The test $\phi_{n, s}^{\max}$ in \eqref{eq:lmaxpower} is asymptotically powerful for \eqref{eq:H0maximum}, when $\varepsilon \gg \sqrt{\frac{\log n}{n^{s-1}}}$. 

\item[(b)] On the other hand, all tests are asymptotically powerless for \eqref{eq:H0maximum}, when \revsag{$\varepsilon \ll \sqrt{\frac{\log n}{n^{s-1}}}$}. 

\end{itemize}

\end{thm}

The proof of Theorem \ref{thm:H0lmax} (b) is given in Appendix \ref{sec:H0lmaxpf}. Note that in this case minimax rates of estimation and testing are the same, since the effect of high-dimensional aggregation does not arise when separation is measured in the $L_\infty$ norm.

\section{Numerical Experiments}
\label{sec:simulations}

In  this section we study the performance of the ML estimates and the LR tests discussed above in simulations. To begin with we simulate a 3-uniform hypergraph $\bm \beta$-model $\mathsf{H}_{3}(n, \bm \beta)$, with $n=400$ vertices and $\bm \beta = \bm 0 \in \R^n$.  
Figure \ref{fig:LRestimation}(a) shows the quantile-quantile (QQ) plot (over 200 iterations) of the first coordinate of the ML estimate $[\bm D]_1 ([\hat {\bm \beta} - \bm \beta]_1)$ (where $\hat {\bm \beta}$ is computed using the fixed point algorithm described in \cite{stasi2014beta}. \revsag{Here} $\bm D$ is as defined in Corollary \ref{cor:central_lim_thm}). 
We observe that the empirical quantiles closely follow the quantiles of the standard normal distribution, validating the result in Corollary \ref{cor:central_lim_thm}.

\begin{figure}[h!]
\begin{minipage}[b]{0.3\linewidth}
\centering
\includegraphics[width=\textwidth]{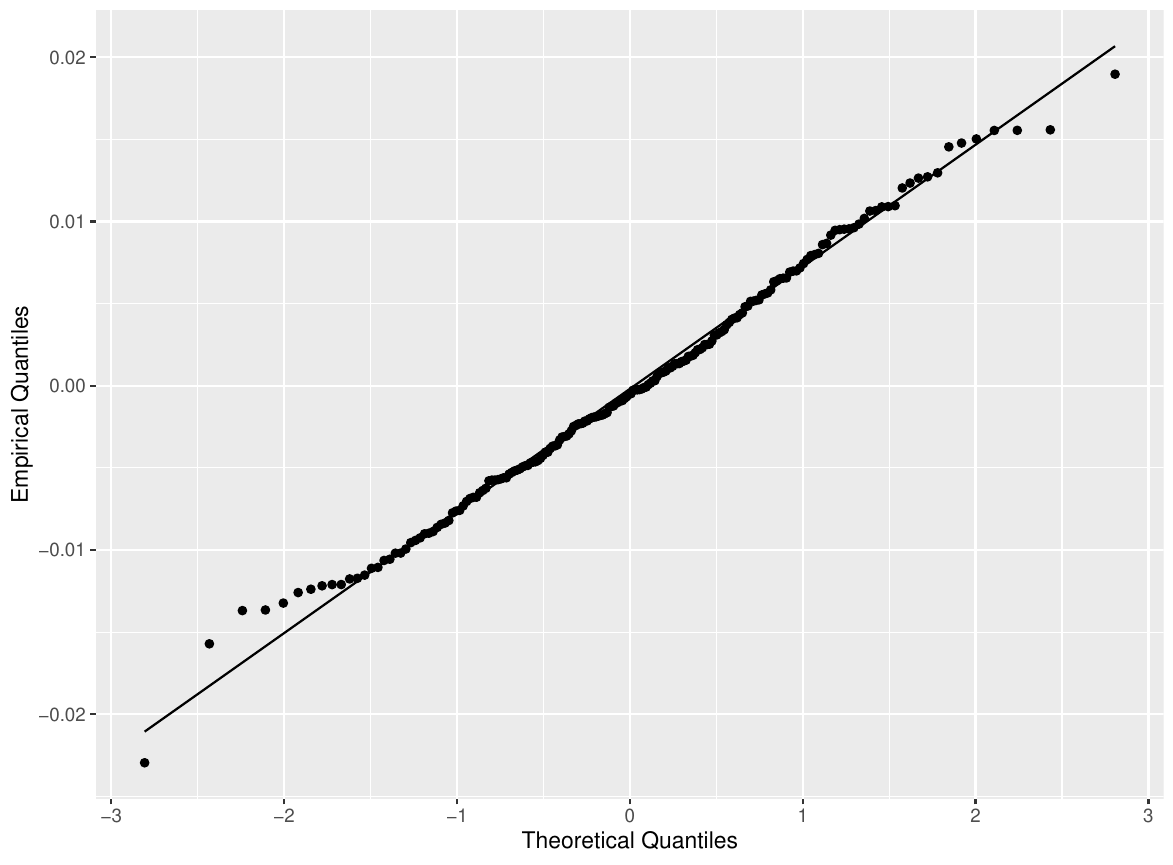}
{ \small{ (a) } }
\end{minipage}
\begin{minipage}[b]{0.3\linewidth}
\centering
\includegraphics[width=\textwidth]{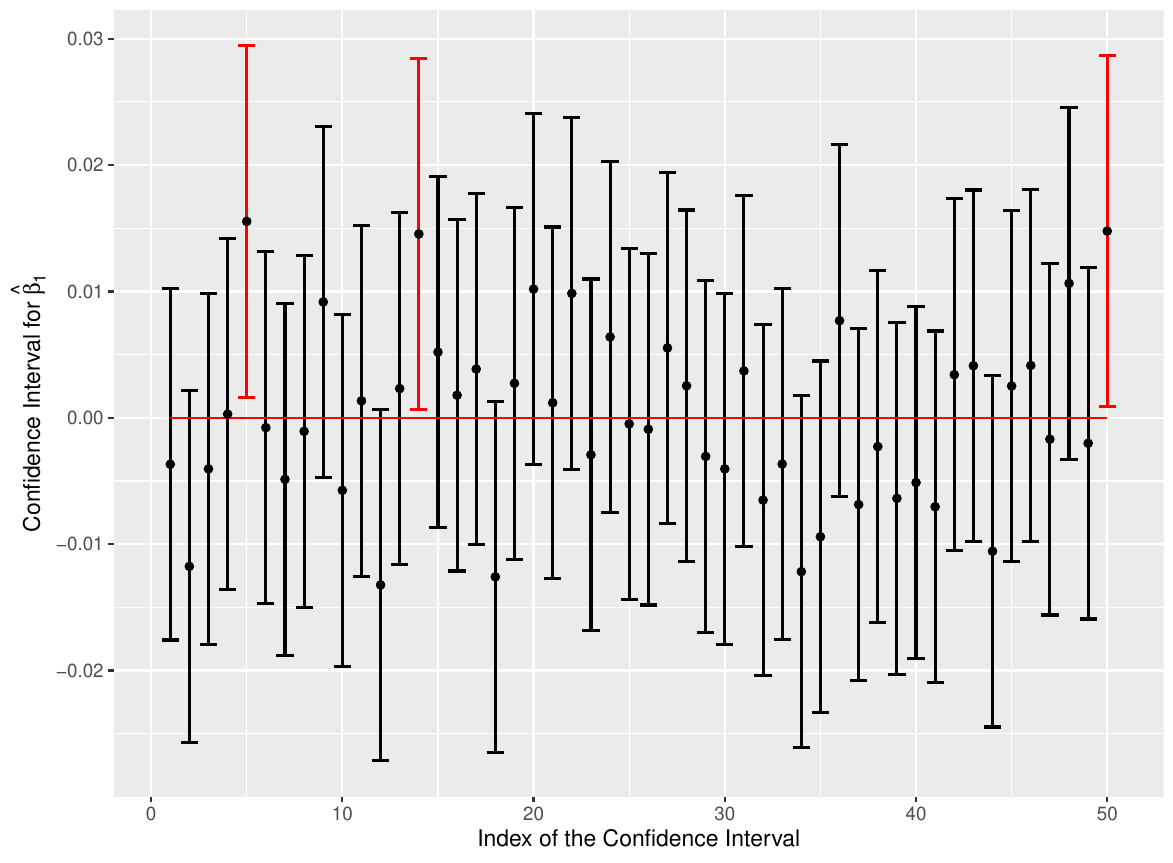}
{ \small{ (b) } }
\end{minipage} 
\begin{minipage}[b]{0.3\linewidth}
\centering
\includegraphics[width=\textwidth]{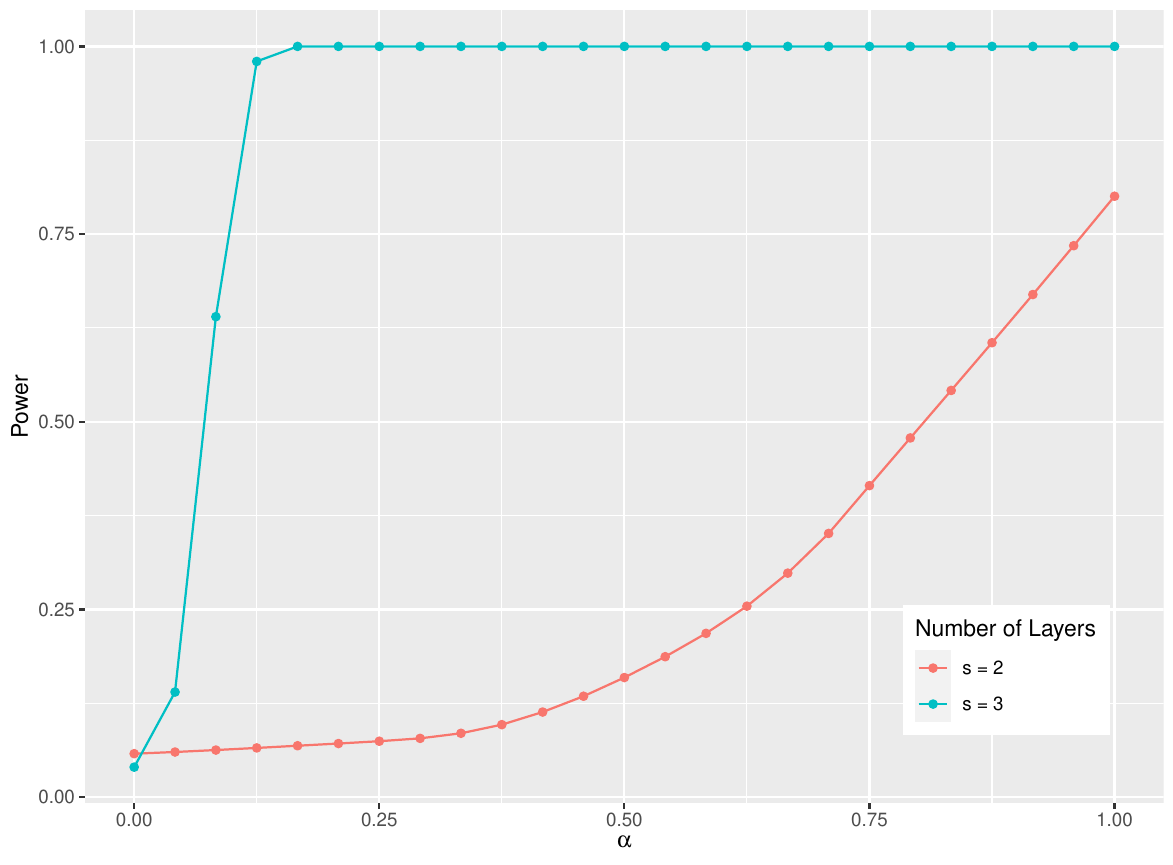}
{ \small{ (c) } }
\end{minipage} 
 \caption{\small{(a) QQ plot of the ML estimate $\hat \beta_1$, (b) confidence intervals for $\beta_1$, and (c) power of the LR test for the goodness of fit problem \eqref{eq:H01s}. \revsag{The first two plots correspond  to the 3-uniform hypergraph $\bm \beta$-model whereas plot (c) corresponds to $s$-uniform hypergraph $\bm \beta$-model for $s=2,3$.}}} 
\label{fig:LRestimation} 
\end{figure}

In the same setup as above, Figure \ref{fig:LRestimation}(b) shows the 95\% confidence interval for  $[\bm \beta]_1$ over 50 iterations. Specifically, we plot the intervals  
$$\left[ [\hat {\bm \beta}]_1 - \frac{1.96}{[\hat{ \bm D}]_1}, [\hat {\bm \beta}]_1 + \frac{1.96}{[\hat{\bm D}]_1} \right], $$ 
where $\hat {\bm D}$ is the estimate of $\bm D$ as defined in Theorem \ref{thm:conf_int_thm}. 
This figure shows that 47 out of 50 of intervals cover the true parameter, which gives an empirical coverage of $47/50= 0.94$. 

Next, we consider the goodness of fit problem in $s$-uniform hypergraph $\bm \beta$-model: 
\begin{align}
\label{eq:H01s}
\revsag{
H_0:\bm \beta_s=0 \quad \mbox{ versus } \quad H_1:\bm \beta_s \neq 0,}
\end{align}
for $s=2,3$. For this we simulate \revsag{$H_n \sim \mathsf{H}_{s}(n, \bm \gamma)$}, with $n=250$ and $\bm \gamma  = \alpha \cdot \bm u$, where $\bm u$ is chosen uniformly at random from the $n$-dimensional unit sphere and $\alpha \in [0, 1]$.  Figure \ref{fig:LRestimation}(c) shows the empirical power of the LR test \eqref{eq:testlambda} (over 50 iterations) as $\alpha$ varies over a grid of $25$ uniformly spaced values in $[0,1]$, for $s=2, 3$. 
In both cases, as expected, the power increases with $\alpha$, which, in this case, determines the signal strength. Also, the LR test is more powerful in the 3-uniform case compared to the 2-uniform case. This aligns with conclusions of Theorem \ref{thm:testing_inhom_power}, which shows that the detection threshold of the LR test in the 3-uniform case is $n^{-\frac{3}{4}}$, while for 2-uniform case it is $n^{-\frac{1}{4}}$. Hence, one expects to see more power at lower signal strengths (smaller $\alpha$) for $s=3$ compared to $s=2$.

\small 

\subsection*{Acknowledgements} 
B. B. Bhattacharya was supported by NSF CAREER grant DMS 2046393, NSF grant DMS 2113771, and a Sloan Research Fellowship. The authors thank the anonymous referees for their insightful comments which improved the quality and the presentation of the paper. The authors also thank Rui Feng for pointing out an error in a previous draft.

\small 
\bibliographystyle{abbrvnat}
\bibliography{Hypergraph_beta_model}


\normalsize 

\appendix

\section{Proof of Theorem \ref{thm:layermle}}
\label{mle_rate}

\subsection{Convergence Rate in the $L_2$ Norm} 
\label{sec:layermleL2pf} 

As mentioned in the Introduction, the proof of Theorem \ref{thm:layermle} involves showing the following: (1) a concentration bound on the gradient of negative log-likelihood $\ell_{n, s}$ (recall \eqref{eq:lik_k}) at the true parameter value \revsag{$\bm B = (\bm \beta_2, \ldots, \bm \beta_r)$}, and (2) the strong convexity of $\ell_{n, s}$ in a neighborhood of the true parameter. We begin with the concentration of the gradient $\nabla {\ell}_{n, s}$ in both the $L_2$ and the $L_\infty$ norms: 

\begin{lem}
\label{lem:gradientH}  
Suppose the assumptions of Theorem \ref{thm:layermle} hold. Then for each $2 \leq s \leq r$, there exists a constant $C>0$ (depending on $r$ and $M$) such that the following hold:  
\begin{align}\label{eq:gradientloglikelihood}
\|\nabla {\ell}_{n, s}(\bm{\beta}_s)\|_2^2 \le C\,n^{s} \quad \text{ and } \quad \|\nabla {\ell}_{n, s}(\bm{\beta}_s)\|^2_\infty \le C \,n^{s-1}\log n , 
\end{align} 
with probability $1-O\left( \frac{1}{n^2} \right)$. 
\end{lem}

The next step is to establish the strong convexity of $\ell_{n, s}$. Towards this we need to show that the smallest eigenvalue $\lambda_{\min}(\nabla^2 \ell_{n, s})$ of the Hessian matrix $\nabla^2 \ell_{n, s}$ (appropriately scaled) is bounded away from zero in a neighborhood of the true value $\bm \beta_s$. This is the content of the following lemma, which also establishes a matching upper bound on the largest eigenvalue $\lambda_{\max}(\nabla^2 \ell_{n, s})$ of the Hessian matrix $\nabla^2 \ell_{n, s}$.

\begin{lem}
\label{lem:lower_bound_hess} 
Suppose the assumptions of Theorem \ref{thm:layermle} hold. Fix $2 \leq s \leq r$ and a constant $K > 0$. 
Then there exists constants $C_1', C_2' >0$ (depending on $r$ and $M$) such that the following hold: 
\begin{align}\label{eq:hessainlambdal2}
{
C_1' e^{-s \| \bm \beta - \bm \beta_s \|_\infty} n^{s-1} \le  \lambda_{\min}(\nabla^2{\ell}_{n, s}(\bm{\beta})) \le \lambda_{\max}(\nabla^2{\ell}_{n, s}(\bm{\beta})) \le {C}_2' n^{s-1}.}
\end{align}
As a consequence, there exists a constants $C_1, C_2 >0$ (depending on $r$, $K$, and $M$) such that the following hold:  
\begin{align}\label{eq:hessainlambda}
C_1 n^{s-1} \le \inf_{\bm \beta: \| \bm \beta - \bm \beta_s \|_2 \leq K}\lambda_{\min}(\nabla^2{\ell}_{n, s}(\bm{\beta})) \le \sup_{\bm \beta: \| \bm \beta - \bm \beta_s \|_2 \leq K} \lambda_{\max}(\nabla^2{\ell}_{n, s}(\bm{\beta})) \le {C}_2 n^{s-1}.
\end{align}
\end{lem}

The proofs of Lemma \ref{lem:gradientH} and Lemma \ref{lem:lower_bound_hess} are given in Appendix \ref{sec:gradientHpf} and Appendix \ref{sec:hessianHpf}, respectively. We first apply these results to prove the rate of convergence in the $L_2$ norm in Theorem \ref{thm:layermle}. 

\subsubsection{Deriving the $L_2$ Norm Bound in \eqref{eq:mlerlayer} } 

To begin with suppose the ML equations \eqref{eq:def_k_unif} have a solution $\hat{\bm B}=(\hat{\bm \beta}_2, \ldots, \hat{\bm \beta}_r)$. This implies, $\grad \ell_{n, s}(\hat{\bm \beta}_s) = 0$, for $2 \leq s \leq r$, 
where $\ell_{n, s}$ is as defined in \eqref{eq:lik_k}.   
For $2 \leq s \leq r$ and $0 \leq t \leq 1$, define
$$\revmark{\bar{\bm \beta}}_s(t) := t \hat{\bm \beta}_s + (1- t) \bm \beta_s,$$ 
and $g_s(t) := (\hat{\bm \beta}_s - \bm \beta_s)^\top \grad \ell_{n, s}(\revsag{\bar{\bm \beta}}_s(t))$. Note that $\grad \ell_{n, s}(\bm \beta_s(1)) = \grad \ell_{n, s}(\hat{\bm \beta}_s) = 0$. Hence, by the Cauchy-Schwarz inequality, 
\begin{align}\label{eq:gs01}
|g_s(1) - g_s(0)| = |(\hat{\bm \beta}_s - \bm \beta_s)^\top \grad \ell_{n, s}(\bm \beta_s)| \leq  \| \hat{\bm \beta}_s - \bm \beta_s \|_2 \cdot \| \grad \ell_{n, s}(\bm \beta_s) \|_2 . 
\end{align}
Also,
\begin{align}\label{eq:gderivative}
g_s'(t) = (\hat{\bm \beta}_s - \bm \beta_s)^\top \grad^2 \ell_{n, s}(\revsag{\bar{\bm \beta}}_s(t) )   (\hat{\bm \beta}_s - \bm \beta_s) \geq \lambda_{\min}(\grad^2 \ell_{n, s}(\revsag{\bar{\bm \beta}}_s(t) )) \| \hat{\bm \beta}_s - \bm \beta_s \|_2^2. 
\end{align} 
We now consider two cases: To begin with assume $s \geq 3$. By Lemma \ref{lem:lower_bound_hess}, given a constant $K  > 0$ there exists a constant $C_1 > 0$ (depending on $r, K, M$) such that 
\begin{align}\label{eq:gderivativelambda} 
\inf_{\bm \beta: \| \bm \beta - \bm \beta_s \|_2 \leq K} \lambda_{\min}(\grad^2 \ell_{n, s}(\bm \beta)) \geq C_1 n^{s-1}. 
\end{align}
Note that $\| \revsag{\bar{\bm \beta}}_s(t) - \bm \beta_s \|_2 = \revsag{t} \| \hat{\bm \beta}_s - \bm \beta_s \|_2 $. Then 
\begin{align*} 
| g_s(1) - g_s(0)| \geq g_s(1) - g_s(0) & = \int_0^1 g_s'(t) \mathrm dt \nonumber \\ 
& \geq \int_0^{\min\{1, \frac{K}{\| \hat{\bm \beta}_s - \bm \beta_s \|_2}\}} g_s'(t) \mathrm dt \nonumber \\ 
& \geq C_1 n^{s-1} \| \hat{\bm \beta}_s - \bm \beta_s \|_2^2 \min \left\{1, \frac{K}{\| \hat{\bm \beta}_s - \bm \beta_s \|_2} \right \} , 
\end{align*} 
where the last step follows from \eqref{eq:gderivative} and \eqref{eq:gderivativelambda}. Therefore, by \eqref{eq:gs01} and Lemma \ref{lem:gradientH}, with probability $1-O(\frac{1}{n^2})$,   
\begin{align}\label{eq:hatbetaminimum} 
\min\{ \| \hat{\bm \beta}_s - \bm \beta_s \|_2, K \} 
\lesssim_{r, K, M} \frac{1}{n^{s-1}}  \cdot \| \grad \ell_{n, s}(\bm \beta_s) \|_2 \lesssim_{r, K, M} \sqrt{\frac{1}{n^{s-2}} } .  
\end{align} 
Since $K > 0$ is fixed and the \revsag{RHS of \eqref{eq:hatbetaminimum}} converges to zero for $s \geq 3$, the $L_2$ norm bound in \eqref{eq:mlerlayer} follows, under the assumption that  ML equations \eqref{eq:def_k_unif} have a solution.  

\revmark{
Next, suppose $s=2$. Since $\| \bar{\bm \beta}_2(t) - \bm \beta_2 \|_\infty =t \| \hat{\bm \beta}_2 - \bm \beta_2 \|_\infty $. Since $t \in [0,1]$, by Lemma \ref{lem:lower_bound_hess}, 
\begin{align}\label{eq:gderivativelambda_1} 
\lambda_{\min}(\grad^2 \ell_{n, 2}(\bar{\bm \beta}_2(t) )) \geq C_1' e^{-2t \| \hat{\bm \beta}_2 - \bm \beta_2 \|_\infty } n  \geq C_1' e^{-2 \| \hat{\bm \beta}_2 - \bm \beta_2 \|_\infty } n, 
\end{align} 
for some constant $C_1' > 0$ depending on $M$.  
Then 
\begin{align*} 
| g_2(1) - g_2(0)| \geq g_2(1) - g_2(0) & = \int_0^1 g_2'(t) \mathrm dt \nonumber \\ 
& \geq C_1' n  \| \hat{\bm \beta}_2 - \bm \beta_2 \|_2^2 e^{-2 \| \hat{\bm \beta}_2 - \bm \beta_2 \|_\infty }  \nonumber \tag*{(by \eqref{eq:gderivative} and \eqref{eq:gderivativelambda_1}) } . \nonumber 
\end{align*} 
Therefore, by \eqref{eq:gs01} and Lemma \ref{lem:gradientH}, with probability $1-o(1)$,   
\begin{align}\label{eq:hatbetaminimum_2} 
 \| \hat{\bm \beta}_2 - \bm \beta_2 \|_2e^{-2 \| \hat{\bm \beta}_2 - \bm \beta_2 \|_\infty } 
\leq \frac{1}{C_1' n }  \cdot \| \grad \ell_{n, 2}(\bm \beta_2) \|_2 \leq C'  , 
\end{align}  
for some constant $C' > 0$ depending on $M$. Hence, if there exists a bounded solution to \eqref{eq:def_k_unif}, the $L_2$ norm rate will follow for $s=2$. 
}

To complete the proof we need to show that bounded solution to equation \eqref{eq:def_k_unif} exists.
To this end, for $2 \leq s \leq r$, denote by $\mathcal{D}_s$, the set of all possible degree sequences in an $s$-uniform hypergraph on $n$ vertices. Moreover, let $\mathcal{R}_s$ be the set of all expected degree sequences in a hypergraph on $n$ vertices sampled from the $s$-uniform model \eqref{eq:Huniformbeta}. 
The following result shows that any convex combination of $s$-degree 
sequences in $\mathcal{D}_s$ can be reached by the limit of expected degree sequences of the $s$-uniform hypergraph $\bm \beta$-model. This was proved in the graph case ($s=2$) by Chatterjee et al.~\cite[Theorem 1.4]{chatterjee2011random}. Here, we show that the same holds for all $2 \leq s \leq r$.

\begin{prop}
\label{ppn:conv_hull} 
Fix $2 \leq s \leq r$ and let $\mathcal{D}_s$ and $\mathcal{R}_s$ be as defined above. Then $\mathrm{conv}\,(\mathcal{D}_s)=\bar{\mathcal{R}}_s$, 
where $\mathrm{conv}\,(\mathcal{D}_s)$ denotes the convex hull of $\mathcal D_s$ and $\widebar{\mathcal{R}}_s$ is the closure of $\mathcal{R}_s$.  
\end{prop}

The proof of the above result is given in Appendix \ref{sec:lemmapf}. Using this proposition we now show the existence of bounded solutions of the ML equations \eqref{eq:def_k_unif}. Note that by Proposition \ref{ppn:conv_hull}, given $H_n \sim \mathsf{H}_{n, [r]}(n, \bm B)$ the $s$-degree sequence $(d_s(1),\ldots, d_s(n)) \in \cD_s \subseteq \bar{\mathcal{R}}_s$.  This implies, there exists a sequence $\{\bm x_t\}_{t \geq 0} \in \cR_s$ satisfying 
$$\lim_{t \rightarrow \infty} \bm x_t = (d_s(1), \ldots, d_s(n)) . $$ Since $\bm x_t \in \cR_s$, there exists \revsag{$\{\hat{\bm \beta}_2^{(t)}, \ldots, \hat {\bm \beta}_r^{(t)}\}$} such that 
\begin{equation}
\label{eq:defsequencet}
\revsag{(\bm x_t)_v = \sum_{\{v_2, \ldots, v_s\} \in {{ [n] \backslash\{v\} } \choose { s-1 }} }\frac{ e^{\hat{\beta}^{(t)}_{s, v}+\hat{\beta}^{(t)}_{s, v_2}+\ldots+\hat{\beta}^{(t)}_{s, v_s}} }{1+ e^{\hat{\beta}^{(t)}_{s, v}+\hat{\beta}^{(t)}_{s, v_2}+\ldots+\hat{\beta}^{(t)}_{s, v_s}}} }, 
\end{equation} 
for \revsag{$1 \le v \le n$ and} $2 \leq s \leq r$. In other words, for each $t \geq 0$, \revsag{$\{\hat{\bm \beta}_2^{(t)}, \ldots, \hat {\bm \beta}_r^{(t)}\}$} is a solution of the ML equations \eqref{eq:def_k_unif} with $(d_s(1), d_s(2), \ldots, d_s(n))$ replaced by $\bm x_t$.  
By the previous argument, there exists $C > 0$ (not depending on $t$) such that with probability $1-o(1)$, $$\max_{2 \leq s \leq r}\|\hat{\bm \beta}^{(t)}_s\|_\infty \le C,$$ 
for all $t \geq 0$. Therefore, the sequence \revsag{$\{(\hat{\bm \beta}^{(t)}_2, \ldots, \hat{\bm \beta}^{(t)}_r) \}_{\geq 0}$} has a limit point. This limit point is a solution to \eqref{eq:def_k_unif} (by taking limit as $t \rightarrow \infty$ in \eqref{eq:defsequencet}) and is bounded. Finally, since $\ell_{n,s}$ is strongly convex for \revsag{$\bm \beta \in \cB_M$} (see \eqref{eq:hessainlambdal2}), if the gradient equations have a bounded solution, it is the unique minimizer. Therefore, there exists a unique bounded solution to \eqref{eq:def_k_unif} which is the minimizer of $\ell_{n,s}$.

\subsubsection{Proof of Lemma \ref{lem:gradientH}}
\label{sec:gradientHpf}

Recalling \eqref{eq:def_k_unif} note that, for $v \in [n]$, $v$-th coordinate of the gradient of $\grad \ell_{n, s}$ is given by:  
\begin{align} \label{eq:likelihoodgrad} 
\grad \ell_{n, s}(\bm \beta_s)_{v}  & = \E[d_s(v)] - d_s(v)
\end{align}  
where 
\begin{align}\label{eq:expecteddegree}
\E[d_s(v)] = \sum_{\{v_2, \ldots, v_s\} \in {{ [n] \backslash\{v\} } \choose { s-1 }} }\frac{ e^{\beta_{s, v}+\beta_{s, v_2}+\ldots+\beta_{s, v_s}} }{1+ e^{\beta_{s, v}+\beta_{s, v_2}+\ldots+\beta_{s, v_s}}}. 
\end{align}

Since $d_s(v)$ is the sum of $O(n^{s-1})$ independent random variables \revsag{bounded by $1$}, by Hoeffding's inequality and the union bound, 
\begin{align*} 
\P\left( \| \grad \ell_{n, s}(\bm \beta_s) \|_\infty^2 \geq 4\,C_{s,M} n^{s-1} \log  n \right)  \le \frac{1}{n^2},
\end{align*} 
for some constant $C_{s,M}>0$ (depending on $s$ and $M$). This establishes the second bound in \eqref{eq:gradientloglikelihood}. 

Next, we prove the first bound in \eqref{eq:gradientloglikelihood}. Denote by $\mathbb{B}^n := \{ \bm x \in \R^n : \| \bm x \|_2 \leq 1\}$ the unit ball in $\mathbb{R}^n$.  By \cite[Lemma 5.2]{vershynin2010introduction}, we can construct an $\frac{1}{2}$-net $\mathcal{V}$ of $\mathbb{B}^n$ satisfying $\log|\mathcal{V}| \le C_1 n$ for some constant $C_1 >0$. \revmark{By a $\frac{1}{2}$-net of $\mathbb{B}^n$, we refer to a set $\mathcal{V}$, such that for all $u \in \mathbb{B}^n$ there exists $v \in \mathcal{V}$ satisfying $\|u-v\|_2 \le \frac{1}{2}$. }Now, for any unit vector $\bm a = (a_1, a_2, \ldots, a_n)^\top \in \B^n$ and the corresponding point $\bm b = (b_1, b_2, \ldots, b_n)^\top \in \mathcal{V}$, recalling \eqref{eq:likelihoodgrad} gives, 
\begin{align}
\label{eq:likelihoodHnsbeta}
\sum_{v=1}^{n} a_v \grad \ell_{n, s}(\bm \beta_s)_{v}  = \sum_{v=1}^{n} a_v \left(\E[d_s(v)] - d_s(v)\right) &= \sum_{v=1}^{n} b_v \left(\E[d_s(v)] - d_s(v)\right) + \Delta,
\end{align}
where 
\begin{align}
\label{eq:bound_delta_net_half}
\Delta & :=  \sum_{v=1}^{n} (a_v - b_v)\left(\E[d_s(v)] - d_s(v)\right) \nonumber \\  
& \leq \sqrt{\sum_{v=1}^{n} (a_v - b_v)^2 \sum_{v=1}^n \left(\E[d_s(v)] - d_s(v)\right)^2 }  \nonumber \\  
& \leq \frac{1}{2}\sqrt{ \sum_{v=1}^n \left(\E[d_s(v)] - d_s(v)\right)^2 } = \frac{1}{2} \| \grad \ell_{n, s}(\bm \beta_s) \|_2 , 
\end{align} 
by the Cauchy-Schwarz inequality and the fact that \revsag{$\| \bm a - \bm b \|_2 \leq \frac{1}{2}$}. Using the above in \eqref{eq:likelihoodHnsbeta} gives, 
\begin{align}\label{eq:likelihoodHnsdelta}
\sum_{v=1}^{n} a_v \grad \ell_{n, s}(\bm \beta_s)_{v} & \leq \sum_{v=1}^{n} b_v \left(\E[d_s(v)] - d_s(v)\right) + \frac{1}{2} \| \grad \ell_{n, s}(\bm \beta_s) \|_2 . 
\end{align} 
Maximizing over $\bm a \in \B^n$ and $\bm b \in \cV$ on both sides of \eqref{eq:likelihoodHnsdelta} and rearranging the terms gives, 
\begin{align}\label{eq:gradientbetaHns}
\| \grad \ell_{n, s}(\bm \beta_s) \|_2 \leq 2 \max_{\bm b \in \mathcal V}\sum_{v=1}^{n} b_v \left(\E[d_s(v)] - d_s(v)\right) . 
\end{align}
For \revsag{$\bm e = (u_1, u_2, \ldots, u_s) \in {[n] \choose s}$} denote $\bm \beta_{s, \bm e} = (\beta_{s, u_1}, \beta_{s, u_2}, \ldots, \beta_{s, u_s})^\top$.  Hence, by \eqref{eq:gradientbetaHns}, Hoeffding's inequality, and union bound,  
\begin{align} 
& \P\left ( \| \grad \ell_{n, s}(\bm \beta_s) \|_2^2 >  4 C^2 n^s \right) \nonumber \\ 
& \leq \sum_{\bm b \in \cV} \P\left ( \sum_{v=1}^{n} b_v \left(\E[d_s(v)] - d_s(v)\right) >  2 C n^{\frac{s}{2}} \right) \nonumber \\ 
& = \sum_{\bm b \in \cV} \P \left( \sum_{v=1}^{n} \sum_{\bm e \in {[n] \choose s} : v \in \bm e} b_v \left\{\frac{e^{\bm \beta_{s, \bm e}^\top\bm 1}}{1+e^{\bm \beta_{s, \bm e}^\top\bm 1}} - \mathbbm 1 \{ \bm e \in E(H_n)\} \right\}  > 2 C n^{\frac{s}{2}} \right) \nonumber \\ 
& \leq  \sum_{\bm b \in \cV} e^{-\frac{ 2 C^2 n}{\sum_{v=1}^n b_v^2}} \leq 2^{C_1 n} e^{-2 C^2 n } \rightarrow 0 , \nonumber 
\end{align} 
by choosing $C > C_1$ to be large enough. This proves the first inequality in \eqref{eq:gradientloglikelihood}. \hfill $\Box$

\subsubsection{Proof of Lemma \ref{lem:lower_bound_hess}} 
\label{sec:hessianHpf} 

For $\bm e = (u_1, u_2, \ldots, u_s) \in {[n] \choose s}$ and $\bm \beta = (\beta_1, \beta_2, \ldots, \beta_n) \in \R^n$, denote $\bm \beta_{\bm e} = (\beta_{u_1}, \beta_{u_2}, \ldots, \beta_{u_s})^\top$. Recalling \eqref{eq:def_k_unif} note that, 
the Hessian matrix $\grad^2 \ell_{n, s}$ can be expressed as: 
\begin{align} 
\grad^2 \ell_{n, s}(\bm \beta)  = & \sum_{u, v \in [n]}\sum_{\bm e \in {[n] \choose s}}   \frac{ e^{\bm \beta^\top_{\bm e}\bm 1} }{(1+ e^{\bm \beta^\top_{\bm e} \bm 1})^2 } \bm \eta_u \bm \eta_v^\top \mathbbm{1}\{ u, v \in \bm e \} , 
\nonumber 
\end{align} 
where $\bm \eta_u$ is the $u$-th basis vector in $\R^n$, for $1 \leq u \leq n$.

Note that for $\bm \beta \in \R^n$, \revsag{since $\bm \beta_s \in \cB_M$, we have}, 
$$
{
| \bm 1^\top \bm \beta_{\bm e} | \leq s \| \bm \beta \|_\infty \leq s \| \bm \beta_s \|_\infty + s \| \bm \beta_s - \bm \beta \|_\infty.}
$$
Hence, 
\begin{align}\label{eq:exponentialH}
\frac{1}{4}e^{-s (M + \| \bm \beta_s - \bm \beta \|_\infty ) } \leq \frac{e^{\bm 1^\top \bm \beta_{\bm e}}}{(1+e^{\bm 1^\top \bm \beta_{\bm e}})^2} \leq 1  . 
\end{align}
For $\bm x \in \R^n$, consider
\begin{align*}
\bm x^\top \grad^2 \ell_{n, s}(\bm \beta) \bm x  &= \sum_{u, v \in [n]}\sum_{\bm e \in {[n] \choose s}}   \frac{ e^{\bm \beta^\top_{\bm e}\bm 1} }{(1+ e^{\bm \beta^\top_{\bm e} \bm 1})^2 } x_u x_v\mathbbm 1\{ u, v \in \bm e \} 
\nonumber\\
& = \sum_{\bm e \in {[n] \choose s}} \frac{ e^{\bm \beta^\top_{\bm e}\bm 1} }{(1+ e^{\bm \beta^\top_{\bm e} \bm 1})^2 } \left(  \sum_{u, v \in [n] }x_ux_v \mathbbm 1\{ u, v \in \bm e \}\right) \nonumber\\
& = \sum_{\bm e \in {[n] \choose s}} \frac{ e^{\bm \beta^\top_{\bm e}\bm 1} }{(1+ e^{\bm \beta^\top_{\bm e} \bm 1})^2 } \left(\sum_{ u \in [n] } x_u \mathbbm 1\{ u \in \bm e \}\right)^2 \nonumber\\
& \ge \frac{1}{4}e^{-s (M + \| \bm \beta_s - \bm \beta \|_\infty ) }\sum_{\bm e \in {[n] \choose s}} \left(\sum_{ u \in [n] } x_u \mathbbm 1\{ u \in \bm e \}\right)^2 , 
\end{align*} 
where the last step uses \eqref{eq:exponentialH}.  Observe that for any $\bm x \in \R^n$ 
\[
\sum_{\bm e \in {[n] \choose s}} \left(\sum_{ u \in [n] } x_u \mathbbm 1\{ u \in \bm e \}\right)^2 = \bm x^\top\bm L\bm x,
\]
where 
$$\bm L :=  \sum_{u, v \in [n]} \sum_{\bm e \in {[n] \choose s}} \bm \eta_u \bm \eta_v^\top \mathbbm 1\{ u, v \in \bm e \} = \left( {n-1 \choose s-1} - {n-2 \choose s-2} \right) \bm I _n+ {n-2 \choose s-2} \bm 1 \bm 1^\top ,$$ 
where $\bm I_n$ is the $n \times n$ identity matrix and $\bm 1 = (1, 1, \ldots, 1)^\top$.  Similarly, we can show from \eqref{eq:exponentialH} that for any $\bm x \in \R^n$ 
\[
\bm x^\top \grad^2 \ell_{n, s}(\bm \beta) \bm x \le \bm x^\top \bm L \bm x.
\]
Thus, for $\bm \beta \in \R^n$ 
\begin{align}\label{eq:matrixLlambda}
\frac{1}{4}e^{-s (M + \| \bm \beta_s - \bm \beta \|_\infty ) }  \lambda_{\min}\left( \bm L \right) \leq \lambda_{\min}(\nabla^2 \ell_{n, s}(\bm{\beta}))  \leq \lambda_{\max}(\nabla^2 \ell_{n, s}(\bm{\beta}))  \leq  \lambda_{\max}\left( \bm L \right) . 
\end{align} 
Note that $\bm L$ is a circulant matrix with 2 non-zero eigenvalues: 
$${n-1 \choose s-1} \quad \text{ and } \quad {n-1 \choose s-1}-{n-2 \choose s-2}.$$ 
\revsag{Further}, there exists constants $C_1'', C_2''>0$ (depending on $r$), such that
\[
{n-1 \choose s-1} \le C_1'' n^{s-1} \quad \mbox{and} \quad {n-1 \choose s-1}-{n-2 \choose s-2} \ge C_2'' \,n^{s-1} . 
\] 
This implies, from \eqref{eq:matrixLlambda}, that there exists constants $C_1', C_2'>0$ (depending on $r$ and $M$) such that \eqref{eq:hessainlambdal2} hold. 
The result in \eqref{eq:hessainlambda} from hold from \eqref{eq:hessainlambdal2} by noting that  $\| \bm \beta_s - \bm \beta \|_\infty \leq \| \bm \beta_s - \bm \beta \|_2$. 

\subsection{Convergence Rate in the $L_\infty$ Norm}  
\label{sec:layermleLinftypf} 

Suppose $H_n \sim \mathsf{H}_{n, [r]}(n, \bm B)$ \revsag{be} as in the statement of Theorem \ref{thm:layermle}. From the arguments in Appendix \ref{sec:layermleL2pf} we know that, with probability $1-o(1)$, the ML equations \eqref{eq:def_k_unif} have a bounded solution \revsag{$\hat{\bm B}=(\hat{\bm \beta}_2, \ldots, \hat{\bm \beta}_r)$}, that is, $\grad \ell_{n, s} (\hat{\bm \beta}_s) = 0$, for $2 \leq s \leq r$, and $\max_{2 \leq s \leq r} \| \hat{\bm \beta}_s \|_\infty = O(1)$.  To establish the rate in $L_\infty$ norm we decompose the likelihood for the $s$-th layer as follows.
\begin{align}\label{eq:likelihoodHbeta} 
\ell_{n, s}(\bm \beta) & =  \sum_{ \{v_{1}, v_{2}, \ldots, v_{s}\} \in {[n] \choose s} }\log \left(1+ e^{\beta_{v_1}+\ldots+\beta_{v_s} } \right) - \sum_{v=1}^{n}\beta_{v} d_s(v) \nonumber \\ 
& = \sum_{\bm e \in {[n] \choose s} } \left\{  \log\left(1+ e^{ \bm \beta^\top_e \bm 1 } \right)  - \mathbbm 1\{ \bm e \in E(H_n)\}  \bm \beta^\top_e \bm 1 \right\} \nonumber \\ 
& = \ell^{+}_{n, s}(\beta_u|\bm \beta_{\bar{u}}) + \ell^{-}_{n, s}(\bm \beta_{\bar{u}}) ,
\end{align} 
where $\bm \beta_{\bar{u}} = (\beta_1, \ldots, \beta_{u-1}, \beta_{u+1}, \ldots, \beta_n)$, 
\begin{align} 
\label{eq:likelihoodHbeta_split} 
\ell^{+}_{n, s}(\beta_u|\bm \beta_{\bar{u}}) & :=  \sum_{\bm e \in {[n] \choose s} : u \in \bm e} \left\{  \log\left(1+ e^{ \bm \beta^\top_e \bm 1 } \right)  -\mathbbm 1 \{ \bm e \in E(H_n)\}  \bm \beta^\top_e \bm 1 \right\} \nonumber \\ 
\ell^{-}_{n, s}(\bm \beta_{\bar{u}}) & :=  \sum_{\bm e \in {[n] \choose s} : u \notin \bm e} \left\{ \log\left(1+ e^{ \bm \beta^\top_e \bm 1 } \right)  - \mathbbm 1 \{ \bm e \in E(H_n)\}  \bm \beta^\top_e \bm 1 \right\} . 
\end{align} 
Fix a constant $K > 0$ and define
\begin{align}\label{eq:beta_minus_u}
\hat{\bm \beta}^{\circ}_{s,\bar{u}} = \argmin_{\bm \beta_{\bar{u}}: \|\bm \beta_{\bar{u}}-\bm \beta_{s,\bar{u}}\|_2 \le K}\ell^{-}_{n, s}(\bm \beta_{\bar{u}}) , 
\end{align} 
where $\bm \beta_{s, \bar{u}} = (\beta_{s, 1}, \ldots, \beta_{s, u-1}, \beta_{s, u+1}, \ldots, \beta_{s, n})$.  This is the leave-one-out ML estimate on the constrained set $\|\bm \beta_{\bar{u}}-\bm \beta_{s,\bar{u}}\|_2 \le K$. First we bound the difference  (in $L_2$ norm) of  constrained leave-one-out ML estimate defined above and the leave-one-out true parameter $\bm \beta_{s, \bar{u}} $. 

\begin{lem}\label{lm:Kbeta} Let $\hat{\bm \beta}^{\circ}_{s,\bar{u}}$ and $\bm \beta_{s,\bar{u}}$ be as defined above. Then, for $u \in [n]$, with probability $1 - o(1)$, 
\begin{align}\label{eq:Kbeta}
   \max_{u \in [n]}\, \|\hat{\bm \beta}^{\circ}_{s,\bar{u}}-\bm \beta_{s,\bar{u}}\|_2^2 \lesssim _{s,M,K}  \frac{1}{n^{s-2}} . 
\end{align} 
\end{lem} 

\begin{proof}
To begin with, observe that 
\begin{align}
    \ell^{-}_{n, s}(\bm \beta_{s,\bar{u}}) & \ge \ell^{-}_{n, s}(\hat{\bm \beta}^{\circ}_{s,\bar{u}}) \nonumber \\
    & = \ell^{-}_{n, s}(\bm \beta_{s,\bar{u}}) + (\hat{\bm \beta}^{\circ}_{s,\bar{u}}-\bm \beta_{s,\bar{u}})^\top\nabla \ell^{-}_{n, s}(\bm \beta_{s,\bar{u}}) + \frac{1}{2}(\hat{\bm \beta}^{\circ}_{s,\bar{u}}-\bm \beta_{s,\bar{u}})^\top\nabla^2 \ell^{-}_{n, s}(\tilde{\bm \beta})(\hat{\bm \beta}^{\circ}_{s,\bar{u}}-\bm \beta_{s,\bar{u}}), \nonumber 
\end{align}
where $\|\tilde{\bm \beta} - \bm \beta_{s,\bar{u}}\|_2 \le \|\hat{\bm \beta}^{\circ}_{s,\bar{u}}-\bm \beta_{s,\bar{u}}\|_2 \leq K$. This implies, 
\begin{align}
\label{eq:likeh_del_one}
 \| \hat{\bm \beta}^{\circ}_{s,\bar{u}}-\bm \beta_{s,\bar{u}} \|_2 \cdot \|\nabla \ell^{-}_{n, s}(\bm \beta_{s,\bar{u}})\|_2 & \geq  -  (\hat{\bm \beta}^{\circ}_{s,\bar{u}}-\bm \beta_{s,\bar{u}})^\top\nabla \ell^{-}_{n, s}(\bm \beta_{s,\bar{u}}) \nonumber \\ 
 & \ge  \frac{1}{2}(\hat{\bm \beta}^{\circ}_{s,\bar{u}}-\bm \beta_{s,\bar{u}})^\top\nabla^2 \ell^{-}_{n, s}(\tilde{\bm \beta})(\hat{\bm \beta}^{\circ}_{s,\bar{u}}-\bm \beta_{s,\bar{u}}) . 
\end{align} 
By Lemma \ref{lem:lower_bound_hess}, 
\[
(\hat{\bm \beta}^{\circ}_{s,\bar{u}}-\bm \beta_{s,\bar{u}})^\top\nabla^2 \ell^{-}_{n, s}(\tilde{\bm \beta})(\hat{\bm \beta}^{\circ}_{s,\bar{u}}-\bm \beta_{s,\bar{u}}) \gtrsim_{ s,M,K } 
\|\hat{\bm \beta}^{\circ}_{s,\bar{u}}-\bm \beta_{s,\bar{u}}\|^2 n^{s-1} . 
\]
Also, by Lemma \ref{lem:gradientH}, $\|\nabla \ell^{-}_{n, s}(\bm \beta_{s,\bar{u}})\|^2_2 \lesssim _{ s,M,K }  n^s$ with probability $1 - O(\frac{1}{n^2})$. Plugging in the above inequalities in \eqref{eq:likeh_del_one},  and using the union bound we get \eqref{eq:Kbeta}. 
\end{proof}

Next, we bound the difference between the  constrained leave-one-out ML estimate $\hat{\bm \beta}^{\circ}_{s,\bar{u}}$ and the (unconstrained) leave-one-out ML estimate $\hat{\bm \beta}_{s, \bar{u}} = (\hat \beta_{s, 1}, \ldots, \hat \beta_{s, u-1}, \hat \beta_{s, u+1}, \ldots, \hat \beta_{s, n})$. 

\begin{lem}\label{lm:leave1mlel2} Let $\hat{\bm \beta}^{\circ}_{s,\bar{u}}$ and $\hat{\bm \beta}_{s,\bar{u}}$ be as defined above. Then, with probability $1-o(1)$,  
\begin{align}\label{eq:leave1max2}
    \max_{u \in [n]}\|\hat{\bm \beta}^{\circ}_{s,\bar{u}}-\hat{\bm \beta}_{s, \bar{u}}\|^2_2 &\lesssim _{ s,M,K }  \frac{1}{n} 
   + \frac{\|\hat{\bm \beta}_s-\bm \beta_s\|^2_\infty}{n^{s-1}} .
\end{align} 
\end{lem} 

\begin{proof}
\revsag{[Please read this proof a bit....there may be some mistakes....I made some major changes.]}
By the definition of $\hat{\bm \beta}^{\circ}_{s,\bar{u}}$ (recall \eqref{eq:beta_minus_u}) 
\begin{align*}
    \ell^{-}_{n, s}(\hat{\bm \beta}_{s,\bar{u}}) & \ge \ell^{-}_{n, s}(\hat{\bm \beta}^{\circ}_{s,\bar{u}}) \nonumber \\ 
    & = \ell^{-}_{n, s}(\hat{\bm \beta}_{s, \bar{u}}) + (\hat{\bm \beta}^{\circ}_{s,\bar{u}}-\hat{\bm \beta}_{s, \bar{u}})^\top\nabla \ell^{-}_{n, s}(\hat{\bm \beta}_{s,\bar{u}}) + (\hat{\bm \beta}^{\circ}_{s,\bar{u}}-\hat{\bm \beta}_{s, \bar{u}})^\top\nabla \ell^{-}_{n, s}(\bar{\bm \beta})(\hat{\bm \beta}^{\circ}_{s,\bar{u}}-\hat{\bm \beta}_{s,\bar{u}}),
\end{align*}
where $\|\bar{\bm \beta}-\hat{\bm \beta}_{s,\bar{u}}\|_2 \le \|\hat{\bm \beta}^{\circ}_{s,\bar{u}}-\hat{\bm \beta}_{s,\bar{u}}\|_2$. Note that $\|\hat{\bm \beta}^{\circ}_{s,\bar{u}}-\hat{\bm \beta}_{s,\bar{u}}\|_2  = O(1)$, since $\| \hat{\bm \beta}_s \|_\infty = O(1)$ and $\|\hat{\bm \beta}^{\circ}_{s,\bar{u}}\| = O(1)$. \revsag{Therefore, we have $K>0$ such that} by Lemma \ref{lem:lower_bound_hess}, 
\begin{align}\label{eq:betaumlel2}
    \|\hat{\bm \beta}^{\circ}_{s, \bar{u}}-\hat{\bm \beta}_{s, \bar{u}}\|^2_2 \lesssim _{s,M,K}  \frac{\|\nabla \ell^{-}_{n, s}(\hat{\bm \beta}_{s,\bar{u}})\|^2_2 }{n^{2(s-1)}}. 
\end{align} 
Since $\nabla \ell_{n, s}(\hat{\bm \beta}_s)=0$, that is, $\frac{\partial}{\partial \beta_v} \ell_{n,s}(\hat{\bm \beta}_s)=0$, for $v \in [n]$. Hence, we have from \eqref{eq:likelihoodHbeta}, 
\begin{align*}
    \frac{\partial}{\partial \beta_v} \ell^{-}_{n, s}(\hat{\bm \beta}_{s, \bar{u}})  = - \frac{\partial}{\partial \beta_v}  \ell^{+}_{n, s}(\hat{\beta}_{s, u}|\hat{\bm \beta}_{s,\bar{u}}) & = - \sum_{\bm e \in {[n] \choose s}: \{u, v\} \in \bm e} \{ X_{\bm e} - \psi( \bm 1^\top \hat{\bm \beta}_{s, \bm e})\} ,
\end{align*} 
where $\psi(x):= \frac{e^x}{1+e^x}$ and \revsag{$X_{\bm e}=\mathbbm 1\{\bm e \in E(H_n)\}$ for $\bm e \in {[n] \choose s}$}. 
This implies, 
\begin{align}\label{eq:leave1mlel2}
    & \|\nabla \ell^{-}_{n, s}(\hat{\bm \beta}_{s,\bar{u}})\|^2_2 \nonumber \\ 
    & = \sum_{v \in [n]\backslash\{ u \}}\left(\sum_{\bm e \in {[n] \choose s}: \{u, v\} \in \bm e} \{X_{\bm e} - \psi(\hat{\bm \beta}^\top_{s, \bm e}1)\}\right)^2 \nonumber \\
    & \lesssim  \sum_{v \in [n]\backslash\{ u \}} \left[ \left(\sum_{\bm e \in {[n] \choose s}: \{u, v\} \in \bm e} \{X_{\bm e} - \psi( \bm 1^\top\bm \beta_{s, \bm e})\}\right)^2  + \left(\sum_{\bm e \in {[n] \choose s}: \{u, v\} \in \bm e} \{\psi(\bm 1^\top \hat{\bm \beta}_{s, \bm e}) - \psi( \bm 1^\top \bm \beta_{s, \bm e})\}\right)^2 \right] \nonumber \\
    & \lesssim_{r} \sum_{v \in [n]\backslash\{ u \}}\left(\sum_{\bm e \in {[n] \choose s}: \{u, v\} \in \bm e} \{X_{\bm e} - \psi(\bm 1^\top\bm \beta_{s, \bm e})\}\right)^2 + \revsag{n^{2s-3}} \|\hat{\bm \beta}_s-\bm \beta_s\|^2_\infty  ,  
\end{align} 
using 
$$|\psi(\bm 1^\top \hat{\bm \beta}_{s, \bm e}) - \psi(\bm 1^\top \bm \beta_{s, \bm e})| \lesssim  |\bm 1^\top \hat{\bm \beta}_{s, \bm e} - \bm 1^\top \bm \beta_{s, \bm e} | \lesssim_r \|\hat{\bm \beta}_s-\bm \beta_s\|^2_\infty.$$ 

By \eqref{eq:betaumlel2} and \eqref{eq:leave1mlel2}, to prove the result in \eqref{eq:leave1max2}  it suffices show the following holds with probability $1-o(1)$, 
\begin{align}\label{eq:leave1mle}
\max_{1 \leq u \leq n}\sum_{v \in [n]\backslash\{u\} } \left(\sum_{\bm e \in {[n] \choose s}: u, v \in \bm e } \left\{ \mathbbm 1\{ \bm e \in E(H_n)\} -  \psi(\bm \beta_{s, \bm e}^\top \bm 1) \right\}  \right)^2 \lesssim  n^{s-1} . 
\end{align} 
This is proved in Appendix \ref{sec:leave1mlepf}. 
\end{proof}

We now apply the above lemmas to derive the bound in the $L_\infty$ norm. To begin with note that since $\ell_{n,s}(\hat{\bm\beta}_s)=\min_{\bm \beta_s}\ell_{n,s}(\bm \beta_s)$, 
\[
{\ell}^{+}_{n, s}(\beta_{s, u}|\hat{\bm \beta}_{s, \bar u}) + {\ell}^{-}_{n, s}(\hat {\bm \beta}_{s, \bar u}) \ge \ell_{n,s}(\hat{\bm \beta}_s) = {\ell}^{+}_{n, s}(\hat \beta_{s, u}|\hat{\bm \beta}_{s, \bar u}) + {\ell}^{-}_{n, s}(\hat {\bm \beta}_{s, \bar u}) . 
\]

The above inequality implies 
\begin{align}
    {\ell}^{+}_{n, s}& (\beta_{s, u}|\hat{\bm \beta}_{s, \bar u}) \nonumber \\ 
    & \ge {\ell}^{+}_{n, s}(\hat{\beta}_{s, u}|\hat{\bm \beta}_{s, \bar u}) \nonumber \\ 
    & = {\ell}^{+}_{n, s}(\beta_{s, u}|\hat{\bm \beta}_{s, \bar u}) + (\hat{\beta}_{s, u}-\beta_{s, u})\frac{\partial}{\partial \beta_u}{\ell}^{+}_{n, s}(\beta_{s, u}|\hat{\bm \beta}_{s, \bar u})  + \frac{1}{2}(\hat{\beta}_{s, u}-\beta_{s, u})^2\frac{\partial^2}{\partial \beta_u^2}{\ell}^{+}_{n, s}( \tilde \beta |\hat{\bm \beta}_{s, \bar u}) , \nonumber 
\end{align} 
where $\tilde \beta$ is a convex combination of $\hat{\beta}_{s, u}$ and $\beta_{s, u}$. Therefore,  
\begin{align}\label{eq:l2linfty} 
(\hat{\beta}_{s, u}-\beta_{s, u})^2 \leq \frac{4 |\frac{\partial}{\partial \beta_u}{\ell}^{+}_{n, s}(\beta_{s, u} | \hat{\bm \beta}_{s, \bar{u}})|^2}{|\frac{\partial^2}{\partial \beta_u^2}{\ell}^{+}_{n, s}(\tilde \beta|\hat{\bm \beta}_{s, \bar{u}})|^2 } . 
\end{align} 
From arguments in Appendix \ref{sec:layermleL2pf} we know that with probability $1-o(1)$, $\|\hat{\bm \beta}_s-\bm \beta_{s}\|_\infty \leq \|\hat{\bm \beta}_s-\bm \beta_{s}\|_2 \lesssim  1$. Note that for $\bm \beta \in \R^n$ such that $\| \bm \beta - \bm \beta_s \|_\infty \lesssim  1$, we have $\| \bm \beta \|_\infty \lesssim  1$ and hence, $| \bm 1^\top \bm \beta_{\bm e} | \lesssim  1$. 
This implies, $\psi(\bm 1^\top \bm \beta_{\bm e, s})(1-\psi(\bm 1^\top \bm \beta_{\bm e, s})) \gtrsim 1$ and hence,  
\[
\frac{\partial^2}{\partial \beta_u^2}{\ell}^{+}_{n, s}(\tilde \beta|\hat{\bm \beta}_{s, \bar{u}}) = \sum_{\bm e \in {[n] \choose s}: u \in \bm e} \psi(\bm 1^\top \bar{\bm \beta}_{\bm e, s})(1-\psi(\bm 1^\top \bar{\bm \beta}_{\bm e, s})) \gtrsim  n^{s-1} ,  
\] 
where $\bar{\bm \beta}_{s} = (\hat \beta_{s, 1}, \ldots, \hat \beta_{s, u-1}, \beta_{s, u}, \hat \beta_{s, u+1}, \ldots, \hat \beta_{s, n})^\top$. Hence, \eqref{eq:l2linfty} implies, 
\begin{align}
\label{eq:one_gradient}
    (\hat{\beta}_{s, u}-\beta_{s, u})^2 \lesssim  \frac{|\frac{\partial}{\partial \beta_u}{\ell}^{+}_{n, s}(\beta_{s, u}|\hat{\bm \beta}_{s, \bar u})|^2}{n^{2s-2}}.
\end{align} 
Now, we bound $|\frac{\partial}{\partial \beta_u}{\ell}^{+}_{n, s}(\beta_{s, u}|\hat{\bm \beta}_{s, \bar u})|^2$.  For this define $$\bar{\bm \beta}_{s}^\circ = ([\hat {\bm \beta}_{s, \bar u}^\circ]_1, \ldots, [\hat {\bm \beta}_{s, \bar u}^\circ]_{u-1}, \beta_{s, u}, [\hat {\bm \beta}^\circ_{s, \bar u}]_{u+1}, \ldots, [\hat{\bm \beta}_{s, \bar u}^\circ]_n)^\top.$$ 
Then we have
\begin{align}\label{eq:gradientln1}
    \left|\frac{\partial}{\partial \beta_u}{\ell}^{+}_{n, s}(\beta_{s, u}|\hat{\bm \beta}_{s, \bar u})\right| & = \Bigg|\sum_{\bm e \in {[n] \choose s}: u \in \bm e}\{X_{\bm e}-\psi(\bm 1^\top \bar{\bm \beta}_{\bm e, s})\}\Bigg| \leq T_1(u) + T_2(u) + T_3(u), 
    \end{align} 
    where
    \begin{align*}
    T_1(u) &:=\Bigg|\sum_{\bm e \in {[n] \choose s}: u \in \bm e}\{X_{\bm e}-\psi(\bm 1^\top\bm \beta_{\bm e, s})\}\Bigg|, \quad    T_2(u) := \Bigg|\sum_{\bm e \in {[n] \choose s}: u \in \bm e}\{\psi(\bm 1^\top\bar{\bm \beta}^\circ_{\bm e, s})-\psi(\bm 1^\top\bm \beta_{\bm e, s})\}\Bigg| , 
    \end{align*}
    and 
\begin{align*}
    T_3(u) &:= \Bigg|\sum_{\bm e \in {[n] \choose s}: u \in \bm e}\{\psi(\bm 1^\top\bar{\bm \beta}_{\bm e, s}^\circ)-\psi(\bm 1^\top\bar{\bm \beta}_{\bm e, s})\}\Bigg|.
\end{align*} 
Note that since $\{X_{\bm e}\}_{\bm e \in {[n] \choose s}}$ are independent and bounded random variables, using Hoeffding's inequality and union bound gives $$\max_{u \in [n]} T_1(u) \lesssim  \sqrt{n^{s-1} \log n } , $$ 
with probability $1-o(1)$.  Next, we consider $T_2(u)$. By Lemma \ref{lm:Kbeta}, with probability $1-o(1)$, 
\begin{align*}
   \max_{u \in [n]} \,T_2(u) \lesssim \max_{u \in [n]} \, \sum_{\bm e \in {[n] \choose s}:u \in \bm e}\Bigg\{\sum_{v \in \bm e}|\beta_{s, v}-[\hat{\bm \beta}^{\circ}_{s,\bar u}]_v|\Bigg\} & =\max_{u \in [n]} \, \sum_{ v \in [n]\backslash\{u\} } n^{s-2}|\beta_{s, v} - [\hat{\bm \beta}^{\circ}_{s,\bar u}]_v | \nonumber \\ 
    & \lesssim n^{s-\frac{3}{2}}\max_{u \in [n]}  \,\|\bm \beta_{s, \bar u} - \hat{\bm \beta}^{\circ}_{s, \bar u}\|_2 \lesssim  \sqrt{n^{s - 1}} .  
\end{align*}   
  A similar argument shows that, with probability $1-o(1)$, $\max_{u \in [n]} \,T_3(u) \lesssim n^{s-\frac{3}{2}} \|\hat{\bm \beta}_{s, \bar u}- \hat{\bm \beta}^{\circ}_{s, \bar u}\|_2$. 
Combining the bounds on $T_1$, $T_2$ and $T_3$ with \eqref{eq:one_gradient} and \eqref{eq:gradientln1} gives, with probability $1-o(1)$, 
\begin{align} \label{eq:infinity_norm_bound}
    \|\hat{\bm \beta}_{s}-\bm \beta_s\|_\infty & \lesssim   \sqrt{\frac{\log n}{n^{s - 1}}} + \frac{\max_{u \in [n]} \,\|\hat{\bm \beta}_{s, \bar u}- \hat{\bm \beta}^{\circ}_{s, \bar u}\|_2}{\sqrt n} . 
    \end{align} 
Applying \eqref{eq:infinity_norm_bound} in \eqref{eq:leave1max2} now gives,  with probability $1-o(1)$, 
    \begin{align*}
    \max_{u \in [n]}\|\hat{\bm \beta}^{\circ}_{s,\bar{u}}-\hat{\bm \beta}_{s, \bar{u}}\|_2 & \lesssim _{s,M,K}  \sqrt{\frac{1}{n^{s-1}} } 
    + \frac{\|\hat{\bm \beta}_s-\bm \beta_s\|_\infty}{\revsag{\sqrt{n}}} \nonumber \\ 
    & \lesssim_{s,M,K}  \sqrt{\frac{1}{n^{s-1}} } + \frac{ \max_{u \in [n]}\|\hat{\bm \beta}^{\circ}_{s,\bar{u}}-\hat{\bm \beta}_{s, \bar{u}}\|_2}{\revsag{n}} \lesssim \sqrt{\frac{1}{n^{s-1}} } . 
\end{align*} 
Using this inequality with \eqref{eq:infinity_norm_bound} gives, with probability $1-o(1)$, 
\[
\|\hat{\bm \beta}_s - \bm \beta_s \|_\infty \lesssim_{s,M,K} \sqrt{\frac{\log n}{n^{s-1}}} , 
\] 
establishing the desired bound in \eqref{eq:mlerlayer}.

\subsubsection{Proof of \eqref{eq:leave1mle}} 
\label{sec:leave1mlepf}

\begin{proof} 
Denote by $\mathbb{B}^{n-1}=\{\bm x  \in \mathbb R^{n-1}: \|\bm x\|_2 \le 1\}$. Using \cite[Lemma 5.2]{vershynin2010introduction}, we can construct an $\frac{1}{2}$-net $\mathcal{V}_1$ of $\mathbb{B}^{n-1}$ (defined in Section \ref{sec:gradientHpf}) satisfying $\log|\mathcal{V}_1| \le C_2 n$ for some constant $C_2 >0$. Now, for any $u \in [n]$, any unit vector $\tilde{\bm a} = (\tilde{a}_1, \ldots, \tilde{a}_{u-1}, \tilde{a}_{u+1}, \ldots, \tilde{a}_{n})^\top \in \B^{n-1}$ and the corresponding point $\tilde{\bm b} = (\tilde{b}_1, \ldots, \tilde{b}_{u-1}, \tilde{b}_{u+1}, \ldots, \tilde{b}_{n})^\top \in \mathcal{V}_1$, 
\begin{align}
\label{eq:delte_u_net_half}
& \sum_{v \in [n]\backslash\{ u \}}\tilde{a}_v\left\{\sum_{\bm e \in {[n] \choose s}: u, v \in \bm e }\left\{X_{\bm e}-\frac{e^{\bm 1^\top \bm \beta_{\bm e}}}{1+e^{\bm 1^\top \bm \beta_{\bm e}}}\right\}\right\} \nonumber \\ 
& = \sum_{v \in [n]\backslash\{ u \}}\tilde{b}_v\left\{\sum_{ \bm e \in {[n] \choose s}, u, v \in \bm e }\left\{X_{\bm e}-\frac{e^{\bm 1^\top \bm \beta_{\bm e}}}{1+e^{\bm 1^\top \bm \beta_{\bm e}}}\right\}\right\}+\Delta_{u},
\end{align}
where
\[
\Delta_{u} := \sum_{v \in [n]\backslash \{u\} }(\tilde{a}_v-\tilde{b}_v)\left\{\sum_{ \bm e \in {[n] \choose s} : u, v \in \bm e }\left\{X_{\bm e}-\frac{e^{\bm 1^\top \bm \beta_{\bm e}}}{1+e^{\bm 1^\top \bm \beta_{\bm e}}}\right\}\right\} . 
\]
Proceeding as in \eqref{eq:bound_delta_net_half}, for all $u \in [n]$, we can show
\[
|\Delta_{u}| \le \frac{1}{2}\sqrt{\sum_{v \in [n]\backslash\{ u \}}\left\{\sum_{\bm e \in {[n] \choose s} : u, v \in \bm e}\left\{X_{\bm e}-\frac{e^{\bm 1^\top \bm \beta_{\bm e}}}{1+e^{\bm 1^\top \bm \beta_{\bm e}}}\right\}\right\}^2}.
\]
Maximizing over $\tilde{\bm a} \in \mathbb{B}^{n-1}$ and $\tilde{\bm b} \in \mathcal{V}_1$ on both sides of \eqref{eq:delte_u_net_half} we get
\[
\sqrt{\sum_{v \in [n]\backslash\{ u \}}\left\{\sum_{\bm e \in {[n] \choose s} : u, v \in \bm e}\left\{X_{\bm e}-\frac{e^{\bm 1^\top \bm \beta_{\bm e}}}{1+e^{\bm 1^\top \bm \beta_{\bm e}}}\right\}\right\}^2} \le 2 \max_{\tilde{\bm b} \in \mathcal{V}_1}\sum_{v \in [n]\backslash\{ u \}}\tilde{b}_v\left\{\sum_{\bm e \in {[n] \choose s} : u, v \in \bm e}\left\{X_{\bm e}-\frac{e^{\bm 1^\top \bm \beta_{\bm e}}}{1+e^{\bm 1^\top \bm \beta_{\bm e}}}\right\}\right\}.
\]
As the above relation holds for all $u \in [n]$ we get
\begin{align}
\label{eq:one_delted_max}
& \sqrt{\max_{u \in [n]}\sum_{ v \in [n]\backslash\{ u \} }\left\{\sum_{\bm e \in {[n] \choose s} : u, v \in \bm e}\left\{X_{\bm e}-\frac{e^{\bm 1^\top \bm \beta_{\bm e}}}{1+e^{\bm 1^\top \bm \beta_{\bm e}}}\right\}\right\}^2} \nonumber \\ 
& \le 2\,\max_{u \in [n]}\max_{\tilde{\bm b} \in \mathcal{V}_1}\sum_{ v \in [n]\backslash\{ u \} }\tilde{b}_v\left\{\sum_{\bm e \in {[n] \choose s} : u, v \in \bm e}\left\{X_{\bm e}-\frac{e^{\bm 1^\top \bm \beta_{\bm e}}}{1+e^{\bm 1^\top \bm \beta_{\bm e}}}\right\}\right\} . 
\end{align}
Hence, using \eqref{eq:one_delted_max}, Hoeffding Inequality and union bound we get
\begin{align} 
& \P\left ( \max_{u \in [n]}\sum_{ v \in [n]\backslash\{ u \} }\left\{\sum_{\bm e \in {[n] \choose s} : u, v \in \bm e}\left\{X_{\bm e}-\frac{e^{\bm 1^\top \bm \beta_{\bm e}}}{1+e^{\bm 1^\top \bm \beta_{\bm e}}}\right\}\right\}^2 >  4 K^2 n^{s-1} \right) \nonumber \\ 
& \leq \sum_{u=1}^{n}\sum_{\tilde{\bm b} \in \mathcal{V}_1} \P\left (\sum_{ v \in [n]\backslash\{ u \} }\tilde{b}_v\left\{\sum_{\bm e \in {[n] \choose s} : u, v \in \bm e}\left\{X_{\bm e}-\frac{e^{\bm 1^\top \bm \beta_{\bm e}}}{1+e^{\bm 1^\top \bm \beta_{\bm e}}}\right\}\right\} >  2 K n^{\frac{s-1}{2}} \right) \nonumber \\ 
& \le  \sum_{u=1}^{n}\sum_{\tilde{\bm b} \in \mathcal{V}_1}e^{-\frac{2K^2n}{\sum_{v=1}^{n-1}\tilde{b}^2_v}}\nonumber \\
& \le n\;2^{C_2n}e^{-2K^2n} \rightarrow 0, \nonumber 
\end{align} 
for $K$ large enough. 
\end{proof}

\section{Estimation Lower Bounds: Proof of Theorem \ref{thm:maxl2lb}} 
\label{sec:estimationlbpf}

The lower bound in the $L_2$ norm is proved in Appendix \ref{sec:l2lbpf} and the lower bound in the $L_\infty$ norm is proved Appendix \ref{sec:maxlbpf}.

\subsection{Estimation Lower Bound in the $L_2$ Norm: Proof of \eqref{eq:l2lb}} 
\label{sec:l2lbpf}

For $\bm \gamma \in \R^n$, denote the probability distribution of $s$-uniform model $\mathsf{H}_{s}(n, \bm \gamma)$ by $\P_{\bm \gamma}$. To prove the result \eqref{eq:l2lb} recall \revsag{the following version of  Fano's lemma}: 

\begin{thm}[{\cite[Theorem 2.5]{MR2724359}}] 
Suppose there exists $\bm \gamma^{(0)}, \cdots,\bm  \gamma^{(J)} \in \mathbb R^n$, with $\|\bm \gamma^{(j)}\| \in \cB_M$ for all $0 \leq j \leq J$, such that
\begin{enumerate}
    \item[$(1)$] $\|\bm \gamma^{(j)}-\bm \gamma^{(\ell)}\|_2 \ge 2 s  \,>0$ for all $ 0 \leq j \neq \ell \leq J$, 
    \item[$(2)$] $\frac{1}{J}\sum_{j=1}^{J} \revmark{\mathsf{KL}(\P_{\bm \gamma^{(j)}}\|\,\P_{\bm \gamma^{(0)}})} \le \alpha \log J$,
\end{enumerate}
where $\alpha \in (0, 1/8)$. Then 
\begin{align}\label{eq:fl}
\min_{\hat{\bm \gamma}}\max_{\bm \gamma}\mathbb{P}\left(\|\hat{\bm \gamma}-\bm\gamma\|_2 \ge s \right) \ge \frac{\sqrt{J}}{\sqrt{J}+1}\left(1-2 \alpha -\sqrt{\frac{2 \alpha }{\log J}}\right).
\end{align}
\label{thm:fl}
\end{thm}

To obtain $\bm \gamma^{(0)}, \ldots,\bm  \gamma^{(J)} \in \mathbb R^n$ as in the above lemma we will invoke the Gilbert-Varshamov Theorem (see \cite[Lemma 2.9]{MR2724359}) which states that there exists $\bm \omega^{(0)}, \ldots, \bm \omega^{(J)} \in \{0,1\}^n$, with $J \ge 2^{n/8}$, 
such that $\bm \omega^{(0)}=(0,\cdots,0)^\top $ and
\begin{align}\label{eq:jl}
\|\bm \omega^{(j)} - \bm \omega^{(\ell)}\|_1 \ge \frac{n}{8} , 
\end{align}
for all $0 \le j \neq \ell \le J$. For $\bm \omega^{(0)}, \ldots, \bm \omega^{(J)} \in \{0,1\}^n$ as above and $\delta \in (0, 1/8)$  define, 
$$\bm \gamma^{(j)} = \varepsilon_n \bm \omega^{(j)}, \quad \text{ for } 0 \leq j \leq J , $$ 
where $\varepsilon_n = 16 C n^{-\frac{s-1}{2}}$, with $C= C(\delta, s) \,> 0$ a constant depending on $\delta$ and $s$ to be chosen later. 
By \eqref{eq:jl} we have
\[
\| \bm \gamma^{(j)} - \bm \gamma^{(\ell)} \|_2 \geq 2 C n^{-\frac{s-2}{2}} .
\]  
Now,
\begin{align} 
    & \revsag{\mathsf{KL}(\P_{\bm \gamma^{(j)}}\|\, \P_{\bm \gamma^{(0)}})} \nonumber \\ 
    & = \sum_{t=0}^{s} {\|\bm \omega^{(j)}\|_1 \choose t}{n-\|\bm \omega^{(j)}\|_1 \choose s-t}  \Big\{\psi\left(t \varepsilon_n\right) \log\left(2\psi\left(t \varepsilon_n\right)\right)+\left(1-\psi\left(t \varepsilon_n\right)\right) \log\left(2\left(1-\psi\left(t \varepsilon_n\right)\right)\right)\Big\}, \nonumber 
\end{align} 
where 
 $\psi(x)= \frac{e^x}{1 + e^x}$ is the logistic function defined in Lemma \ref{lm:leave1mlel2}. By  a Taylor expansion, 
 for small enough $x>0$,  
\[
\psi(x)\log(2\psi(x))+(1-\psi(x))\log(2(1-\psi(x))) = \frac{x^2}{8} + O(x^3) .
\]
Hence, using ${\|\bm \omega^{(j)}\|_1 \choose t}{n-\|\bm \omega^{(j)}\|_1 \choose s-t} \lesssim_s n^{s}$ gives \[
 \frac{1}{J}\sum_{j=1}^{J} \revsag{\mathsf{KL}(\P_{\bm \gamma^{(j)}}\| \P_{\bm \gamma^{(0)}})} \lesssim_s n^s \varepsilon^2_n  \lesssim_s C^2 n \le \delta \log J , 
\] 
for $C = C(\delta, s)$ chosen appropriately. 
Hence, applying Theorem \ref{thm:fl} and taking $J \rightarrow \infty$ in \eqref{eq:fl} gives 
$$\min_{\hat{\bm \gamma}}\max_{\bm \gamma}\mathbb{P}\left(\|\hat{\bm \gamma}-\bm\gamma\|_\infty \ge C n^{- \frac{s-2}{2}} \right) \ge 1- 2 \delta . 
$$ 
This completes the proof of \eqref{eq:l2lb}.

\subsection{Estimation Lower Bound in $L_\infty$ Norm: Proof of \eqref{eq:maxlb}} 
\label{sec:maxlbpf}
\revmark{Let us consider $n+1$ points $\bm \gamma^{(0)},\bm \gamma^{(1)},\ldots,\bm \gamma^{(n)} \in \R^n$ defined as follows:
\[
\bm \gamma^{(j)} =\begin{cases}
2\varepsilon \bm{e}_j & \mbox{if $j \in \{1,\ldots,n\}$,}\\
\bm 0 & \mbox{if $j=0$,}
\end{cases}
\]
where $\bm{e}_j$ is the $j$-th unit vector for $j \in [n]$ and $\varepsilon=\tilde{C}\sqrt{\frac{\log n}{n^{s-1}}}$ for a constant $\tilde C = \tilde{C}(\delta,s)>0$ (depending on $\delta$ and $s$) to be chosen later. By definition, we have
\[
\|\bm \gamma^{(j)}-\bm \gamma^{(\ell)}\|_\infty \ge 2\tilde{C}\sqrt{\frac{\log n}{n^{s-1}}}.
\]
Denote the probability distribution of the $s$-uniform models $\mathsf{H}_{s}(n, \bm \gamma^{(j)})$ by $\P_{\bm \gamma^{(j)}}$. Observe that for all $j \in \{1,\ldots,n\}$
\begin{align}\label{eq:Pgamma}
   \mathsf{KL}(\P_{\bm \gamma^{(j)}}\| \P_{\bm \gamma^{(0)}}) & = \frac{1}{2}\sum_{\bm e \in {[n] \choose s}: j \in \bm e} \left[ \log\left\{\frac{(1+e^{2\varepsilon} ) } {2\,e^{2\varepsilon}} \right\}  + \log\left\{\frac{1}{2}(1+e^{2\varepsilon}) \right\} \right] .  
\end{align} 
By Taylor's theorem, we get 
$$\log\left\{\frac{(1+e^{2\varepsilon} ) } {2\,e^{2\varepsilon}} \right\}  + \log\left\{\frac{1}{2}(1+e^{2\varepsilon}) \right\} = 4\varepsilon^2 + O(\varepsilon^3)= \frac{4\tilde{C}^2\log n}{n^{s-1}} + O\left( \frac{(\log n)^{\frac{3}{2}}}{n^{\frac{3}{2}(s-1)} } \right) .$$ 
Hence, from \eqref{eq:Pgamma}, 
\[
\mathsf{KL}(\P_{\bm \gamma^{(j)}}\| \P_{\bm \gamma^{(0)}})= L_s \tilde{C}^2\log n + O\left( \frac{(\log n)^{\frac{3}{2}}}{n^{\frac{1}{2}(s-1)} } \right), 
\]
for some constant $L_s$ depending on $s$. Therefore, we have
\[
 \frac{1}{n}\sum_{j=1}^{n} \mathsf{KL}(\P_{\bm \gamma^{(j)}}\| \P_{\bm \gamma^{(0)}}) \lesssim_s L_s \tilde{C}^2\log n \le \delta \log n , 
\] 
for $\tilde{C} = \tilde{C}(\delta, s)$ chosen appropriately to ensure $\delta \in (0,1/8)$. Hence, applying Theorem \ref{thm:fl} (with the $L_2$ norm replaced by the $L_\infty$ norm) and taking large enough $n$ in \eqref{eq:fl} gives 
$$\min_{\hat{\bm \gamma}}\max_{\bm \gamma}\mathbb{P}\left(\|\hat{\bm \gamma}-\bm\gamma\|_\infty \ge \tilde{C}\sqrt{\frac{\log n}{n^{s-1}}} \right) \ge 1- 2 \delta . 
$$ 
This completes the proof of \eqref{eq:maxlb}. }

\section{Proof of Theorem \ref{thm:central_lim_thm} and Theorem \ref{thm:conf_int_thm}}
\label{sec:CLT} 

We begin with the proof of Theorem \ref{thm:central_lim_thm} in Section \ref{sec:distributionpf}. The proof of Theorem \ref{thm:conf_int_thm} is given in Section \ref{sec:distributionestimationpf}.

\subsection{Proof of Theorem \ref{thm:central_lim_thm}} 
\label{sec:distributionpf} 

Recall that, for $2 \leq s \leq r$, $\bm d_s = (d_s(1), d_s(2), \ldots, d_s(n))^\top$ is the vector of $s$-degrees. The first step in the proof of Theorem \ref{thm:central_lim_thm} is to derive a linearization of $\hat{\bm \beta}_{s}$ in terms of the $s$-degrees as in Proposition \ref{ppn:betadegree} below. The proof is given in Appendix \ref{sec:betadegreepf}. 

\begin{prop}
\label{ppn:betadegree} 
Fix $2 \leq s \leq r$. Then under the assumptions of Theorem \ref{thm:central_lim_thm}, with probability $1-o(1)$ as $n \rightarrow \infty$, 
\begin{align}\label{eq:betadegreeexpansion}
\| \hat{\bm \beta}_{s}-\bm \beta_{s} - \bm\Sigma_{n, s}^{-1}(\bm d_s-\mathbb{E}[\bm d_s]) \|_\infty = O \left( \frac{\log n}{n^{s-1}} \right) , 
\end{align} 
where  $\bm \Sigma_{n, s}=((\sigma_{s}(u, v)))_{u, v \in [n]}$ is a $n \times n$ matrix with  
\begin{equation}
\label{eq:v_s}
\sigma_{s}(u, v):=\sum_{\bm e \in {[n] \choose s}: u, v \in \bm e}\frac{e^{\bm 1^\top \bm \beta_{s, \bm e}}}{(1+e^{\bm 1^\top \bm \beta_{s, \bm e}})^2} \text{ and } \sigma_{s}(u, u):= \sigma_s(u)^2= \sum_{ \bm e \in {[n] \choose s}: u \in \bm e}\frac{e^{\bm 1^\top \bm \beta_{s, \bm e}}}{(1+e^{\bm 1^\top \bm \beta_{s, \bm e}})^2} , 
\end{equation} 
where $\sigma_s(u)^2$ is also defined in \eqref{eq:sigmadegree}. 
\end{prop}

Next, define the matrix \revsag{$\bm \Gamma_{n, s}=(\gamma_s(u, v))_{u, v \in [n]}$} as follows: 
\begin{equation}
\label{eq:s_n}
\gamma_s(u, v) := \frac{\mathbbm 1\{ u = v \}}{\sigma_{s}(u)^2} . 
\end{equation} 
The following lemma shows that it is possible to replace the matrix $\bm\Sigma_{n, s}^{-1}$ in \eqref{eq:betadegreeexpansion} with the matrix $\bm \Gamma_{n, s}$ asymptotically. The proof of the lemma is given in Appendix \ref{sec:approxvariancepf}.

\begin{lem}
\label{lem:approxvariance}  
Suppose $\bm \Sigma_{n, s}$ and $\bm \Gamma_{n, s}$ be as defined in \eqref{eq:v_s} and \eqref{eq:s_n}, respectively. Then under the assumptions of Theorem \ref{thm:central_lim_thm}, 
\begin{align}\label{eq:delta}
\|\bm \Gamma_{n, s} - \bm\Sigma_{n, s}^{-1}\|_\infty \le O \left(\frac{1}{n^s}\right) , 
\end{align} 
where $\|\bm A\|_\infty = \max_{u, v \in [n]} |a_{u, v}|$ for a matrix $\bm A = ((a_{u, v}))_{u, v \in [n]}$. 
Furthermore, 
\begin{align}\label{eq:deltacovariance}
\|\Cov[(\bm \Gamma_{n, s} - \bm\Sigma_{n, s}^{-1})(\bm d_s-\mathbb{E}[\bm d_s])]\|_\infty \le \|\bm \Gamma_{n, s} - \bm\Sigma_{n, s}^{-1}\|_\infty + O \left(\frac{1}{n^s}\right).
\end{align} 
\end{lem}

To complete the proof of Theorem \ref{thm:central_lim_thm}, consider $J_s \in {[n] \choose a_s}$, for $a_s \geq 1$ fixed. Proposition \ref{ppn:betadegree} and Lemma \ref{lem:approxvariance} combined implies, 
\[
\| [ (\hat{\bm \beta}_{s} - \bm \beta_{s}) ]_{J_s} - [\bm\Gamma_{n, s}(\bm d_s-\mathbb{E}[\bm d_s])]_{J_s} \|_\infty  =  O \left(\revsag{ \frac{1}{\sqrt{n^{s}}}} \right) , 
\] 
with probability $1-o(1)$. Now, recall from the statement of Theorem \ref{thm:central_lim_thm} that $\bm D_s = \mathrm{diag}\,(\sigma_s(v))_{v \in [n]}$. From \eqref{eq:v_s} observe that $\max_{v \in [n]}\sigma_{s}(v)^2 \asymp  n^{s-1}$, since $\|\bm \beta_s\|_\infty \leq M=O(1)$. Hence, 
\[ \| [ \bm D_{s} (\hat{\bm \beta}_s - \bm \beta_{s})]_{J_s} - [\bm D_s(\bm\Gamma_{n, s}(\bm d_s-\mathbb{E}[\bm d_s])]_{J_s} \|_\infty = O \left( \revsag{\frac{1}{\sqrt{n}} } \right) . 
\]
Note that for $v \in J_s$, 
\begin{align}\label{eq:betaDs}
\revsag{\sigma_{s}(v)[\bm \Gamma_{n, s}(\bm d_s-\mathbb{E}[\bm d_s])_v]}&= \frac{d_s(v)-\mathbb{E}[d_s(v)]}{\sigma_{s}(v)} . 
\end{align} 
Therefore, from \eqref{eq:betaDs}, 
\begin{align}
[ \bm D_{s}]_{J_s} ( [(\hat{\bm \beta}_s - \bm \beta_{s})]_{J_s} ) & = \left( \left( \frac{d_s(v)-\mathbb{E}[d_s(v)]}{\sqrt{\Var[d_s(v)]}} \right) \right)_{v \in J_s} + \revsag{O \left(\frac{1}{\sqrt n} \right)} \nonumber \\ 
& \dto \cN_{a_s}(\bm 0, \bm I) , \nonumber 
\end{align} 
using the central limit theorem for sums of independent bounded random variables. Since $\hat{\bm \beta}_s$ are independent across $2 \leq s \leq r$, the result in \eqref{eq:central_lim} follows.

\subsubsection{Proof of Proposition \ref{ppn:betadegree}} 
\label{sec:betadegreepf}

For $2 \leq s \leq r$ and $\bm e = (u_1, u_2, \ldots, u_s) \in {[n] \choose s}$, let $\bm \beta_{s, \bm e} = (\beta_{s, u_1}, \beta_{s, u_2}, \ldots, \beta_{s, u_s})^\top$ and $\hat{\bm \beta}_{s, \bm e} = (\hat{\beta}_{s, u_1}, \hat{\beta}_{s, u_2}, \ldots, \hat{\beta}_{s, u_s})^\top$. Moreover, $\bm 1$ will denote the vector of ones in the appropriate dimension. To begin with, \eqref{eq:def_k_unif} and \eqref{eq:expecteddegree}  gives, for $v \in [n]$, 
\begin{align}\label{eq:dexpansion}
d_s(v) -\mathbb{E}[d_s(v) ] = \sum_{\bm e \in {[n] \choose s}: v \in \bm e} \left\{ \frac{e^{\bm 1^\top \hat{\bm \beta}_{s, \bm e}}}{1+e^{\bm 1^\top \hat{\bm \beta}_{s, \bm e}}}-\frac{e^{\bm 1^\top \bm \beta_{s, \bm e}}}{1+e^{\bm 1^\top \bm \beta_{s, \bm e}}} \right\} . 
\end{align} 
Note that for $\bm e \in {[n] \choose s}$, by a Taylor expansion, 
\[
\frac{e^{\bm 1^\top \hat{\bm \beta}_{s, \bm e}}}{1+e^{\bm 1^\top \hat{\bm \beta}_{s, \bm e}}}-\frac{e^{\bm 1^\top \bm \beta_{s, \bm e}}}{1+e^{\bm 1^\top \bm \beta_{s, \bm e}}} = \frac{e^{\bm 1^\top \bm \beta_{s, \bm e}}}{(1+e^{\bm 1^\top \bm \beta_{s, \bm e}})^2}\left(\bm 1^\top \hat{\bm \beta}_{s, \bm e}-\bm 1^\top \bm \beta_{s, \bm e}\right) + T_{s, \bm e} , 
\]
where 
\begin{equation}
\label{eq:h_bound}
|T_{s, \bm e}| \le \frac{1}{2}\left|\bm 1^\top \hat{\bm \beta}_{s, \bm e}-\bm 1^\top \bm \beta_{s, \bm e}\right|^2 \lesssim_r \|\hat{\bm \beta}_{s}-\bm \beta_{s}\|^2_\infty.
\end{equation} 
Then, from \eqref{eq:dexpansion}, 
\begin{equation}
\label{eq:degree} 
d_s(v) -\mathbb{E}[d_s(v) ]=\left[ \bm\Sigma_{n, s}(\hat{\bm \beta}_s-\bm \beta_{s})\right]_v + R_{v, s},
\end{equation}
where $R_{v, s}=\sum_{\bm e \in {[n] \choose s}: v \in \bm e} T_{s, \bm e}$. 
From \eqref{eq:degree}, we have
\begin{equation}
\label{eq:con_deg_beta}
\hat{\bm \beta}_{s}-\bm \beta_{s}=\bm\Sigma_{n, s}^{-1}(\bm d_s-\mathbb{E}[\bm d_s]) - \bm\Sigma_{n, s}^{-1}\bm R_{n, s} , 
\end{equation} 
where $\bm R_{n, s} = (R_{1, s}, R_{2, s}, \ldots, R_{n, s})^\top$. Note that from \eqref{eq:h_bound}, 
\begin{align}\label{eq:Rbeta}
|R_{v, s}| \le \sum_{\bm e \in {[n] \choose s}: v \in \bm e} |T_{s, \bm e}| \lesssim_r n^{s-1}\|\hat{\bm \beta}_{s}-\bm \beta_{s}\|^2_\infty.
\end{align}
To bound $\| \bm\Sigma_{n, s}^{-1}\bm R_{n, s} \|_\infty$, note that for $v \in [n]$, 
\begin{equation}
\label{eq:var_decom}
 |[\bm\Sigma_{n, s}^{-1}\bm R_{n, s}]_v| \le |[\bm \Gamma_{n, s}\bm R_{n, s}]_v| +| [ (\bm\Sigma_{n, s}^{-1}- \bm \Gamma_{n, s})\bm R_{n, s}]_v|.    
\end{equation}
Observe that 
\[
[\bm \Gamma_{n, s}\bm R_{n, s}]_v = \frac{R_{v, s}}{\sigma_{s}(v)^2} . 
\]
Using $\sigma_{s}(v)^2 \asymp  n^{s-1}$,  \eqref{eq:Rbeta}, and \eqref{eq:mlerlayer} gives, 
\begin{align*}
|[\bm \Gamma_{n, s}\bm R_{n, s}]_v| \lesssim  \|\hat{\bm \beta}_{s}-\bm \beta_{s}\|^2_\infty = O \left(\frac{\log n}{\;n^{s-1}}\right) , 
\end{align*}
with probability $1- o(1)$.  Further, by Lemma \ref{lem:approxvariance}, \eqref{eq:Rbeta}, and \eqref{eq:mlerlayer}, 
\begin{align*}
|[ (\bm\Sigma_{n, s}^{-1}- \bm \Gamma_{n, s}) \bm R_{n, s}]_v| \le \| (\bm\Sigma_{n, s}^{-1}- \bm \Gamma_{n, s}) \|_\infty \times n\|\bm R_{n, s}\|_\infty 
& \lesssim  \|\hat{\bm \beta}_{s}-\bm \beta_{s}\|^2_\infty \\
& \le O \left(\frac{\log n}{\;n^{s-1}}\right)  , 
\end{align*} 
with probability $1- o(1)$. Hence, by \eqref{eq:con_deg_beta} and \eqref{eq:var_decom} the result in \eqref{eq:betadegreeexpansion} follows. \hfill $\Box$

\subsubsection{Proof of Lemma \ref{lem:approxvariance} }
\label{sec:approxvariancepf} 

\begin{proof}[Proof of \eqref{eq:delta}]
Denote $\bm \Delta_{n, s}= \bm \Gamma_{n, s} - \bm\Sigma_{n, s}^{-1}  = (( \delta_s(u, v) ))_{u, v \in [n]}$, $\bm Z_{n, s}= \bm I_{n}- \bm\Sigma_{n, s} \bm \Gamma_{n, s} = (( z_s(u, v) ))_{u, v \in [n]}$, and $\bm \Theta_{n, s}=\bm \Gamma_{n, s} \bm Z_{n, s} = \revsag{( \theta_s(u, v) )_{u, v \in [n]}}$. Then
\[
\bm \Delta_{n, s}=(\bm \Gamma_{n, s} - \bm\Sigma_{n, s}^{-1})(\bm I_n- \bm\Sigma_{n, s} \bm \Gamma_{n, s}) - \bm \Gamma_{n, s}(\bm I_n - \bm\Sigma_{n, s} \bm \Gamma_{n, s})=\bm \Delta_{n, s}\bm Z_{n, s} - \bm \Theta_{n, s}.
\] 
Hence, for $u, v \in [n]$, 
\begin{align} 
\delta_s(u, v) & =\sum_{w=1}^{n} \delta_s(u, w) z_s(w, v) - \theta_s(u, v) \nonumber \\ 
& =\sum_{w=1}^{n} \delta_s(u, w)\left\{ \mathbbm 1\{ w = v \} - \sum_{b=1}^n \sigma_s(w, b) \gamma_s(b, v)  \right\} - \theta_s(u, v) \nonumber \\ 
& =\sum_{w=1}^{n} \delta_s(u, w)\left\{ \mathbbm 1\{ w = v \} - \sum_{b=1}^n \sigma_s(w, b)  \frac{\mathbbm 1\{v=b\}}{\sigma_s(v)^2}   \right\} - \theta_s(u, v) \tag*{(by \eqref{eq:s_n})} \nonumber \\ 
& =\sum_{w=1}^{n} \delta_s(u, w)\left\{ \mathbbm 1\{ w = v \} - \frac{\sigma_s(w, v) }{\sigma_s(v)^2} \right\} - \theta_s(u, v) \nonumber \\ 
\label{eq:deltathetauv} & = - \sum_{w=1}^{n} \delta_s(u, w)\left\{\mathbbm 1\{ w \ne v \} \frac{\sigma_s(w, v) }{\sigma_s(v)^2} \right\} - \theta_s(u, v) , 
\end{align} 
since $\sum_{b \in [n]\backslash \{w\}} \sigma_s(w, b) = \sigma_s(w, w) = \sigma_s(w)^2$. 
The following lemma bounds the maximum norm of $\bm \Theta_{n, s} = \bm \Gamma_{n, s} \bm Z_{n, s} = (( \theta_s(u, v) ))_{u, v \in [n]}$.  

\begin{lem}\label{lm:gammauv} For $u, v, w \in [n]$, 
\begin{align}\label{eq:sigmagammauv}
\max\left\{|\theta_s(u, v)|,|\theta_s(u, v)-\theta_s(v, w)|\right\} \lesssim \frac{\sigma_{s, \max}}{\sigma_{s, \min}^2n^2}  , 
\end{align}
where $\sigma_{s, \min} :=\min_{1 \leq u < v \leq n} \sigma_{s}(u, v)$ and $\sigma_{s, \max} :=\max_{1 \leq u < v \leq n} \sigma_{s}(u, v)$. 
\end{lem}

\begin{proof} Note that $\bm \Theta_{n, s}=\bm \Gamma_{n, s} \bm Z_{n, s} = \bm \Gamma_{n, s} - \bm \Gamma_{n, s} \bm\Sigma_{n, s} \bm \Gamma_{n, s}$. This means for $u, v \in [n]$, 
\begin{align}\label{eq:gammauv}
\theta_s(u, v) = \gamma_s(u, v) - \sum_{x, y \in [n]} \gamma_s(u, x) \sigma_s(x, y) \gamma_s(y, v) . 
\end{align}
Then recalling the definition of $\gamma_s(u, v)$ from \eqref{eq:s_n} gives, 
\begin{align}
\sum_{x, y \in [n]} \gamma_s(u, x) \sigma_s(x, y) \gamma_s(y, v) 
& = \sum_{x, y \in [n]} \frac{\revsag{\mathbbm 1\{ u = x \} \mathbbm 1\{ y = v \}} \sigma_s(x, y) }{\sigma_{s}(u)^2 \sigma_{s}(v)^2}  \nonumber \\ 
& = \frac{ \sigma_s(u, v) }{\sigma_{s}(u)^2 \sigma_{s}(v)^2}  . \nonumber 
\end{align}
Hence, from \eqref{eq:s_n} and \eqref{eq:gammauv}, 
$$|\theta_s(u, v)| =  \left| \frac{ \sigma_s(u, v)\mathbbm 1\{ u \ne v \} }{\sigma_{s}(u)^2 \sigma_{s}(v)^2} \right|  \lesssim \frac{\sigma_{s, \max}}{\sigma_{s, \min}^2n^2} . $$
This completes the proof of \eqref{eq:sigmagammauv}. 
\end{proof}

Now, for $u \in [n]$, let $\overline m, \vec m \in [n]$ be such that 
$$\delta_s(u, \overline m)=\max_{w \in [n]}\delta_s(u, w) \quad \text{ and } \quad \delta_s(u, \vec m)=\min_{w \in [n]}\delta_s(u, w). $$ The following lemma gives bounds on $\delta_s(u, \vec m)$ and $\delta_s(u, \overline m)$.

\begin{lem} \label{lm:deltauv} For $u \in [n]$,  
\begin{align*}
\sum_{w=1}^{n}\delta_s(u, w) \sigma_s(w, u) = 0. 
\end{align*} 
This implies, $\delta_s(u, \overline m) \geq 0$ and $\delta_s(u, \vec m)  \leq 0$. 
\end{lem}

\begin{proof}
Note that $\sum_{w=1}^{n}\delta_s(u, w) \sigma_s(w, u)$ is the $u$-th diagonal element of the matrix 
$\bm \Delta_{n, s} \bm \Sigma_{n, s} =   \bm \Gamma_{n, s} \bm\Sigma_{n, s} - \bm I_{n}$ (recall that $\bm \Delta_{n, s}=\bm \Gamma_{n, s} - \bm\Sigma_{n, s}^{-1}$). Note that the $u$-th diagonal element of $\bm \Gamma_{n, s} \bm\Sigma_{n, s}$ is given by 
\begin{align*} 
 \sum_{w \in [n]}  \gamma_s(u, w)  \sigma_s(w, u) & = \sum_{w \in [n]} \frac{\mathbbm 1\{u=w\}}{\sigma_s(u)^2} \sigma_s(w, u)  = 1 , 
\end{align*} 
since $\sigma_s(u, u) =  \sigma_s(u)^2$. Hence, $u$-th diagonal element of $\bm \Delta_{n, s} \bm \Sigma_{n, s}$ is zero. 
\end{proof}

Now, recalling \eqref{eq:deltathetauv} note that 
\begin{align} 
\label{eq:w_f} 
& \delta_s(u, \overline m)-\delta_s(u, \vec m) + (\theta_s(u, \overline m)-\theta_s(u, \vec m)) \nonumber \\ 
& =\sum_{w=1}^{n} \delta_s(u, w) \left\{   \frac{\mathbbm 1\{ w \ne \vec m\} \sigma_s(w, \vec m)}{\sigma_s(\vec m)^2} -  \frac{\mathbbm 1\{ w \ne \overline m \} \sigma_s(w, \overline m)}{\sigma_s( \overline m )^2} \right\}  \nonumber \\ 
& = \sum_{w=1}^{n}(\delta_s(u, w)-\delta_s(u, \vec m))\left\{  \frac{\mathbbm 1\{ w \ne \vec m\}  \sigma_s(w, \vec m)}{\sigma_s(\vec m)^2} -  \frac{ \mathbbm 1\{ w \ne \overline m \} \sigma_s(w, \overline m)}{\sigma_s( \overline m )^2} \right\} ,  
\end{align} 
since $\sum_{w \in [n]\backslash \{\vec m\}} \sigma_s(w, \vec m) = \sigma_s(\vec m)^2$ and $\sum_{w \in [n]\backslash \{\overline m\} } \sigma_s(w, \overline m) = \sigma_s(\overline m)^2$. Define 
\[
\Omega:=\left\{ w \in [n] :  \frac{\mathbbm 1\{ w \ne \vec m \} \sigma_s(w, \vec m)}{\sigma_s(\vec m)^2} \geq  \frac{\mathbbm 1\{ w \ne \overline m\} \sigma_s(w, \overline m)}{\sigma_s(\overline m)^2}\right\},
\]
and $\lambda := |\Omega|$.  Then, we have
\begin{align}\label{eq:omega} 
&\sum_{w \in \Omega} (\delta_s(u, w)-\delta_s(u, \vec m))\left\{  \frac{\mathbbm 1\{ w \ne \vec m\}  \sigma_s(w, \vec m)}{\sigma_s(\vec m)^2} -  \frac{\mathbbm 1\{ w \ne \overline m\} \sigma_s(w, \overline m)}{\sigma_s(\overline m)^2} \right\}  \nonumber \\
&\le (\delta_s(u, \overline m)-\delta_s(u, \vec m))\left\{ \frac{ \sum_{w \in \Omega}\sigma_s(w, \vec m)}{\sigma_s(\vec m)^2} - \frac{\sum_{w \in \Omega}\mathbbm 1\{ w \ne \overline m \}\sigma_s(w, \overline m)}{\sigma_s(\overline m)^2}\right\}. 
\end{align} 
Note that 
$$\frac{ \sum_{w \in \Omega}\sigma_s(w, \vec m)}{\sigma_s(\vec m)^2}= \frac{\sum_{w \in \Omega}\sigma_s(w, \vec m)}{ \sum_{w \in \Omega}\sigma_s(w, \vec m) + \sum_{w \in [n]\backslash(\Omega \bigcup \vec m) }\sigma_s(w, \vec m)} = \frac{1}{ 1 + \frac{ \sum_{w \in [n]\backslash(\Omega \bigcup \vec m) }\sigma_s(w, \vec m)}{\sum_{w \in \Omega}\sigma_s(w, \vec m)}} , $$ 
since $\vec m \notin \Omega$. Now, observe that  
$$\frac{ \sum_{w \in [n]\backslash(\Omega \bigcup \vec m) }\sigma_s(w, \vec m)}{\sum_{w \in \Omega}\sigma_s(w, \vec m)} \geq \frac{(n-\lambda-1) \sigma_{s, \min}}{ \lambda \sigma_{s, \max}}$$
This implies, 
\begin{align}\label{eq:omegaM}
\frac{ \sum_{w \in \Omega}\sigma_s(w, \vec m)}{\sigma_s(\vec m)^2} \leq \frac{\lambda \sigma_{s, \max}}{ \lambda \sigma_{s, \max} + (n-\lambda-1) \sigma_{s, \min}} . 
\end{align}
Similarly, 
$$\frac{ \sum_{w \in \Omega} \mathbbm 1\{ w \ne \overline m \}\sigma_s(w, \overline m)}{\sigma_s(\overline m)^2} =   \frac{\sum_{w \in \Omega} \mathbbm 1\{ w \ne \overline m \} \sigma_s(w, \overline{m} )}{ \sum_{w \in [n]}\mathbbm 1\{ w \ne \overline m \} \sigma_s(w, \overline{m}) } = \frac{1}{1 + \frac{\sum_{w \in [n]\backslash \Omega} \mathbbm 1\{ w \ne \overline{m} \} \sigma_s(w, \overline{\vec m} ) }{\sum_{w \in \Omega}\mathbbm 1\{ w \ne \overline m \} \sigma_s(w, \overline{\vec m} ) }}  $$ 
Therefore, since $\overline m \in \Omega$, 
$$\frac{\sum_{w \in [n]\backslash \Omega}\mathbbm 1\{ w \ne \overline m \} \sigma_s(w, \overline m) }{\sum_{w \in \Omega}\mathbbm 1\{ w \ne \overline m \} \sigma_s(w, \overline m) } \leq \frac{(n-\lambda) \sigma_{s, \max}}{ (\lambda-1) \sigma_{s, \min}}.$$ 
Hence, 
\begin{align}\label{eq:omegam}
\frac{ \sum_{w \in \Omega} \mathbbm 1\{ w \ne \overline m \}\sigma_s(w, \overline m)}{\sigma_s(\overline m)^2} \geq \frac{(\lambda-1) \sigma_{s, \min} }{(\lambda-1) \sigma_{s, \min} + (n-\lambda) \sigma_{s, \max}} . 
\end{align}
Applying \eqref{eq:omegaM} and \eqref{eq:omegam} in \eqref{eq:omega} gives, 
\begin{align}\label{eq:deltamMpf} 
&\sum_{w \in \Omega} (\delta_s(u, w)-\delta_s(u, \vec m))\left\{  \frac{\mathbbm 1\{ w \ne \vec m\}  \sigma_s(w, \vec m)}{\sigma_s(\vec m)^2} -  \frac{ \mathbbm 1\{ w \ne \overline m\} \sigma_s(w, \overline m)}{\sigma_s(\overline m)^2} \right\}  \nonumber \\ 
& \le (\delta_s(u, \overline m)-\delta_s(u, \vec m))f(\lambda) , 
\end{align} 
where 
\[
f(\lambda)=\frac{\lambda \sigma_{s, \max}}{\lambda \sigma_{s, \max}+(n-1-\lambda)\sigma_{s, \min}}-\frac{(\lambda-1)\sigma_{s, \min}}{(\lambda-1)\sigma_{s, \min}+(n-\lambda)\sigma_{s, \max}}.
\] 
Note that $f(\lambda)$ attains maximum at $\lambda=n/2$ over $\lambda \in (1,n-1)$ and
\[
f(n/2)=\frac{n\sigma_{s, \max}-(n-2)\sigma_{s, \min}}{n\sigma_{s, \max} + (n-2)\sigma_{s, \min}}.
\] 
Therefore, from Lemma \ref{lm:gammauv}, \eqref{eq:w_f}, there exists a constant $C> 0$ such that \eqref{eq:deltamMpf}, 
\begin{align}
\delta_s(u, \overline m)-\delta_s(u, \vec m) &\le \frac{n\sigma_{s, \max}-(n-2)\sigma_{s, \min}}{n\sigma_{s, \max}+(n-2)\sigma_{s, \min}}(\delta_s(u, \overline m)-\delta_s(u, \vec m))+ \frac{C \sigma_{s, \max}}{\sigma_{s, \min}^2n^2} . \nonumber 
\end{align}
This implies, 
\begin{align*}
\delta_s(u, \overline m)-\delta_s(u, \vec m) & \leq \frac{C \sigma_{s, \max}(n \sigma_{s, \max} + (n-2)\sigma_{s, \min})}{ 2 (n-2) \sigma_{s, \min} ^3 n^2} \lesssim \frac{\sigma_{s, \max}^2}{\sigma_{s, \min}^3 n^2} . 
\end{align*}
Hence, from Lemma \ref{lm:deltauv}, 
\begin{align*}
\max_{1 \le w \le n}|\delta_s(u, w)| & \le \delta_s(u, \overline m)-\delta_s(u, \vec m)  
\le \frac{\sigma_{s, \max}^2}{\sigma_{s, \min}^3 n^2} \lesssim  \frac{1}{n^s} ,  
\end{align*} 
since $\sigma_{s, \min} \asymp  n^{s-2}$ and $\sigma_{s, \max} \asymp  n^{s-2}$, using $\|\bm \beta_s\|_\infty \leq M=O(1)$. This completes the proof of \eqref{eq:delta}.  
\end{proof}

\begin{proof}[Proof of \eqref{eq:deltacovariance}]

Define 
$$\bm U_{n, s}= \Cov[(\bm \Gamma_{n, s} - \bm\Sigma_{n, s}^{-1})(\bm d_s-\mathbb{E}[\bm d_s] ) ] = \Cov[ \bm \Delta_{n, s}(\bm d_s-\mathbb{E}[\bm d_s])] ,$$ 
since $\bm \Delta_{n, s}=\bm \Gamma_{n, s} - \bm\Sigma_{n, s}^{-1} $. Observe that 
\begin{align}\label{eq:Uns}
\bm U_{n, s} & = \bm \Delta_{n, s} \E[(\bm d_s-\mathbb{E}[\bm d_s]) (\bm d_s-\mathbb{E}[\bm d_s])^\top ] \bm \Delta_{n, s}^\top \nonumber \\ 
& = \bm \Delta_{n, s} \bm \Sigma_{n, s} \bm \Delta_{n, s}^\top \nonumber \\ 
& = (\bm \Gamma_{n, s} - \bm\Sigma_{n, s}^{-1})-\bm \Gamma_{n, s}(\bm I_n-\bm \Sigma_{n, s} \bm \Gamma_{n, s}) \nonumber \\ 
& = (\bm \Gamma_{n, s} - \bm\Sigma_{n, s}^{-1})-\bm \Theta_{n, s} , 
\end{align} 
since $\bm \Theta_{n, s} = \bm \Gamma_{n, s} \bm Z_{n, s}$ and $\bm Z_{n, s} = \bm I_n-\bm \Sigma_{n, s} \bm \Gamma_{n, s}$. By Lemma \ref{lm:gammauv}, 
\begin{equation}
\label{eq:maximumsigma}
\| \bm \Theta_{n, s}\|_\infty \lesssim \frac{\sigma_{s, \max}}{\sigma_{s, \min}^2 n^2} \lesssim  \frac{1}{n^s} , 
\end{equation} 
since $\sigma_{s, \min} \asymp  n^{s-2}$ and $\sigma_{s, \max} \asymp  n^{s-2}$, using $\|\bm \beta_s\|_\infty \leq M=O(1)$. By \eqref{eq:delta}, \eqref{eq:Uns}, and \eqref{eq:maximumsigma} the result in \eqref{eq:deltacovariance} follows. 
\end{proof}

\subsection{Proof of Theorem \ref{thm:conf_int_thm}} 
\label{sec:distributionestimationpf}

For $\bm x = (x_1, x_2, \ldots, x_n) \in \R^n$ and $u \in [n]$ define the function 
$$g_u(\bm x) = \sum_{ \bm e \in {[n] \choose s}: u \in \bm e} \frac{e^{\bm 1^\top \bm x_{\bm e}}}{(1+e^{\bm 1^\top \bm x_{\bm e}})^2} , 
$$
where $\bm x_{\bm e} = (x_{u_1}, x_{u_2}, \ldots, x_{u_s})$ for $\bm e = (u_1, u_2, \ldots, u_s)$. Then recalling \eqref{eq:sigmadegree} and \eqref{eq:sigmadegreeestimate},  
$\sigma_{s}(v)^2 = g_v(\bm \beta_s)$ and $\hat \sigma_{s}(v)^2 = g_v(\hat{\bm \beta}_s)$ \revsag{for all $v \in [n]$}. Hence, by a Taylor expansion, 
\begin{align}\label{eq:Dsestimate} 
|\hat \sigma_{s}(v)^2-\sigma_{s}(v)^2| =  |g_v(\hat{\bm \beta}_s) - g_v(\bm \beta_s)| & 
= \revsag{\left| \sum_{ \bm e \in {[n] \choose s}: v \in \bm e}  \left\{ \frac{e^{\bm 1^\top \hat{\bm \beta}_{s, \bm e}}}{(1+e^{\bm 1^\top \hat{\bm \beta}_{s, \bm e}})^2}-\frac{e^{\bm 1^\top \bm \beta_{s, \bm e}}}{(1+e^{\bm 1^\top \bm \beta_{s, \bm e}})^2} \right\} \right| }\nonumber \\ 
& \lesssim_r \revsag{\|\hat{\bm \beta}_{s}-\bm \beta_{s}\|_\infty n^{s-1} }. 
\end{align} 
Recalling the definition of $J_s = \{ v_{s,1}, \ldots, v_{s,a_s} \}$ from Theorem \ref{thm:conf_int_thm}, this implies 
\begin{align} 
& \sum_{s=2}^r ([(\hat{\bm \beta}_s - \bm \beta_{s})]_{J_s})^\top [ \hat{\bm D}_{s}^2]_{J_s} ([(\hat{\bm \beta}_s - \bm \beta_{s})]_{J_s}) \nonumber \\ 
& = \sum_{s=2}^r \sum_{j = 1}^{a_s} \hat \sigma_s(v_{a_j})^2 ( \hat \beta_{s, v_{a_j}} - \beta_{s, v_{a_j}})^2 \nonumber \\ 
& = \sum_{s=2}^r \sum_{j = 1}^{a_s} \sigma_s(v_{a_j})^2 ( \hat \beta_{s, v_{a_j}} - \beta_{s, v_{a_j}})^2 + \sum_{s=2}^r \sum_{j = 1}^{a_s} (\hat \sigma_s(v_{a_j})^2 - \sigma_s(v_{a_j})^2) ( \hat \beta_{s, v_{a_j}} - \beta_{s, v_{a_j}})^2  \nonumber \\ 
& \dto \chi^2_{\sum_{s=2}^r a_s} + o_P(1) , \nonumber 
\end{align}
by Theorem \ref{thm:central_lim_thm}, \eqref{eq:Dsestimate} and \eqref{eq:mlerlayer}. This completes the proof of \eqref{eq:confidenceinterval}.

\section{Proofs of Theorems \ref{thm:testing_inhom_level} and \ref{thm:testing_inhom_power}}
\label{sec:hypothesistestingpf}

\subsection{Proof of Theorem \ref{thm:testing_inhom_level}}
\label{sec:hypothesisdistributionpf}

Suppose $H_n \sim \mathsf{H}_{n, s}(n, \bm \gamma)$ for $\bm \gamma$ as in \eqref{eq:H0gamma}. 
Let $\bm \Sigma_{n, s}$ be as defined in \eqref{eq:v_s} with $\bm \beta_s$ replaced by $\bm \gamma = (\gamma_1, \gamma_2, \ldots, \gamma_n)^\top$. Then $\grad^2 \ell_{n, s}(\bm \gamma) = \bm \Sigma_{n, s}$. \revsag{By} a Taylor expansion, 
\begin{align}\label{eq:likelihoodexpansionH0}
    \ell_{n, s}(\bm \gamma)-\ell_{n, s}(\hat{\bm \beta}_{s}) & = (\hat{\bm \beta}_{s} - \bm \gamma)^\top \grad \ell_{n, s}(\bm \gamma) + \frac{1}{2}(\hat{\bm \beta}_{s} - \bm \gamma)^\top \bm \Sigma_{n, s} (\hat{\bm \beta}_{s} - \bm \gamma) + \cT_{n, s}, 
\end{align}
where 
\begin{align}\label{eq:gradient3}
\cT_{n, s} = \cT_{n, s}^{(1)} + \cT_{n, s}^{(2)} + \cT_{n, s}^{(3)} , 
\end{align} 
with 
\begin{align}
   \cT_{n, s}^{(1)} & :=\frac{1}{6} \sum_{u=1}^{n}\frac{\partial^3\ell_{n, s}(\bm \gamma+\theta(\hat{\bm \beta}_{s}-\bm \gamma))}{\partial(\beta_{s, u})^3}(\hat{\beta}_{s, u}-\gamma_{u})^3 , \nonumber \\ 
    \cT_{n, s}^{(2)} & :=   \frac{1}{3} \sum_{1 \leq u \ne v \leq n} \frac{\partial^3\ell_{n, s}(\bm \gamma+\theta(\hat{\bm \beta}_{s}-\bm \gamma))}{\partial(\beta_{s, u})^2\partial\beta_{s, v}}(\hat{\beta}_{s, u}-\gamma_{u})^2(\hat{\beta}_{s, v}-\gamma_{v}) , \nonumber \\ 
       \cT_{n, s}^{(3)} & :=   \frac{1}{6} \sum_{1 \leq u \ne v \ne w \leq n} \frac{\partial^3\ell_{n, s}(\bm \gamma+\theta(\hat{\bm \beta}_{s}-\bm \gamma))}{\partial \beta_{s, u} \partial\beta_{s, v} \partial \beta_{s, w} }(\hat{\beta}_{s, u}-\gamma_{u}) (\hat{\beta}_{s, v}-\gamma_{v}) (\hat{\beta}_{s, w}-\gamma_{w}) , \nonumber  
\end{align} 
for some $\theta \in (0, 1)$.

Now, by arguments as in \eqref{eq:con_deg_beta} it follows that 
\begin{equation}
\label{eq:degreegamma}
\hat{\bm \beta}_{s}-\bm \gamma = \bm\Sigma_{n, s}^{-1}(\bm d_s-\mathbb{E}_{\bm \gamma}[\bm d_s]) + \bm\Sigma_{n, s}^{-1}\bm R_{n, s} , 
\end{equation}  
where $\bm R_{n, s}$ is as defined in \eqref{eq:degree} and \eqref{eq:con_deg_beta} with $\bm \beta_s$ replaced by $\bm \gamma$. Using this and noting that $-\grad \ell_{n, s}(\bm \gamma)  = \bm d_s-\mathbb E_{\bm \gamma}[\bm d_s]$, 
\begin{align}\label{eq:gradientH0}
(\hat{\bm \beta}_{s} - \bm \gamma)^\top \grad \ell_{n, s}(\bm \gamma) & = (\bm d_s-\mathbb{E}_{\bm \gamma}[\bm d_s])^\top \bm\Sigma_{n, s}^{-1} \grad \ell_{n, s}(\bm \gamma) + \bm R_{n, s}^\top \bm\Sigma_{n, s}^{-1} \grad \ell_{n, s}(\bm \gamma) \nonumber \\ 
& = - (\bm d_s-\mathbb{E}_{\bm \gamma}[\bm d_s])^\top \bm\Sigma_{n, s}^{-1} (\bm d_s-\mathbb E_{\bm \gamma}[\bm d_s]) - \bm R_{n, s}^\top \bm\Sigma_{n, s}^{-1} (\bm d_s-\mathbb E_{\bm \gamma}[\bm d_s]) .  
\end{align}
Similarly, using \eqref{eq:degreegamma}, 
\begin{align}\label{eq:gradient2H0}
& (\hat{\bm \beta}_{s} - \bm \gamma)^\top \bm \Sigma_{n, s} (\hat{\bm \beta}_{s} - \bm \gamma) \nonumber \\  
& = (\bm d_s-\mathbb E_{\bm \gamma}[\bm d_s])^\top \bm\Sigma_{n, s}^{-1}(\bm d_s-\mathbb E_{\bm \gamma}[\bm d_s]) + 2 \bm R_{n, s}^\top \bm\Sigma_{n, s}^{-1} (\bm d_s-\mathbb E_{\bm \gamma}[\bm d_s]) + \bm R_{n, s}^\top \bm\Sigma_{n, s}^{-1}\bm R_{n, s} . 
\end{align}
Combining \eqref{eq:likelihoodexpansionH0}, \eqref{eq:gradientH0}, and \eqref{eq:gradient2H0} gives, 
\begin{align}\label{eq:lgammaexpansionH0}
    \ell_{n, s}(\hat{\bm \beta}_{s})-\ell_{n, s}(\bm \gamma) 
        &= - \frac{1}{2}(\bm d_s-\mathbb E_{\bm \gamma}[\bm d_s])^\top \bm\Sigma_{n, s}^{-1}(\bm d_s-\mathbb E_{\bm \gamma}[\bm d_s]) + \frac{1}{2} \bm R_{n, s}^\top \bm\Sigma_{n, s}^{-1}\bm R_{n, s} + \cT_{n, s} . 
\end{align}

We begin by showing that $\bm R_{n, s}^\top \bm\Sigma_{n, s}^{-1}\bm R_{n, s} = o_P(\sqrt n)$. 
To this end, \eqref{eq:Rbeta} and $\sigma_s(u)^2 \asymp  n^{s-1}$ gives, 
\begin{align}\label{eq:Rsigmagamma} 
\left|\bm R_{n, s}^\top \bm \Gamma_{n, s} \bm R_{n, s}\right| = \left|\sum_{u=1}^{n}\frac{R_{s, u}^2}{\sigma_{s}(u)^2} \right|  \lesssim  n^{s} \|\hat{\bm \beta}_{s}-\bm \beta_{s}\|^4_\infty  \lesssim  \frac{\log^2n}{n^{s-2} } , 
\end{align}
with probability $1-o(1)$ by \eqref{eq:mlerlayer}. Next, observe that
\begin{align} 
\left|\bm R_{n, s}^\top \bm \Delta_{n, s}\bm R_{n, s}\right| \le n \| \bm \Delta_{n, s}\bm R_{n, s} \|_\infty \cdot \| \bm R_{n, s}\|_\infty 
& \leq n^2 \| \bm R_{n, s}\|_\infty^2 \cdot \|\bm \Delta_{n, s}\|_\infty  \nonumber \\ 
& \lesssim  n^{s} \|\hat{\bm \beta}_{s}-\bm \beta_{s}\|^4_\infty \tag*{(by \eqref{eq:delta} and \eqref{eq:Rbeta})} \nonumber \\ 
\label{eq:Rdelta} & \lesssim  \frac{\log^2n}{n^{s-2} } , 
\end{align} 
with probability $1-o(1)$ by \eqref{eq:mlerlayer}. Combining \eqref{eq:Rsigmagamma} and \eqref{eq:Rdelta} it follows that with probability $1-o(1)$, 
\begin{align}\label{eq:Rgamma} 
\left|\bm R_{n, s}^\top \bm \Sigma_{n, s}^{-1} \bm R_{n, s}\right| \leq \left|\bm R_{n, s}^\top \bm \Gamma_{n, s} \bm R_{n, s}\right| + \left|\bm R_{n, s}^\top \bm \Delta_{n, s} \bm R_{n, s}\right| \lesssim  \frac{\log^2n}{n^{s-2} } = o_P(\sqrt n) .
\end{align}
This implies, the second term in the RHS of \eqref{eq:lgammaexpansionH0} does not contribute to the CLT of the log-likelihood ratio $\log \Lambda_{n, s}$.

Next, we show that the third term in the RHS of \eqref{eq:lgammaexpansionH0} is $o_P(\sqrt n)$, hence, it also does not contribute to the CLT of $\log \Lambda_{n, s}$.

\begin{lem}
\label{lm:conditiongradient}
Suppose $s \geq 3$ and $\bm \gamma \in \mathcal{B}(M)$. Then $\cT_{n, s} = o_P(\sqrt n)$. 
\end{lem}

\begin{proof}
Define $\tilde{\bm \beta}_{ s}=\bm \gamma+\theta(\hat{\bm \beta}_{s}-\bm \gamma)$, for $\theta \in (0, 1)$. Then recalling \eqref{eq:gradient3} observe that 
\begin{align*}
   \cT_{n, s}^{(1)} & =\frac{1}{6} \sum_{u=1}^{n}\left[\sum_{\bm e \in {[n] \choose s}: u \in \bm e}\frac{e^{\bm 1^\top \tilde{\bm \beta}_{ s, \bm e}}(1-e^{\bm 1^\top \tilde{\bm \beta}_{ s, \bm e}})}{(1+e^{\bm 1^\top \tilde{\bm \beta}_{ s, \bm e}})^3}\right](\hat{\beta}_{s, u}-\gamma_{u})^3 , \nonumber \\ 
    \cT_{n, s}^{(2)} & =   \frac{1}{3} \sum_{1 \leq u \ne v \leq n} \left[\sum_{\bm e \in {[n] \choose s}: u, v \in \bm e}\frac{e^{\bm 1^\top \tilde{\bm \beta}_{ s, \bm e}}(1-e^{\bm 1^\top \tilde{\bm \beta}_{ s, \bm e}})}{(1+e^{\bm 1^\top\tilde{\bm \beta}_{ s, \bm e}})^3}\right](\hat{\beta}_{s, u}-\gamma_{u})^2(\hat{\beta}_{s, v}-\gamma_{v}) , \nonumber \\ 
       \cT_{n, s}^{(3)} & :=   \frac{1}{6} \sum_{1 \leq u \ne v \ne w \leq n}  \left[\sum_{\bm e \in {[n] \choose s}: u, v, w \in \bm e}\frac{e^{\bm 1^\top \tilde{\bm \beta}_{ s, \bm e}}(1-e^{\bm 1^\top \tilde{\bm \beta}_{ s, \bm e}})}{(1+e^{\bm 1^\top \tilde{\bm \beta}_{ s, \bm e}})^3}\right](\hat{\beta}_{s, u}-\gamma_{u}) (\hat{\beta}_{s, v}-\gamma_{v}) (\hat{\beta}_{s, w}-\gamma_{w}) ,   
\end{align*} 
where $\tilde{\bm \beta}_{ s, \bm e} = (\tilde{\beta}_{s, u_1}, \tilde{\beta}_{s, u_2}, \ldots, \tilde{\beta}_{s, u_s})^\top$, for $\bm e = (u_1, u_2, \ldots, u_s)$.  
Since $\bm \gamma \in \mathcal{B}_M$ and $\| \hat{\bm \beta}_s - \bm \gamma  \|_\infty \lesssim_{s, M} \sqrt{\log n/n^{s-1}}$ with probability $1-o(1)$, $\hat{\bm \beta}_s \in \mathcal{B}_{2M}$ for large $n$ with probability $1-o(1)$. This implies, 
\begin{equation}\label{eq:Tns}
\cT_{n, s}^{(1)} \lesssim_{M} n^s \| \hat{\bm \beta}_s - \bm \gamma  \|^{3}_\infty  \lesssim_{M, s}  \sqrt{\frac{(\log n)^3}{n^{s-3}}} = o_P(\sqrt n) , 
\end{equation} 
for $s \geq 3$. Similarly, we can show that for $s \geq 3$, $\cT_{n, s}^{(2)} = o_P(\sqrt n)$ and $\cT_{n, s}^{(3)} = o_P(\sqrt n)$. This completes the proof of the Lemma \ref{lm:conditiongradient}. 
\end{proof}

\begin{remark} Note that Lemma \ref{lm:conditiongradient} assumes that $s \geq 3$. This is because when $s=2$ (that is, the graph case), the proof of Lemma \ref{lm:conditiongradient} gives the bound $\cT_{n, 2} = O(\mathrm{polygon}(n)/\sqrt n)$ which is not $o_P(\sqrt n)$ (see \eqref{eq:Tns}). Nevertheless, it follows from the proof of Theorem 1 (a) in \citet{yan2022wilks}, where the asymptotic null distribution of the LR test for the graph $\bm \beta$-model was derived, that the result in Lemma \ref{lm:conditiongradient} also holds when $s=2$, that is, $\cT_{n, 2} = o_P(\sqrt n)$.
For this one has to expand $\ell_{n, s}(\hat{\bm \beta}_{s})-\ell_{n, s}(\bm \gamma)$ up to the fourth order term, and show that the third order term is $o_P(\sqrt n)$ at the true parameter value and the fourth order term is $o_P(\sqrt{n})$ at an intermediate point. 
For $s \geq 3$, the third order term at an intermediate point is $o_P(\sqrt n)$, hence, we do not have to consider the fourth order term. 
\end{remark}

Now, recall the definition of $\log \Lambda_{n, s}$ from \eqref{eq:lambda}. Then by Lemma \ref{lm:conditiongradient} and \eqref{eq:lgammaexpansionH0}  
\begin{align}\label{eq:loglambdaexpansion} 
 \frac{ 2 \log \Lambda_{n, s} - n}{\sqrt{2n}} = \frac{(\bm d_s-\mathbb E_{\bm \gamma}[\bm d_s])^\top \bm\Sigma_{n, s}^{-1}(\bm d_s-\mathbb E_{\bm \gamma}[\bm d_s]) - n }{\sqrt{2n}} + o_P(1) . 
\end{align} 
By the following lemma we can replace $\bm\Sigma_{n, s}^{-1}$ with $\bm \Gamma_{n, s}$ in the RHS above. The proof of the lemma is given in Appendix \ref{sec:deltasigmapf}. 

\begin{lem}\label{lm:deltasigma} For $L > 0$, 
$$\P\left(\bm d_s-\mathbb E_{\bm \gamma}[\bm d_s])^\top (\bm\Sigma_{n, s}^{-1} - \bm \Gamma_{n, s}) (\bm d_s-\mathbb E_{\bm \gamma}[\bm d_s]) > L \right) \lesssim  \frac{1}{L^2}.$$ 
This implies, $(\bm d_s-\mathbb E_{\bm \gamma}[\bm d_s])^\top (\bm\Sigma_{n, s}^{-1} - \bm \Gamma_{n, s}) (\bm d_s-\mathbb E_{\bm \gamma}[\bm d_s])$ is bounded in probability.  
\end{lem}

By Lemma \ref{lm:deltasigma} and recalling \eqref{eq:s_n}, 
 \begin{align}\label{eq:degreequadratic}
\frac{(\bm d_s-\mathbb E_{\bm \gamma}[\bm d_s])^\top \bm\Sigma_{n, s}^{-1}(\bm d_s-\mathbb E_{\bm \gamma}[\bm d_s])}{\sqrt{2 n}} & = \frac{(\bm d_s-\mathbb E_{\bm \gamma}[\bm d_s])^\top \bm\Gamma_{n, s}(\bm d_s-\mathbb E_{\bm \gamma}[\bm d_s])}{\sqrt{2 n}} + o_P\left(1\right)  \nonumber \\ 
& = \frac{1}{\sqrt{2n}}\sum_{u=1}^{n}\frac{(d_s(u)  - \mathbb{E}_{\bm \gamma}[d_s(u) ])^2}{\sigma_{s}(u)^2 }  + o_P\left(1\right) . 
 \end{align} 
Proposition \ref{ppn:lambdaH0} establishes the asymptotic normality of the leading term in the RHS above. The proof is given in Appendix \ref{sec:lambdaH0pf}. \hfill $\Box$ 
  
\begin{prop}\label{ppn:lambdaH0} Under the assumption of Theorem \ref{thm:testing_inhom_level}, 
\begin{align}\label{eq:lambdacltdiagonal}
\frac{1}{\sqrt{2n}} \left\{ \sum_{u=1}^{n} \frac{(d_s(u)  - \mathbb{E}_{\bm \gamma}[d_s(u) ])^2}{\sigma_{s}(u)^2 } -n \right\} \dto \cN(0, 1) . 
\end{align}  
\end{prop}

The result in \eqref{eq:lambdaH0} now follows from \eqref{eq:loglambdaexpansion}, \eqref{eq:degreequadratic}, and Proposition \ref{ppn:lambdaH0}.

\subsubsection{Proof of Lemma \ref{lm:deltasigma}} 
\label{sec:deltasigmapf}
To begin with note that  
\begin{align*} 
\E_{\bm \gamma}[(\bm d_s-\mathbb E_{\bm \gamma}[\bm d_s])^\top (\bm\Sigma_{n, s}^{-1} - \bm \Gamma_{n, s}) (\bm d_s-\mathbb E_{\bm \gamma}[\bm d_s])] 
& = \tr (\E_{\bm \gamma}[(\bm d_s-\mathbb E_{\bm \gamma}[\bm d_s]) (\bm d_s-\mathbb E_{\bm \gamma}[\bm d_s])^\top ]  (\bm\Sigma_{n, s}^{-1} - \bm \Gamma_{n, s}) )\nonumber \\ 
& = \tr ( \bm I_{n} - \bm \Sigma_{n, s} \bm \Gamma_{n, s} ) \nonumber \\  
& = n- \sum_{u, v \in [n]} \sigma_s(u, v) \gamma_s(u, v) \nonumber \\ 
& = n- \sum_{u, v \in [n]} \sigma_s(u, v) \frac{\mathbbm 1\{u=v\}}{\sigma_{s}(u)^2} = 0 . 
 \end{align*} 
 
Next, we will show that $\Var_{\bm \gamma}[(\bm d_s-\mathbb E_{\bm \gamma}[\bm d_s])^\top (\bm\Sigma_{n, s}^{-1} - \bm \Gamma_{n, s}) (\bm d_s-\mathbb E_{\bm \gamma}[\bm d_s])] = O (1)$. The result in Lemma \ref{lm:deltasigma} then follows by Chebyshev's inequality. Recall that $\bm \Delta_{n, s}:=\bm\Sigma_{n, s}^{-1} - \bm \Gamma_{n, s}$. We shall denote the entries of $\bm \Delta_{n, s}$ by $(( \delta_{u,v} )) $ for $u, v \in [n]$. Then 
\[
(\bm d_s-\mathbb E_{\bm \gamma}[\bm d_s])^\top (\bm\Sigma_{n, s}^{-1} - \bm \Gamma_{n, s}) (\bm d_s-\mathbb E_{\bm \gamma}[\bm d_s]) = \sum_{u,v \in [n]} \delta_{u, v} (d_s(u)-\mathbb E_{\bm \gamma}[d_s(u)])(d_s(v)-\mathbb E_{\bm \gamma}[d_s(v)]).
\]
Define $\bar{d}_{s}(u) := d_s(u)-\mathbb E_{\bm \gamma}[d_s(u)]$, for $u \in [n]$. Then 
\begin{align}
\label{eq:var_quad_form}
& \Var_{\bm \gamma}[(\bm d_s-\mathbb E_{\bm \gamma}[\bm d_s])^\top (\bm\Sigma_{n, s}^{-1} - \bm \Gamma_{n, s}) (\bm d_s-\mathbb E_{\bm \gamma}[\bm d_s])] \nonumber \\ 
& = \sum_{ u,v,u',v' \in [n] } \delta_{u, v} \delta_{u', v'}\Cov_{\bm \gamma}[\bar{d}_{s}(u)\bar{d}_{s}(v), \bar{d}_{s}(u')\bar{d}_{s}(v')] . 
\end{align} 
To analyze the RHS of \eqref{eq:var_quad_form} we consider the following 4 cases. 

\begin{itemize} 

\item[{\it Case} 1:] $u=v=u'=v'$.  In this case we have
\[
\Cov_{\bm \gamma}[\bar{d}_{s}(u)\bar{d}_{s}(v), \bar{d}_{s}(u')\bar{d}_{s}(v')] = \Var_{\bm \gamma}[\bar{d}_{s}(u)^2].
\]
For $\bm e \in {[n] \choose s}$, denote \revsag{$X_{\bm e}: =\mathbbm 1\{ \bm e \in E(H_n) \}$} 
and \revsag{$\bar X_{\bm e} :=\mathbbm 1\{ \bm e \in E(H_n) \} - \E[\mathbbm 1\{ \bm e \in E(H_n) \}]$}. Since $\{\bar X_{\bm e}: \bm e \in {[n] \choose s} \}$ are independent and have zero mean, $\{\bar X_{\bm e} \bar X_{\bm e'}: \bm e, \bm e' \in {[n] \choose s} \}$ are pairwise uncorrelated. 
Hence, 
\begin{align}\label{eq:degreesecond}
\Var_{\bm \gamma}[\bar{d}_{s}(u)^2] & = \Var_{\bm \gamma} \left[ \sum_{\bm e, \bm e' \in {[n] \choose s} : u \in \bm e \bigcap \bm e' } \bar X_{\bm e} \bar X_{\bm e'} \right] \nonumber \\ 
& = \sum_{\bm e, \bm e' \in {[n] \choose s} : u \in \bm e \bigcap \bm e' } 
\Var_{\bm \gamma}\left[\bar X_{\bm e} \bar X_{\bm e'} \right] \nonumber \\ 
& = \sum_{\bm e \in {[n] \choose s}: u \in \bm e}\Var_{\bm \gamma}[\bar{X}^2_{\bm e}] + \sum_{\bm e \ne \bm e' \in {[n] \choose s} : u \in \bm e \cap \bm e' }  \Var_{\bm \gamma}[\bar X_{\bm e}] \Var_{\bm \gamma}[ \bar X_{\bm e'} ] . 
\end{align} 
Since $\|\bm \gamma\|_\infty \le M$, 
$$\Var_{\bm \gamma}[\bar X_{\bm e}] = \Var_{\bm \gamma}[ X_{\bm e}]= \frac{e^{\bm 1^\top \gamma_{\bm e}}}{(1 + e^{\bm 1^\top \bm \gamma_{\bm e}}) }  \asymp_M 1 , $$ 
where $\gamma_{\bm e} = (\gamma_{u_1}, \gamma_{u_2}, \ldots, \gamma_{u_s})^\top$, for $\bm e = (u_1, u_2, \ldots, u_s)$. Similarly, $\Var_{\bm \gamma}[\bar{X}^2_{\bm e}] \asymp_M 1$. Hence, \eqref{eq:degreesecond} implies that 
\[
\Var_{\bm \gamma}[\bar{d}_{s}(u)^2] \lesssim_{M} n^{2s-2}.
\] 

\item [{\it Case} 2:] $u \neq v=u'=v'$. Observe that 
\begin{align*}
& \Cov_{\bm \gamma}[\bar{d}_{s}(u)\bar{d}_{s}(v), \bar{d}_{s}(u')\bar{d}_{s}(v')] \nonumber \\ 
& = \Cov_{\bm \gamma}[\bar{d}_{s}(u)\bar{d}_{s}(v), \bar{d}_{s}(v)^2 ] \nonumber \\ 
& = \sum_{\substack{\bm e_1,  \bm e_2, \bm e_3, \bm e_4 \in {[n] \choose s} \\ u \in \bm e_1, v \in \bm e_1 \cap \bm e_2 \cap \bm e_3} } \left\{\mathbb E_{\bm \gamma}[\bar{X}_{\bm e_1}\bar{X}_{\bm e_2}\bar{X}_{\bm e_3}\bar{X}_{\bm e_4}]-\mathbb E_{\bm \gamma}[\bar{X}_{\bm e_1}\bar{X}_{\bm e_2}]\mathbb E_{\bm \gamma}[\bar{X}_{\bm e_3}\bar{X}_{\bm e_4}]\right\}.
\end{align*}  
Note that the non-zero contributions in the RHS above come from the terms when $\bm e_i = \bm e_j$ and $\bm e_k = \bm e_\ell$ for $i,j,k,\ell \in \{1,\ldots,4\}$. Hence, 
\begin{align*}
& \Cov_{\bm \gamma}[\bar{d}_{s}(u)\bar{d}_{s}(v), \bar{d}_{s}(v)^2 ] \nonumber \\ 
& = \sum_{\bm e \in {[n] \choose s}, u, v \in \bm e} \left( \mathbb E_{\bm \gamma}[\bar{X}^4_{\bm e}] - (\mathbb E_{\bm \gamma}[\bar{X}_{\bm e}^2])^2 \right) + 2\sum_{\substack{\bm e_1 \ne \bm e_2 \in {[n] \choose s} \\ u, v \in \bm e_1, v \in \bm e_2}}\mathbb{E}_{\bm \gamma}[\bar{X}^2_{\bm e_1}]\mathbb{E}[\bar{X}^2_{\bm e_2}] \nonumber \\ 
 & \lesssim_M n^{2s-3} , 
\end{align*} 
since $\mathbb E_{\bm \gamma}[\bar{X}^4_{\bm e}] \asymp_M 1$ and $\mathbb E_{\bm \gamma}[\bar{X}_{\bm e}^2]) \asymp_M 1$.  

\item[{\it Case} 3:] $u \neq v \neq u'=v'$:  By similar reasoning as the previous two cases it can be shown that 
\[
\Cov_{\bm \gamma}[\bar{d}_{s}(u)\bar{d}_{s}(v),\bar{d}_{s}(u')\bar{d}_{s}(v')] = \Cov_{\bm \gamma}[\bar{d}^2_{s}(u),\bar{d}_{s}(u')\bar{d}_{s}(u')] \lesssim_{M} n^{2s-3}.
\]

\item[{\it Case} 4:] $u \neq v \neq u' \neq v'$. In this case, it can be shown that 
\[
\Cov_{\bm \gamma}[\bar{d}_{s}(u)\bar{d}_{s}(v),\bar{d}_{s}(u')\bar{d}_{s}(v')] \lesssim_{M} n^{2s-4}.
\]

\end{itemize} 
Combining the 4 cases and using \eqref{eq:var_quad_form}, 
\begin{align*}
\Var_{\bm \gamma}[(\bm d_s-\mathbb E_{\bm \gamma}[\bm d_s])^\top (\bm\Sigma_{n, s}^{-1} - \bm \Gamma_{n, s}) (\bm d_s-\mathbb E_{\bm \gamma}[\bm d_s])] 
& \lesssim_{M} \max_{u, v \in [n]}|\delta_{u, v}|^2  n^{2s} = O(1) , 
\end{align*} 
where the last step uses \eqref{eq:delta}.

\subsubsection{Proof of Proposition \ref{ppn:lambdaH0}} 
\label{sec:lambdaH0pf}
Suppose $H_n = (V(H_n), E(H_n)) \sim \mathsf{H}_{n, s}(n, \bm \gamma)$ for $\bm \gamma$ as in \eqref{eq:H0gamma}. For $\bm e= \{v_1, v_2, \ldots, v_s\} \in {[n] \choose s}$, denote 
\[
X_{\bm e} := X_{\{v_1, v_2, \ldots, v_s\}} : =\mathbbm 1\{ \bm e \in E(H_n) \} . 
\] 
and $\revsag{\bar X_{\bm e} :=\mathbbm 1\{ \bm e \in E(H_n) \} - \E_{\bm \gamma}[\mathbbm 1\{ \bm e \in E(H_n) \}]}$. Also, for $u \in [n]$ denote 
$$\bar d_s(u) = d_s(u) - \E_{\bm \gamma}[d_s(u)] = \sum_{\bm e \in {[n] \choose s} : u \in \bm e} \bar X_{\bm e}.$$
Observe that 
\begin{align}\label{eq:dsu} 
\bar d_s(u)^2 =  \sum_{\bm e \in {[n] \choose s} : u \in \bm e} \bar X_{\bm e}^2 +  \sum_{\bm e \ne \bm e' \in {[n] \choose s} : u \in \bm e \cap \bm e'} \bar X_{\bm e} \bar X_{\bm e'} . 
\end{align} 
This implies, 
$$\E_{\bm \gamma}[\bar d_s(u)^2] = \sum_{\bm e \in {[n] \choose s} : u \in \bm e} \E_{\bm \gamma}[\bar X_{\bm e}^2] = \sum_{\bm e \in {[n] \choose s} : u \in \bm e} \Var_{\bm \gamma}[\bar X_{\bm e}^2] = \Var_{\bm \gamma}[d_s(u)] = \sigma_s(u)^2 . $$
Hence, 
\begin{align} 
& \frac{1}{\sqrt{2n}} \left\{ \sum_{u=1}^{n} \frac{(d_s(u)  - \mathbb{E}_{\bm \gamma}[d_s(u) ])^2}{ \sigma_{s}(u)^2 } -n \right\} \nonumber \\ 
& = \frac{1}{\sqrt{2n}} \sum_{u=1}^{n}  \frac{ \bar d_s(u)^2 - \E_{\bm \gamma}[\bar d_s(u)^2]  }{ \sigma_s(u)^2 } \nonumber \\ 
& = \frac{1}{\sqrt{2n}} \sum_{u=1}^{n} \sum_{\bm e \in {[n] \choose s} : u \in \bm e} \frac{ \bar X_{\bm e}^2 - \E_{\bm \gamma}[\bar X_{\bm e}^2] }{\sigma_{s}(u)^2  } +  \frac{1}{\sqrt{2n}} \sum_{u=1}^{n} \sum_{\bm e \ne \bm e' \in {[n] \choose s} : u \in \bm e \cap \bm e'} \frac{ \bar X_{\bm e} \bar X_{\bm e'} }{ \sigma_{s}(u)^2 }  \tag*{(by \eqref{eq:dsu})} \nonumber \\ 
& := T_1 + T_2 . 
\label{eq:cltT} 
\end{align} 

We will first show that $T_1 = o_P(1)$. Towards this note that 
$$T_1 = \frac{s}{\sqrt{2n}} \sum_{\bm e \in { [n] \choose s} } \frac{ \bar X_{\bm e}^2 - \E_{\bm \gamma}[\bar X_{\bm e}^2] }{\sigma_{s}(u)^2  }.$$
Since $\{ \bar X_{\bm e } : \bm e \in  { [n] \choose s} \}$ are independent, 
\begin{align*} 
\Var_{\bm \gamma}[T_1] =  \frac{s^2}{ 2 n } \sum_{\bm e \in { [n] \choose s} } \frac{ \Var_{\bm \gamma} [ \bar X_{\bm e}^2 ] }{\sigma_{s}(u)^4  }  \lesssim_M  \frac{1}{n^{s-1}} , 
\end{align*} 
using $\Var_{\bm \gamma}[\bar X_{\bm e}^2] \asymp_M 1$ and $\sigma_{s}(u)^2 \asymp_M n^{s-1}$. This implies, $T_1 = o_P(1)$. 

Therefore, from \eqref{eq:cltT}, to prove \eqref{eq:lambdacltdiagonal} it remains to show $T_2 \dto N(0, 1)$. For this we will express $T_2$ as a sum of a martingale difference sequence. To this end, define the following sequence of sigma-fields: For $u \in [n]$, 
$$\cF_u := \sigma\left(\bigcup_{v=1}^u\{ \bar X_{\bm e}: v \in \bm e \}\right), $$
is the sigma algebra generated by the collection of random variables $\bigcup_{v=1}^u\{ \bar X_{\bm e}: v \in \bm e \}$. Clearly, $\cF_1 \subseteq \cF_2 \cdots \subseteq \cF_n$, hence $\{\cF_u\}_{u \in [n]}$ is a filtration. Now, for $u \in [n]$, define 
%
%
\begin{align*}
T_{2, u} :=   \sum_{ \substack{ \bm e, \bm e' \in {[n] \choose s}: \bm e \ne \bm e', u \in \bm e \cap \bm e',  \\ 
 \bm e \cap \{1,\ldots, u\} \neq \emptyset \\ 
\text{ and }  \bm e' \cap \{1,\ldots,u-1\} = \emptyset} } w_{\bm e, \bm e'}  \bar X_{\bm e} \bar X_{\bm e'} 
\end{align*} 
where $w_{\bm e, \bm e'}:= \sum_{z \in \bm e \cap \bm e'} \frac{1}{\sigma_s(z)^2}$. Note that $T_{2, u}$ is $\cF_u$ measurable and $\E_{\bm \gamma}[T_{2, u} | \cF_{u-1} ] = 0$, that is, $T_{2, u}$, for $u \in [n]$, is a martingale difference sequence. Also, recalling the definition of $T_2$ from \eqref{eq:cltT} observe that 
\begin{align*} 
T_2  = \frac{1}{\sqrt{2n}} \sum_{u=1}^{n} \sum_{\bm e \ne \bm e' \in {[n] \choose s} : u \in \bm e \cap \bm e'} \frac{ \bar X_{\bm e} \bar X_{\bm e'} }{ \sigma_{s}(u)^2 } 
& = \frac{1}{\sqrt{2n}}  \sum_{\bm e \ne \bm e' \in {[n] \choose s}, \bm e \cap \bm e' \ne \emptyset} w_{\bm e, \bm e'} \bar X_{\bm e} \bar X_{\bm e'} 
\nonumber \\ 
& = \frac{1}{\sqrt{2n}}\sum_{u=1}^n T_{2, u} , 
\end{align*} 
that is, $T_2 $ is the sum of a martingale difference sequence. Now, invoking the martingale central theorem \cite{brownclt} it can be shown that $T_2 \dto N(0, 1)$. The details are omitted.

\subsection{Proof of Theorem \ref{thm:testing_inhom_power}}
\label{sec:asymptoticH1pf} 

Suppose $H_n \sim \mathsf{H}_{n, s}(n, \bm \gamma')$ for $\bm \gamma'$ as in \eqref{eq:H1gamma}. Then by arguments as in \eqref{eq:lgammaexpansionH0},  
\begin{align*} 
    \ell_{n, s}(\hat{\bm \beta}_{s})-\ell_{n, s}(\bm \gamma') &= - \frac{1}{2}(\bm d_s-\mathbb E_{\bm \gamma'}[\bm d_s])^\top \overline{\bm\Sigma}_{n, s}^{-1}(\bm d_s-\mathbb E_{\bm \gamma'}[\bm d_s]) + \frac{1}{2} \bm R_{n, s}^\top \overline{\bm\Sigma}_{n, s}^{-1}\bm R_{n, s} + \cT_{n, s} , 
\end{align*} 
where $\overline{\bm\Sigma}_{n, s}$ and $\bm R_{n, s}$ are as defined in \eqref{eq:v_s} and \eqref{eq:degree}, respectively, with $\bm \beta_s$ replaced by $\bm \gamma'$ and  $\cT_{n, s}$ as defined in \eqref{eq:gradient3} with $\bm \gamma$ replaced by $\bm \gamma'$. Therefore, 
\begin{align}\label{eq:lgammaexpansionH1}
    \ell_{n, s}(\hat{\bm \beta}_{s})-\ell_{n, s}(\bm \gamma) = - \frac{1}{2}(\bm d_s-\mathbb E_{\bm \gamma'}[\bm d_s])^\top & \overline{\bm\Sigma}_{n, s}^{-1}(\bm d_s-\mathbb E_{\bm \gamma'}[\bm d_s]) \nonumber \\
    & + \frac{1}{2}\bm R_{n, s}^\top\overline{\bm\Sigma}_{n, s}^{-1}\bm R_{n, s}+ \cT_{n, s} +\ell_{n, s}(\bm \gamma')-\ell_{n, s}(\bm \gamma),
\end{align}
 By Taylor expansion, 
\[
\ell_{n, s}(\bm \gamma')-\ell_{n, s}(\bm \gamma)=(\bm d_s-\mathbb E_{\bm \gamma'}[\bm d_s])^\top(\bm \gamma'-\bm \gamma) + \frac{1}{2}(\bm \gamma'-\bm \gamma)^\top \tilde{\bm \Sigma}_{n, s}(\bm \gamma'-\bm \gamma),
\]
where $\tilde{\bm \Sigma}_{n, s}$ is the covariance matrix defined in \eqref{eq:v_s} with $\bm \beta_s$ replaced by $\tilde{\bm \gamma}=\bm \gamma'+\theta(\bm \gamma'-\bm \gamma)$ for some $0<\theta<1$. Then by arguments as in \eqref{eq:Rgamma} and Lemma \ref{lm:conditiongradient}, Lemma \ref{lm:deltasigma}, \eqref{eq:lgammaexpansionH1} can be written as: 
\begin{align}\label{eq:lgammaH1} 
    \ell_{n, s}(\hat{\bm \beta}_{s})-\ell_{n, s}(\bm \gamma)  = - \frac{1}{2}(\bm d_s-\mathbb E_{\bm \gamma'}[\bm d_s])^\top & \overline{\bm\Gamma}_{n, s} (\bm d_s-\mathbb E_{\bm \gamma'}[\bm d_s]) + (\bm d_s-\mathbb E_{\bm \gamma'}[\bm d_s])^\top(\bm \gamma'-\bm \gamma) \nonumber \\ 
    & + \frac{1}{2}(\bm \gamma'-\bm \gamma)^\top \tilde{\bm \Sigma}_{n, s}(\bm \gamma'-\bm \gamma) + o_P(\sqrt n) , 
\end{align}  
where $\overline{\bm \Gamma}_{n, s}$ is as defined in \eqref{eq:s_n} with the parameter $\bm \beta_s$ replaced by $\bm \gamma'$.

We begin with the case $\|\bm \gamma'-\bm \gamma\|_2 \ll n^{-\frac{2s-3}{4}}$. In this case, since $\grad^2 \ell_{n, s}(\bm \gamma') = \overline{\bm \Sigma}_{n, s}$, by Lemma \ref{lem:lower_bound_hess} 
\begin{equation}
\label{eq:first_approx}
(\bm \gamma'-\bm \gamma)^\top \overline{{\bm \Sigma}}_{n, s} (\bm \gamma'-\bm \gamma)  \asymp  n^{s -1}\|\bm \gamma'-\bm \gamma\|_2^2 \ll \sqrt n .  
\end{equation} 
Similarly, 
\begin{equation}
\label{eq:second_approx}
(\bm \gamma'-\bm \gamma)^\top \tilde {\bm \Sigma}_{n, s} (\bm \gamma'-\bm \gamma)  \asymp  n^{s -1}\|\bm \gamma'-\bm \gamma\|_2^2 \ll \sqrt n .  
\end{equation}  
Hence, 
$$\Var[(\bm d_s-\mathbb E_{\bm \gamma'}[\bm d_s])^\top(\bm \gamma'-\bm \gamma)] = (\bm \gamma'-\bm \gamma)^\top \overline{{\bm \Sigma}}_{n, s}(\bm \gamma'-\bm \gamma) \ll n,$$
which implies, $(\bm d_s-\mathbb E_{\bm \gamma'}[\bm d_s])^\top(\bm \gamma'-\bm \gamma) = o_P(\sqrt n)$, since $\E[(\bm d_s-\mathbb E_{\bm \gamma'}[\bm d_s])^\top(\bm \gamma'-\bm \gamma)] = 0$. 
Therefore, under $H_1$ as in \eqref{eq:H1gamma}, 
\begin{align}
\frac{2\log \Lambda_{n, s}-n}{\sqrt{2n}} & = \frac{2(\ell_{n, s}(\bm \gamma)-\ell_{n, s}(\hat{\bm \beta}_{s})) -n }{\sqrt{2n}} \nonumber \\ 
& = \frac{(\bm d_s-\mathbb E_{\bm \gamma'}[\bm d_s])^\top \overline{\bm \Gamma}_{n, s}(\bm d_s-\mathbb E_{\bm \gamma'}[\bm d_s]) - n }{\sqrt{2 n}} + o_P(1) \tag*{(by \eqref{eq:lgammaH1}, \eqref{eq:first_approx}, and \eqref{eq:second_approx})} \nonumber \\ 
& \dto \cN(0,1) , \nonumber 
\end{align} 
by Proposition \ref{ppn:lambdaH0}. This proves the first assertion in \eqref{eq:asymptoticH1}. 

Next, suppose $\|\bm \gamma'-\bm \gamma\|_2 \gg n^{-\frac{2s-3}{4}}$. In this case, by Lemma \ref{lem:lower_bound_hess}, $(\bm \gamma'-\bm \gamma)^\top \overline{\bm \Sigma}_{n, s} (\bm \gamma'-\bm \gamma)  \asymp  n^{s -1}\|\bm \gamma'-\bm \gamma\|_2^2 \gg \sqrt n$. We will first assume: 
\begin{align}\label{eq:gammaseparationl}
 \sqrt n \ll (\bm \gamma'-\bm \gamma)^\top \overline{\bm \Sigma}_{n, s}(\bm \gamma'-\bm \gamma) \lesssim n . 
 \end{align}
Then we have $\Var[(\bm d_s-\mathbb E_{\bm \gamma'}[\bm d_s])^\top(\bm \gamma'-\bm \gamma)] = (\bm \gamma'-\bm \gamma)^\top \overline{{\bm \Sigma}}_{n, s}(\bm \gamma'-\bm \gamma) = O(n)$. Using this and Proposition \ref{ppn:lambdaH0} it follows that 
$$
\frac{1}{\sqrt{n}}\left[\frac{1}{2}(\bm d_s-\mathbb E_{\bm \gamma'}[\bm d_s])^\top \overline{\bm\Gamma}_{n, s} (\bm d_s-\mathbb E_{\bm \gamma'}[\bm d_s])+(\bm d_s-\mathbb E_{\bm \gamma'}[\bm d_s])^\top(\bm \gamma'-\bm \gamma)\right] 
$$ 
is bounded in probability. Hence, from \eqref{eq:lgammaH1}, 
\begin{align} 
\frac{2\log \Lambda_{n, s}-n}{\sqrt{2n}} & = \frac{(\ell_{n, s}(\bm \gamma)-\ell_{n, s}(\hat{\bm \beta}_{s})) -n }{\sqrt{2n}} \rightarrow \infty , \nonumber 
\end{align} 
in probability, since by Lemma \ref{lem:lower_bound_hess}, $(\bm \gamma'-\bm \gamma)^\top \tilde{\bm \Sigma}_{n, s} (\bm \gamma'-\bm \gamma)  \asymp  n^{s -1}\|\bm \gamma'-\bm \gamma\|_2^2 \gg \sqrt n$. This implies, $\E_{\bm \gamma'}[\phi_{n, s}] \rightarrow 1$, whenever \eqref{eq:gammaseparationl} holds.  Next, we assume 
\begin{align}\label{eq:gammaseparationg}
(\bm \gamma'-\bm \gamma)^\top \overline{\bm \Sigma}_{n, s}(\bm \gamma'-\bm \gamma) \gg n .  
\end{align}
For notational convenience denote $\vartheta_{n, s} :=(\bm \gamma'-\bm \gamma)^\top \overline{\bm \Sigma}_{n, s}(\bm \gamma'-\bm \gamma)$. Then Proposition \ref{ppn:lambdaH0} and \eqref{eq:gammaseparationg} imply that 
\[
\frac{1}{\sqrt{\vartheta_{n, s} }}\left[\frac{1}{2}(\bm d_s-\mathbb E_{\bm \gamma'}[\bm d_s])^\top \overline{\bm \Gamma}_{n, s} (\bm d_s-\mathbb E_{\bm \gamma'}[\bm d_s])+(\bm d_s-\mathbb E_{\bm \gamma'}[\bm d_s])^\top(\bm \gamma'-\bm \gamma)\right] 
\] 
is bounded in probability. 
Using \eqref{eq:first_approx} and \eqref{eq:second_approx} we also get
\[
\frac{(\bm \gamma'-\bm \gamma)^\top \tilde{\bm \Sigma}_{n, s}(\bm \gamma'-\bm \gamma)}{\sqrt{\vartheta_{n, s} }} \asymp  n^{\frac{s-1}{2}} \| \bm \gamma' - \bm \gamma \|_2 \rightarrow \infty.
\]
This implies, from from \eqref{eq:lgammaH1}, 
\[
\E_{ \bm \gamma' }[\phi_{n, s}] =  \mathbb P_{\bm \gamma' } \left(\left|\frac{2\log \Lambda_{n, s} - n}{\sqrt{\vartheta_{n, s} }}\right| \ge z_{\alpha/2}\sqrt{\frac{2n}{\vartheta_{n, s} }}\right) \rightarrow 1. 
\]
This concludes the proof. This completes the proof of the third assertion in \eqref{eq:asymptoticH1}.

Now, we consider the case $n^{\frac{2s-3}{4}}\|\bm \gamma'-\bm \gamma\|_2 \rightarrow \tau \in (0, \infty)$. By Taylor expansion, 
\[
\ell_{n, s}(\bm \gamma')-\ell_{n, s}(\bm \gamma)=(\bm d_s-\mathbb E_{\bm \gamma'}[\bm d_s])^\top(\bm \gamma'-\bm \gamma) + \frac{1}{2}(\bm \gamma'-\bm \gamma)^\top {\bm \Sigma}_{n, s}(\bm \gamma'-\bm \gamma) + \tilde \cT_{n, s},
\] 
where ${\bm \Sigma}_{n, s}$ is as defined in \eqref{eq:v_s} with $\bm \beta_s$ replaced by $\bm \gamma$ and $\tilde \cT_{n, s}$ is as defined in \eqref{eq:gradient3} with the parameter $\tilde{\bm \gamma}=\bm \gamma'+\theta(\bm \gamma'-\bm \gamma)$ for some $0<\theta<1$. By arguments as in Lemma \ref{lm:conditiongradient}, $\tilde \cT_{n, s} = o_P(\sqrt n)$. Then \eqref{eq:Rgamma} and Lemma \ref{lm:conditiongradient}, Lemma \ref{lm:deltasigma}, \eqref{eq:lgammaexpansionH1} can be written as: 
\begin{align}\label{eq:lgammaasymptoticH1} 
    \ell_{n, s}(\hat{\bm \beta}_{s})-\ell_{n, s}(\bm \gamma)  = - \frac{1}{2}(\bm d_s-\mathbb E_{\bm \gamma'}[\bm d_s])^\top & \bm\Gamma_{n, s} (\bm d_s-\mathbb E_{\bm \gamma'}[\bm d_s]) + (\bm d_s-\mathbb E_{\bm \gamma'}[\bm d_s])^\top(\bm \gamma'-\bm \gamma) \nonumber \\ 
    & + \frac{1}{2}(\bm \gamma'-\bm \gamma)^\top \bm \Sigma_{n, s}(\bm \gamma'-\bm \gamma) + o_P(\sqrt n) . 
\end{align}
Note that  $\E[(\bm d_s-\mathbb E_{\bm \gamma'}[\bm d_s])^\top(\bm \gamma'-\bm \gamma)] = 0$ and by Lemma \ref{lem:lower_bound_hess}, 
$$\Var[(\bm d_s-\mathbb E_{\bm \gamma'}[\bm d_s])^\top(\bm \gamma'-\bm \gamma)] = (\bm \gamma'-\bm \gamma)^\top {\bm \Sigma}_{n, s}(\bm \gamma'-\bm \gamma) \asymp_{n, r} \sqrt n , $$ 
when $\|\bm \gamma'-\bm \gamma\|_2 \asymp n^{-\frac{2s-3}{4}}$. Hence, in this case, $(\bm d_s-\mathbb E_{\bm \gamma'}[\bm d_s])^\top(\bm \gamma'-\bm \gamma) = o_P(\sqrt n)$.  This also implies that $$\eta := \lim_{n \rightarrow \infty} \frac{(\bm \gamma'-\bm \gamma)^\top \bm \Sigma_{n, s} (\bm \gamma'-\bm \gamma)}{\sqrt n}$$ exists along a subsequence. (Note that $\Cov_{\bm \gamma}[\bm d_s] = \bm \Sigma_{n, s}$.) 
Hence, from \eqref{eq:lgammaasymptoticH1}, 
\begin{align} 
\frac{2\log \Lambda_{n, s}-n}{\sqrt{2n}} & = \frac{2(\ell_{n, s}(\bm \gamma)-\ell_{n, s}(\hat{\bm \beta}_{s})) -n }{\sqrt{2n}} \nonumber \\ 
& = \frac{(\bm d_s-\mathbb E_{\bm \gamma'}[\bm d_s])^\top \bm \Gamma_{n, s}(\bm d_s-\mathbb E_{\bm \gamma'}[\bm d_s]) -n }{\sqrt{2 n}} - \frac{(\bm \gamma'-\bm \gamma)^\top \bm \Sigma_{n, s}(\bm \gamma'-\bm \gamma)}{\sqrt{2 n}} + o_P(1) \nonumber \\ 
& \dto \cN(-\tfrac{\eta}{\sqrt 2},1) . \nonumber 
\end{align} 
This completes the proof of \eqref{eq:asymptoticgamma}.

\section{Testing Lower Bounds} 
In this section we prove the lower bounds for the goodness-of-fit problem in the $L_2$ and $L_\infty$ norms, that is, Theorem \ref{thm:H0l2} (b) and Theorem \ref{thm:H0lmax} (b), respectively. For this, suppose $\pi_{n}$ be a prior probability distribution on the alternative $H_1$ (as in \eqref{eq:gammal2} or \eqref{eq:H0maximum}). Then the Bayes risk of a test function $\psi_{n}$ is defined as 
\begin{align}\label{eq:Rpsipi}
\cR(\psi_n, \bm \gamma, \pi_{n})= \P_{H_0}( \psi_n=1) + \E_{\bm \gamma' \sim \pi_{n}} \left[\P_{\bm \gamma'}( \psi_n=0) \right].
\end{align} 
For any prior $\pi_{n}$ the worst-case risk of test function $ \psi_n$, as defined in \eqref{eq:Rpsi}, can be bounded below as: 
\begin{lem}\label{lm:lb}
Let $\cH_{n, s}$ denote the collection of $s$-uniform hypergraphs on $n$ vertices. Then 
\begin{align}\label{eq:Rpsi_lb_I}
\cR(\psi_n, \bm \gamma) \geq \cR(\psi_n, \bm \gamma, \pi_{n}) \geq 1- \tfrac{1}{2} \sqrt{\E_{H_0}[L_{\pi_{n}}^2]-1}, 
\end{align} 
where $L_{\pi_{n}}=\frac{\E_{\bm \gamma' \sim \pi_{n}} \left[\P_{\bm \gamma'}(\omega) \right]}{ \P_{H_0}(\omega)},\,\omega\in\cH_{n, s},$ is the $\pi_{n}$-integrated likelihood ratio. 
\end{lem}

\begin{proof} Clearly, $\cR(\psi_n, \bm \gamma) \geq \cR(\psi_n, \bm \gamma,  \pi_{n})$. To show the second inequality in \eqref{eq:Rpsi_lb_I} observe that,
\begin{align}
\cR(\psi_n, \bm \gamma,  \pi_{n}) & \geq \inf_{ \psi_n} \left\{ \P_{H_0}( \psi_n=1) + \E_{\bm \gamma' \sim \pi_{n}} \left(\P_{\bm \gamma'}(\psi_{n} = 0) \right) \right\} \nonumber \\ 
& \geq 1- \sup_{ \psi_n} \left| \P_{H_0}( \psi_n=1) - \E_{\bm \gamma' \sim \pi_{n}} \left(\P_{\bm \gamma'}( \psi_n=1) \right) \right| \nonumber \\ 
& \geq 1- \sup_{\omega \in \cH_{n, s}} \left| \P_{H_0}(\omega) - \E_{\bm \gamma' \sim \pi_{n}} \left[\P_{\bm \gamma'}(\omega) \right] \right| \nonumber \\ 
& \geq 1- \frac{1}{2}\sum_{\omega \in \cH_{n, s}} \left|\frac{\E_{\bm \gamma' \sim \pi_{n}} \left[\P_{\bm \gamma'}(\omega) \right]}{ \P_{H_0}(\omega) } -1 \right|  \P_{H_0}(\omega)  \nonumber \\ 
& = 1-\tfrac{1}{2} \E_{H_0}|L_{\pi_{n}}-1| \nonumber \\
& \geq 1- \tfrac{1}{2} \sqrt{\E_{H_0}[L_{\pi_{n}}^2] - 1}, \nonumber 
\end{align} 
where the last step uses the Cauchy-Schwarz inequality. 
\end{proof}

Therefore, to show all tests are powerless it suffices to construct a prior $\pi_{n}$ on $H_1$ such that $\E_{H_0}[L_{\pi_{n}}^2] \rightarrow 1$. We show this for the $L_2$ norm in Appendix \ref{sec:H0l2pf} and for the $L_\infty$ norm in Appendix \ref{sec:H0lmaxpf}. 

\subsection{Testing Lower Bound in $L_2$ Norm: Proof of Theorem \ref{thm:H0l2} (b)} 
\label{sec:H0l2pf} 

We choose $\bm \gamma = \bm 0$, $\varepsilon \ll n^{-\frac{2s-3}{4}}$, and construct a prior $\pi_n$ on $H_1$ as in \eqref{eq:gammal2} as follows:  Suppose $\bm \gamma' = (\gamma'_1, \gamma'_2, \ldots, \gamma'_n)^\top \in \R^n$ with  
$$\gamma'_{u} = \eta_u \cdot \frac{\varepsilon}{\sqrt{n}},$$ 
for $u \in [n]$, where $\eta_1, \ldots, \eta_n$ are i.i.d Rademacher random variables, taking values $\{ \pm 1 \}$ with probability $\frac{1}{2}$. Clearly, $\| \bm \gamma - \bm \gamma' \|_2 = \varepsilon$. Then, for $H \in \cH_{n, s}$, the $\pi_n$ integrated likelihood ratio is given by 
\[
L_{\pi_n} = \E_{\bm \eta} \left[ \frac{\P_{\bm \gamma'} (H)}{\P_{\bm 0} (H)} \right] = \E_{\bm \eta} \left[ \prod_{\bm e \in {[n] \choose s}} \frac{ 2 e^{w_{\bm \eta}(\bm e) X_{\bm e}}}{  1 + e^{w_{\bm \eta}(\bm e) }  } \right] , 
\] 
where $\revsag{X_{\bm e} := \mathbbm 1\{ \bm e \in E(H) \}}$, $\bm \eta := (\eta_1, \ldots, \eta_n)$, and $w_{\bm \eta}(\bm e) := \frac{\varepsilon}{\sqrt{n}}\sum_{u \in \bm e}\eta_u$, for $\bm e \in {[n] \choose s}$. 
Then 
\[
L_{\pi_n}^2 = \E_{\bm \eta, \bm \eta'} \left[ \prod_{\bm e \in {[n] \choose s}} \frac{ 4 e^{ ( w_{\bm \eta}(\bm e) + w_{\bm \eta'}(\bm e) ) X_{\bm e}}}{ ( 1 + e^{w_{\bm \eta}(\bm e) } ) ( 1 + e^{w_{\bm \eta'}(\bm e) } ) }  \right] , 
\] 
where $\eta_1', \ldots, \eta_n'$ are i.i.d Rademacher random variables which are independent of $\eta_1, \ldots, \eta_n$, $\bm \eta' := (\eta_1', \ldots, \eta_n')$, and $w_{\bm \eta'}(\bm e) := \frac{\varepsilon}{\sqrt{n}}\sum_{u \in \bm e}\eta_u'$, for $\bm e \in {[n] \choose s}$. Taking expectation with respect to $H_0$ gives, 

\begin{align}\label{eq:secondmomentpf1}
    \mathbb{E}_{H_0}[L^2_{\pi_n}] & = \E_{\bm \eta, \bm \eta'} \left[ \prod_{ \bm e \in {[n] \choose s} } \frac{2\, (e^{ ( w_{\bm \eta}(\bm e) + w_{\bm \eta'}(\bm e) ) } + 1)}{ ( 1 + e^{w_{\bm \eta}(\bm e) } ) ( 1 + e^{w_{\bm \eta'}(\bm e) } ) }  \right] \nonumber \\ 
    & = \mathbb E_{\bm \eta, \bm \eta'}\left[\prod_{ \bm e \in {[n] \choose s} }2\,\left\{ \psi(w_{\bm \eta}(\bm e)) \psi(w_{\bm \eta'}(\bm e))+(1- \psi(w_{\bm \eta}(\bm e)))(1- \psi(w_{\bm \eta'}(\bm e)))\right\}\right] , 
\end{align} 
where $\psi(x)$ is the logistic function as defined in Lemma \ref{lm:leave1mlel2}. Using the Taylor expansions of $\psi(x)$ and $1-\psi(x)$ around $0$, we can show that for all $x \in \mathbb{R}$,
\begin{equation*}
\psi(x)  \le \frac{1}{2}+\frac{x}{4}+\frac{x^3}{48} \text{ and } 1-\psi(x) \le \frac{1}{2}-\frac{x}{4}+\frac{x^3}{48}.
\end{equation*} 
As a consequence, for $\bm e \in {[n] \choose s}$, 
\begin{align*}
& 2\,\left\{ \psi(w_{\bm \eta}(\bm e)) \psi(w_{\bm \eta'}(\bm e))+(1- \psi(w_{\bm \eta}(\bm e)))(1- \psi(w_{\bm \eta'}(\bm e)))\right\}\nonumber\\
 & \le 1+\frac{1}{4}w_{\bm \eta}(\bm e)w_{\bm \eta'}(\bm e)+\frac{1}{24}(w_{\bm \eta}(\bm e)^3 +w_{\bm \eta'}(\bm e)^3 )+ \frac{1}{24^2}w_{\bm \eta}(\bm e)^3 w_{\bm \eta'}(\bm e)^3.
\end{align*}  
Using this bound in \eqref{eq:secondmomentpf1} gives, 
\begin{align}\label{eq:secondmomentpf2}
    & \mathbb{E}_{H_0}[L^2_{\pi_n}] \nonumber \\ 
    & \leq \E_{\bm \eta, \bm \eta'} \left[ \prod_{ \bm e \in {[n] \choose s} } \left( 1+\frac{1}{4}w_{\bm \eta}(\bm e)w_{\bm \eta'}(\bm e)+\frac{1}{24}(w_{\bm \eta}(\bm e)^3 +w_{\bm \eta'}(\bm e)^3 )+ \frac{1}{24^2}w_{\bm \eta}(\bm e)^3 w_{\bm \eta'}(\bm e)^3 \right) \right]  \nonumber \\ 
    & \leq \E_{\bm \eta, \bm \eta'} \left[ e^{ \sum_{ \bm e \in {[n] \choose s}} \left\{ \frac{1}{4}w_{\bm \eta}(\bm e)w_{\bm \eta'}(\bm e)+\frac{1}{24}(w_{\bm \eta}(\bm e)^3 +w_{\bm \eta'}(\bm e)^3 )+ \frac{1}{24^2}w_{\bm \eta}(\bm e)^3 w_{\bm \eta'}(\bm e)^3 \right\} } \right] , 
\end{align} 
since $1+x \leq e^x$.  

Recalling the definition of $w_{\bm \eta}(\bm e)$ observe that 
\begin{align*}
\left|\sum_{\bm e \in {[n] \choose s}}w_{\bm \eta}(\bm e)^3 \right|  \leq \frac{\varepsilon^{3}}{n^{\frac{3}{2}}}\sum_{\bm e \in {[n] \choose s}}\left(\sum_{u \in \bm e} |\eta_u| \right)^3  \leq  \varepsilon^{3} n^{s- \frac{3}{2}} . 
\end{align*}  
Hence, 
\begin{equation}\label{eq:T1}
\mathbb{E}\left[ e^{ 2\sum_{\bm e \in {[n] \choose s}}w_{\bm \eta}(\bm e)^3} \right] \leq e^{2 \varepsilon^{3} n^{s- \frac{3}{2}}} \rightarrow 1, 
\end{equation}
since $\varepsilon \ll n^{-\frac{2s-3}{4}}$ and, for $s \ge 2$, $-\frac{s}{2}+\frac{3}{4} > 0$.
Similarly, it can be shown that 
\begin{equation}\label{eq:T2}
\lim_{n \rightarrow \infty}\mathbb{E}\left[ e^{ 2\sum_{\bm e \in {[n] \choose s}}w_{\bm \eta}(\bm e)^3w_{\bm \eta'}(\bm e)^3 } \right] = 1.
\end{equation} 
Then H\"older's inequality applied to \eqref{eq:secondmomentpf2} followed by  \eqref{eq:T1} and \eqref{eq:T2} gives 
\begin{align}\label{eq:eta}
\mathbb E_{H_0}[L^2_\pi] & \le \left\{\mathbb E_{\bm \eta,\bm \eta'}\left[ e^{\frac{3}{4}\sum_{e \in {[n] \choose s}}w_{\bm \eta}(\bm e)w_{\bm \eta'}(\bm e) } \right]\right\}^{1/3}(1+o(1)) . 
\end{align} 
Next, observe that
\begin{align}\label{eq:wn}
 \sum_{\bm e \in {[n] \choose s}}w_{\bm \eta}(\bm e)w_{\bm \eta'}(\bm e) & = \frac{\varepsilon^2}{n}\left\{\sum_{\bm e \in {[n] \choose s}}\left(\sum_{u \in \bm e}\eta_u\right)\left(\sum_{v \in \bm e}\eta'_v\right)\right\}\nonumber\\
& = \frac{\varepsilon^2}{n}\left\{{n-1 \choose s-1}\sum_{u=1}^{n}\eta_u\eta'_u + {n-2 \choose s-2}\sum_{1 \leq u \neq v \leq n} \eta_u\eta'_v\right\}\nonumber\\
& \leq \varepsilon^2 n^{s-2} \sum_{u=1}^{n}\eta_u\eta'_u +  \varepsilon^2 n^{s-3} \sum_{1 \leq u \neq v \leq n} \eta_u\eta'_v  \nonumber \\ 
& = \varepsilon^2 n^{s-2} \sum_{u=1}^{n}\eta_u\eta'_u +    \varepsilon^2 n^{s-3} \left\{ \left(\sum_{u=1}^{n}\eta_u\right)\left(\sum_{v=1}^{n}\eta'_v\right) - \sum_{u=1}^{n}\eta_u\eta'_u \right\} . 
\end{align}  
Note that $ \varepsilon^2 n^{s-3} \left| \sum_{u=1}^{n}\eta_u\eta'_u \right| \leq \varepsilon^2 n^{s-2}$. Hence, 
\begin{align}\label{eq:wnexpansion}
\E\left[ e^{\frac{9}{4} \varepsilon^2 n^{s-3} \sum_{u=1}^{n}\eta_u\eta'_u  } \right] \lesssim e^{\varepsilon^2 n^{s-2} } \rightarrow 1, 
\end{align}
since $\varepsilon \ll n^{-\frac{2s-3}{4}}$. 
From \eqref{eq:eta}, by H\"older's inequality followed by  \eqref{eq:wn} and \eqref{eq:wnexpansion} gives 
\begin{align}\label{eq:Lpi}
\mathbb E_{H_0}[L^2_\pi] 
& \lesssim_{s} \left\{ \mathbb E_{\bm \eta,\bm \eta'}\left[e^{ \frac{9}{4} \varepsilon^2 n^{s-2} \sum_{u=1}^{n}\eta_u\eta'_u }  \right] \right\}^{1/9}\left\{ 
\mathbb E_{\bm \eta, \bm \eta'}\left[ e^{ \frac{9}{4} \varepsilon^2 n^{s-3}  \left(\sum_{u=1}^{n}\eta_u\right) \left(\sum_{v=1}^{n}\eta'_v\right) } \right] \right\}^{1/9}(1+o(1)) . 
\end{align} 
Denote $X_n:= \sum_{u=1}^{n}\eta_u$ and $Y_n:=\sum_{v=1}^{n}\eta'_u$. Since $X_n$ and $Y_n$ are independent and each of them is a sum of i.i.d. Rademacher random variables, 
\begin{align*} 
\mathbb E_{\bm \eta,\bm \eta'}\left[ e^{ \frac{9}{4} \varepsilon^2 n^{s-3} X_nY_n} \right] = \mathbb \E\left [ \mathbb E \left[ e^{ \frac{9}{4} \varepsilon^2 n^{s-3} X_nY_n} | Y_n \right] \right] 
& = \mathbb E\left[ \left( \cosh \left( \frac{9}{4} \varepsilon^2 n^{s-3} Y_n\right) \right)^n \right] \nonumber \\ 
& \leq \mathbb E\left[  e^{\frac{81}{16} \varepsilon^4 n^{2s-5} Y_n^2 }  \right] , 
\end{align*}
where last step uses $\cosh(x) \le e^{x^2}$, for all $x \in \mathbb{R}$. 
Since $|Y_n| \le n $, this implies, 
\begin{align}\label{eq:XY}
\mathbb E_{\bm \eta,\bm \eta'}\left[ e^{ \frac{9}{4} \varepsilon^2 n^{s-3} X_nY_n} \right] \leq e^{\frac{81}{16} \varepsilon^4 n^{2s-5} Y_n^2 }  \leq e^{\frac{81}{16} \varepsilon^4 n^{2s-3}} \rightarrow 1, 
\end{align} 
since $\varepsilon \ll n^{-\frac{2s-3}{4}}$. 
Next, observe that $\eta_u\eta'_u$, for $u=1,\cdots,n$, are i.i.d. Rademacher random variables. Again using $\cosh(x) \le e^{x^2}$ for all $x \in \mathbb{R}$, we can show that
\begin{align}\label{eq:ueta}
\mathbb E_{\bm \eta,\bm \eta'}\left[ e^{ \frac{9}{4} \varepsilon^2 n^{s-2} \sum_{u=1}^{n}\eta_u\eta'_u } \right] & = \left( \cosh \left( \frac{9}{4} \varepsilon^2 n^{s-2}  \right) \right) ^n \leq e^{\frac{81}{16} \varepsilon^4 n^{2s-3}} \rightarrow 1 , 
\end{align} 
since $\varepsilon \ll n^{-\frac{2s-3}{4}}$. 
Hence, using \eqref{eq:XY} and \eqref{eq:ueta} in \eqref{eq:Lpi} gives, 
$$\lim_{n \rightarrow \infty}E_{H_0}[L^2_\pi] =1 . $$ 
By Lemma \ref{lm:lb}, this completes the proof of Theorem \ref{thm:H0l2} (b).

\subsection{Testing Lower Bound in $L_\infty$ Norm: Proof of Theorem \ref{thm:H0lmax} (b)}  
\label{sec:H0lmaxpf}

\revmark{
We choose $\bm \gamma = \bm 0$, $\varepsilon \ll (\log n/n^{s-1})^{1/2}$, and construct a prior $\pi_n$ on $H_1$ as in \eqref{eq:gammal2} as follows:  Suppose $\bm \gamma_u \in \R^n$ with  
$$\bm \gamma_{u} = \varepsilon \bm e_u,$$ 
for $u \in [n]$, where $\bm e_1, \ldots, \bm e_n$ being the canonical basis vectors in $\R^n$. Then $\pi_n$ assigns probability $1/n$ to each $\bm \gamma_{u}$. Clearly, $\| \bm \gamma - \bm \gamma_u \|_\infty = \varepsilon$ for all $u \in [n]$. Then, for $H \in \cH_{n, s}$, the $\pi_n$ integrated likelihood ratio is given by 
\[
L_{\pi_n} = \frac{1}{n}\sum_{u \in [n]}\prod_{\bm e \in {[n] \choose s} : u \in \bm e} \frac{ 2 e^{ \varepsilon X_{\bm e}}}{  1 + e^{  \varepsilon  }  }  , 
\] 
where $X_{\bm e} := \bm 1\{ \bm e \in E(H) \}$.  
Then 
\[
L_{\pi_n}^2 = \frac{1}{n^2}\sum_{u \in [n]}\prod_{\bm e \in {[n] \choose s} : u \in \bm e}  \frac{ 4 e^{ 2 \varepsilon  X_{\bm e}}}{  ( 1 + e^{  \varepsilon   } )^2 }+ \frac{1}{n^2}\sum_{u \ne v \in [n]}\prod_{\bm e \in {[n] \choose s} : u, v \in \bm e}  \frac{ 4 e^{ 2 \varepsilon  X_{\bm e}}}{  ( 1 + e^{  \varepsilon   } )^2 } .
\] 
Observe that 
\begin{align}\label{eq:Lpimaximum}
\mathbb E_{H_0} \left[ \prod_{\bm e \in {[n] \choose s} : u \in \bm e} \frac{ 4 e^{ 2 \varepsilon  X_{\bm e}}}{  ( 1 + e^{  \varepsilon   } )^2 } \right] = \left ( 2\, \psi(\varepsilon)^2 + 2\,(1- \psi(\varepsilon))^2\right )^{ {n \choose s-1} } , 
\end{align}
where $\psi(x)= \frac{e^x}{1+e^x}$. Similarly,
\begin{align}\label{eq:Lpimaximum}
\mathbb E_{H_0} \left[ \prod_{\bm e \in {[n] \choose s} : u ,v \in \bm e} \frac{ 4 e^{ 2 \varepsilon  X_{\bm e}}}{  ( 1 + e^{  \varepsilon   } )^2 } \right] = \left ( 2\, \psi(\varepsilon)^2 + 2\,(1- \psi(\varepsilon))^2\right )^{ {n \choose s-2} } . 
\end{align}
Since $\varepsilon \ll (\log n/n^{s-1})^{1/2}$, a Taylor expansion around zero gives $\psi(\varepsilon) = \frac{1}{2}+\frac{1}{4} \varepsilon + O(\varepsilon^2)$. Hence, 
\[
2 \psi(\varepsilon)^2 + 2 (1-\psi(\varepsilon))^2  = 1 + O(\varepsilon^2) . 
\]
Therefore, by \eqref{eq:Lpimaximum} and using $1+x \leq e^x$ gives, 
\begin{align*}
    \mathbb{E}_{H_0}[L_{\pi_n}^2] 
    & \le \frac{1}{n}e^{O(\varepsilon^2 n^{s-1})}+e^{O(\varepsilon^2 n^{s-2})} \rightarrow 1,
\end{align*}
since $\varepsilon \ll (\log n/n^{s-1})^{1/2}$. 
By Lemma \ref{lm:lb}, this completes the proof of Theorem \ref{thm:H0lmax} (b). }


\section{Proof of Proposition \ref{ppn:conv_hull}}
\label{sec:lemmapf}

Define $g = (g_1, g_2, \ldots, g_n): \mathbb{R}^n\rightarrow\mathbb{R}^n$ where $g_u: \mathbb{R}^n\rightarrow\mathbb{R}$, 
for $u \in [n]$, as follows: 
\[
g_u(\bm x) = \sum_{\bm e \in {[n] \choose s} : u \in \bm e }\frac{e^{\bm x^\top_e\bm 1}}{1+e^{\bm x^\top_e\bm 1}} ,
\] 
where $\bm x = (x_1, x_2, \ldots, x_n)^\top$ and $\bm x_{\bm e} = (x_{u_1}, x_{u_2}, \ldots, x_{u_s})^\top$ for $\bm e = (u_1, u_2, \ldots, u_s)$. 
Observe that $\mathcal{R}_s$ is the range of $g$. Since the expected degree of a vertex is a weighted combination of all the possible degrees in $s$-uniform hypergraphs on $n$ vertices, this implies $\bar{\mathcal{R}}_s \subseteq \mathrm{conv}\,(\mathcal{D}_s)$.

To show the other side, for every $\bm y \in \mathbb{R}^n$ we define,
\[
f_{\bm y}(\bm x)=\sum_{i=1}^{n}x_iy_i-\sum_{ \{v_1, v_2, \ldots, v_s \} \in {[n] \choose s}}\log(1+e^{x_{v_1}+\ldots+x_{v_s}}).
\]
Since the probability of observing an $s$-uniform hypergraph with parameter $\bm x$ and $s$-degree sequence $\bm d_s=(d_s(1), \ldots, d_s(n))$ is 
\[
\frac{e^{\sum_{v=1}^n d_s(v) x_{v}}}{\prod_{ \{v_1, v_2, \ldots, v_s \} \in {[n] \choose s}}(1+e^{x_{v_1}+\ldots+x_{v_s}})}.
\]
and is less than $1$, taking logarithm on both sides we get $f_{\bm d_s}(\bm x) \le 0$. Further as $f_{\bm y}(\bm x)$ depends linearly on $\bm y$, we have $
f_{\bm y}(\bm x) \le 0$ for all $\bm y \in \mathrm{conv}\,(\mathcal{D}_s)$ and $\bm x \in \mathbb{R}^n$. Now, let us fix $\bm y \in \mathrm{conv}\,(\mathcal{D}_s)$. It can be shown that the Hessian $\nabla^2f_{\bm y}(\bm x)$ is uniformly bounded, hence, by \cite[Lemma 3.1]{chatterjee2011random} there exists a sequence $\{\bm x_k\}_{k \geq 1}$ such that $\nabla f_{\bm y}(\bm x_k) \rightarrow 0$. Observing that $\nabla f_{\bm y}(\bm x_k)=\bm y-g(\bm x)$, we get $g(\bm x_k) \rightarrow \bm y$. As $\bm y\in \mathrm{conv}\,(\mathcal{D}_s)$ is arbitrary, this implies $\mathrm{conv}\,(\mathcal{D}_s)\subseteq \bar{\mathcal{R}}_s$. \hfill $\Box$


\end{document}